\documentclass[a4paper,11pt]{amsart}

\usepackage[left=2.7cm,right=2.7cm,top=3.5cm,bottom=3cm]{geometry}
\usepackage{amsthm,amssymb,amsmath,amsfonts,mathrsfs,amscd,dsfont,verbatim}
\usepackage[utf8]{inputenc}
\usepackage[T1]{fontenc}
\usepackage[all,cmtip]{xy}
\usepackage{latexsym}
\usepackage{longtable}
\usepackage{mathtools}
\usepackage[pagebackref]{hyperref}
\usepackage{graphicx}
\usepackage{tikz-cd}

\usepackage{graphicx}
\newcommand{\Bmu}{\mbox{$\raisebox{-0.59ex}
  {$l$}\hspace{-0.18em}\mu\hspace{-0.88em}\raisebox{-0.98ex}{\scalebox{2}
  {$\color{white}.$}}\hspace{-0.416em}\raisebox{+0.88ex}
  {$\color{white}.$}\hspace{0.46em}$}{}}

\numberwithin{equation}{section}
\mathtoolsset{showonlyrefs}

\newfont{\cyr}{wncyr10 scaled 1100}
\newfont{\cyrr}{wncyr9 scaled 1000}

\theoremstyle{plain}
\newtheorem{theorem}{Theorem}[section]
\newtheorem{proposition}[theorem]{Proposition}
\newtheorem{lemma}[theorem]{Lemma}
\newtheorem{corollary}[theorem]{Corollary}

\newtheorem*{ThmA}{Theorem A}

\theoremstyle{definition}
\newtheorem{definition}[theorem]{Definition}
\newtheorem{assumption}[theorem]{Assumption}

\theoremstyle{remark}
\newtheorem{remark}[theorem]{Remark}

\newcommand{\Q}{\mathds Q}
\newcommand{\N}{\mathds N}
\newcommand{\Z}{\mathds Z}
\newcommand{\R}{\mathds R}
\newcommand{\A}{\mathds A}
\newcommand{\C}{\mathds C}

\newcommand{\D}{\mathbf{D}}

\newcommand{\defeq}{\vcentcolon=}

\DeclareMathOperator{\Spec}{Spec}
\DeclareMathOperator{\Pic}{Pic}
\DeclareMathOperator{\End}{End}
\DeclareMathOperator{\Aut}{Aut}
\DeclareMathOperator{\Frob}{Frob}

\DeclareMathOperator{\Norm}{N}
\DeclareMathOperator{\Hom}{Hom}

\DeclareMathOperator{\Gal}{Gal}
\DeclareMathOperator{\GL}{GL}

\DeclareMathOperator{\CH}{CH}

\DeclareMathOperator{\Ind}{Ind}

\DeclareMathOperator{\Ta}{Ta}

\DeclareMathOperator{\Sym}{Sym}

\DeclareMathOperator{\Res}{Res}

\newcommand{\cyc}{{\mathrm{cyc}}}

\newcommand{\dR}{\mathrm{dR}}
\newcommand{\et}{\text{\'et}}

\newcommand{\cO}{\mathcal{O}}
\newcommand{\p}{\mathfrak{p}}
\newcommand{\bomega}{\boldsymbol{\omega}}
\newcommand{\bz}{\boldsymbol{z}}

\newcommand{\longmono}{\mbox{\;$\lhook\joinrel\longrightarrow$\;}}

\newcommand{\longepi}{\mbox{\;$\relbar\joinrel\twoheadrightarrow$\;}}

\newcommand{\smallmat}[4]{\bigl(\begin{smallmatrix}#1&#2\\#3&#4\end{smallmatrix}\bigr)}

\makeatletter
\@namedef{subjclassname@2020}{%
  \textup{2020} Mathematics Subject Classification}
\makeatother

\include{thebibliography}

\title[A generalized Rubin formula for Hecke characters]{A generalized Rubin formula for Hecke characters}
\author{Matteo Longo, Stefano Vigni and Shilun Wang}
\date{}

\begin{document}

\thanks{The first two authors are partially supported by PRIN 2022 ``The arithmetic of motives and $L$-functions'' and by the GNSAGA group of INdAM. The research by the second author is partially supported by the MUR Excellence Department Project awarded to Dipartimento di Matematica, Universit\`a di Genova, CUP D33C23001110001. The third author is supported by the China Scholarship Council.}

\begin{abstract}
The goal of this paper is to generalize Rubin's theorem on values of Katz's $p$-adic $L$-function outside the range of interpolation from the case of Hecke characters of CM elliptic curves to more general self-dual algebraic Hecke characters. We follow the approach by Bertolini--Darmon--Prasanna, based on generalized Heegner cycles, which we extend from characters of imaginary quadratic fields of infinity type $(1,0)$ to characters of infinity type $(1+\ell,-\ell)$ for an integer $\ell\geq0$. 
\end{abstract}

\address{Dipartimento di Matematica, Universit\`a di Padova, Via Trieste 63, 35121 Padova, Italy}
\email{mlongo@math.unipd.it}
\address{Dipartimento di Matematica, Universit\`a di Genova, Via Dodecaneso 35, 16146 Genova, Italy}
\email{stefano.vigni@unige.it}
\address{Dipartimento di Matematica, Universit\`a di Padova, Via Trieste 63, 35121 Padova, Italy}
\email{shilun.wang@phd.unipd.it}



\maketitle

\setcounter{tocdepth}{1}

%\tableofcontents

\section{Introduction} 

The goal of this paper is to extend Rubin's theorem on values of Katz's $p$-adic $L$-function outside the range of classical interpolation (\cite{Ru92}) from the case of characters of CM elliptic curves to more general self-dual Hecke characters. 

To put our results into context and motivate them, we begin by recalling Rubin's theorem. Let $\nu_A$ be an algebraic Hecke character of infinity type $(1,0)$ attached to an elliptic curve $A$ over $\Q$ with complex multiplication by the ring of integers of an imaginary quadratic field $K$; assume that the complex $L$-function $L(A,s)$ vanishes to order $1$ at the central point $s=1$. Rubin's result (\emph{cf.} \cite[Corollary 10.3]{Ru92}) asserts that there exists a global point $P\in A(\Q)$ of infinite order such that 
\begin{equation} \label{rubin-eq}
\mathscr{L}_p(\nu_A^\ast)\equiv\Omega_p(A)^{-1}\cdot\log_{\omega_A}(P)^2 \pmod{K^\times}, 
\end{equation}
where 
\begin{itemize}
\item $\mathscr{L}_p(\nu_A^*)$ is Katz's $p$-adic $L$-function evaluated at the complex conjugate character $\nu_A^*$ of $\nu_A$;  
\item $\Omega_p(A)$ is the $p$-adic period attached to $A$;
\item $\omega_A\in \Omega^1(A/\Q)$ is a regular differential on $A$ over $\Q$;
\item $\log_{\omega_A}\colon A(\Q_p)\rightarrow\Q_p$ is the $p$-adic formal logarithm on $A$ with respect to $\omega_A$. 
\end{itemize}
It is well known that Katz's $p$-adic $L$-function interpolates values of the complex $L$-function of suitable Hecke characters (see, \emph{e.g.}, \S \ref{katz-subsec}): it turns out that $\nu_A$ lies inside the range of interpolation of $\mathscr L_p$, whereas $\nu_A^*$ does not. The original approach by Rubin to congruence \eqref{rubin-eq} combines elliptic units and the theory of Heegner points (see, \emph{e.g.}, the introduction to Rubin's work in \cite[\S1]{BDP2}). Rubin's results were later extended to powers of $\nu_A$ by Agboola in an unpublished paper (\cite{Agboola}), following a strategy similar to that of Rubin. 

More recently, Bertolini, Darmon and Prasanna proposed in \cite{BDP2} a different approach to Rubin's results, based on the arithmetic of \emph{generalized Heegner cycles}: they exploited, on the one hand, a factorization of what is nowadays commonly known as the \emph{BDP $p$-adic $L$-function} into a product of two Katz $p$-adic $L$-functions and, on the other hand, a $p$-adic Gross--Zagier formula relating their $p$-adic $L$-function to generalized Heegner cycles. Both generalized Heegner cycles and BDP $p$-adic $L$-functions were introduced in \cite{BDP1}, where the aforementioned $p$-adic Gross--Zagier formula was also proved. The approach in \cite{BDP2} enables the authors to go beyond the scope of Rubin's work and extend it to a wider class of Hecke characters of imaginary quadratic fields of infinity type $(1,0)$. More precisely, let $\nu$ be an algebraic Hecke character of infinity type $(1,0)$ of an imaginary quadratic field $K$ of odd discriminant $-D_K$ satisfying ${\nu|} _{\A^\times_\Q}=\varepsilon_K \cdot \Norm_\Q$, where $\varepsilon_K$ is the quadratic Dirichlet character attached to $K$ and $\Norm _\Q\colon \A_\Q^\times \rightarrow\R^\times$ is the adelic norm character. Assume that the root number of $\nu$ (\emph{i.e.}, the sign of the functional equation of the complex $L$-function $L(\nu,s)$) is $-1$ and the conductor of $\nu$ is exactly divisible by the different ideal $\mathfrak{d}_K$ of $K$. Let $E_\nu$ be the subfield of $\C$ generated by the values of the Hecke character $\nu$. Attached to $\nu$ there is an abelian variety $B_\nu$ defined over $K$ with the property that the $p$-adic Galois representation of $\nu$ is equal to the Galois representation on the first \'etale cohomology group of $B_\nu$ with coefficients in $\Q_p$. Under the previous conditions on $\nu$, by \cite[Theorem 2]{BDP2} there exists $P_\nu \in B_\nu(K) \otimes_\Z \Q$ such that
\begin{equation} \label{BDPTHM}
\mathscr{L}_p(\nu^\ast)\equiv\Omega_p(\nu^\ast)^{-1}\cdot\log_{\omega_\nu}(P_\nu)^2 \pmod{E_\nu^\times},
\end{equation}
where $\Omega_p(\nu^\ast)\in\C_p^\times$ is the $p$-adic period attached to the complex conjugate $\nu^*$ of $\nu$ and $\log_{\omega_\nu}$ is the $p$-adic formal logarithm on $B_\nu$ with respect to a suitable invariant differential $\omega_\nu$ on $B_\nu$. Moreover, $P_\nu\neq 0$ if and only if $L'(\nu,1)\neq 0$. 
%\begin{itemize}
%    \item $\Omega_p(\nu^\ast) \in \C_p$ is the $p$-adic period attached to $\nu$;
%    \item $\omega_\nu$ is a nonzero element of $\Omega^1(B_\nu/E_\nu)$;
%    \item  We have a pairing
%    \begin{equation}
%        \langle \cdot , \cdot \rangle \colon \Omega^1(B_\nu/K) \times B_\nu(K) \longrightarrow \Q_p,
%    \end{equation}
 %   satisfying $\langle [\lambda]^\ast\omega, P \rangle=\langle \omega, [\lambda]P \rangle$. Then we can extend this pairing to an $E_\nu \otimes \Q_p$-valued pairing between $\Omega^1(B_\nu/E_\nu)$ and $B_\nu(K) \otimes E_\nu$. Hence we define $\log_{\omega_\nu} \colon= \langle \omega_\nu, \cdot \rangle$.
%\end{itemize}
%
%\begin{remark}
%As pointed out in \cite{BDP2}, the assumption (2) is needed since the $p$-adic Gross--Zagier formula of \cite{BDP1} (cf. \cite[Theorem 5.13]{BDP2} ) is only proved for imaginary quadratic fields of odd discriminant. Also we will explain more in Section 4.1 about this assumption. We can remove assumption (3) if we generalize the result in \cite{BDP1} to the case of Shimura curves over $\Q$. However, we have the recent work \cite{Br15} and \cite{}, it is also possible to remove the assumption (3). 
%\end{remark}

In the present paper, we extend the approach of Bertolini--Darmon--Prasanna to Rubin's formula to more general algebraic Hecke characters of imaginary quadratic fields of infinity type $(1+\ell,-\ell)$ for an integer $\ell\geq 0$. A partial generalization of some of the results in \cite{BDP2} to these characters has already been developed (under a slightly different perspective) in  \cite{BDP5}, \cite{BDP3}, \cite{BDP4}: here we complement these works by proving results on values of Katz's $p$-adic $L$-functions that are similar in spirit to those of Bertolini--Darmon--Prasanna (see, \emph{e.g.}, Theorem \ref{thm5.11}) but (to the best of our knowledge) have never been investigated. 

In the following list of assumptions, $\chi$ is an algebraic Hecke character of an imaginary quadratic field $K$ of odd discriminant $-D_K$ having infinity type $(1+\ell,-\ell)$ for some $\ell\in\N$. 

\begin{assumption} \label{assumptionchi}
\begin{enumerate} 
\item $\chi$ is critical self-dual (see Definition \ref{defSD} for the terminology);
\item the central character of $\chi$ is $\varepsilon_K$ (the notion of central character is reviewed in \S\ref{Heckechar}); 
\item the conductor of $\chi$ is exactly divisible by the different ideal $\mathfrak{d}_K$ of $K$; 
\item the sign of the functional equation of the $L$-function $L(\chi,s)$ of $\chi$ is $-1$. 
\end{enumerate}
\end{assumption}

Fix a prime number $p$ that divides neither $D_K$ nor the conductor of $\chi$ and splits in $K$. Let $E_\chi$ be the coefficient field of $\chi$ and write $\mathcal{E}_\chi$ for the completion of $E_\chi$ under a fixed embedding $\bar\Q\hookrightarrow\bar\Q_p$ of algebraic closures (\emph{cf.} \S \ref{notation-subsec}). Denote by $V_\chi$ the $\mathcal{E}_\chi$-component of the \'etale realization of the motive of $\chi$, whose construction is reviewed in detail in Section \ref{secmotives}. One can consider the Bloch--Kato Selmer group $H^1_f\bigl(K,V_{\chi^\ast}(1)\bigr)=H^1_f(K,V_{\chi^{-1}})$; there are a localization map $H^1_f\bigl(K,V_{\chi^\ast}(1)\bigr)\rightarrow H^1_f\bigl(\Q_p,V_{\chi^*}(1)\bigr)$ and the so-called Bloch--Kato logarithm $\log_\chi:H^1_f\bigl(\Q_p,V_{\chi^*}(1)\bigr)\rightarrow\mathbf{D}_{\dR}(V_\chi)^\vee$, where $\mathbf{D}_{\dR}(V_\chi)$ is the de Rham module of $V_\chi$ and $(\cdot)^\vee$ indicates the dual. We still denote by $\log_\chi:H^1_f\bigl(K,V_{\chi^*}(1)\bigr)\rightarrow\mathbf{D}_{\dR}(V_\chi)^\vee$ the composition of the Bloch--Kato logarithm and the localization map. 

Keep Assumption \ref{assumptionchi} in force. Our main result, which corresponds to Theorem \ref{mainthm}, can be stated as follows.

\begin{ThmA} \label{introteo}
Assume $\mathscr{L}_{p,\mathfrak{c}}(\chi^*)\neq 0$. The $\mathcal{E}_\chi$-vector space $H^1_f\bigl(K,V_{\chi^*}(1)\bigr)$ has dimension $1$. Moreover, there is $\bomega_\chi\in \mathbf{D}_{\dR}(V_\chi)$ such that for each basis element $\bz_\chi$ of $H^1_f\bigl(K,V_{\chi^*}(1)\bigr)$ the congruence 
\[ \mathscr{L}_{p,\mathfrak{c}}(\chi^*)\equiv{\Omega_p(\chi^*)}^{-1}\cdot\log_{\chi}(\bz_\chi)(\bomega_{\chi})^2\pmod{\mathcal{E}_{\chi}^\times} \]
holds true, where $\Omega_p(\chi^\ast)\in\C_p^\times$ is a $p$-adic period canonically attached to $\chi^\ast$.  
\end{ThmA}

It is worth emphasizing that, as in the case of Rubin's theorem and of its generalization by Agboola, the character $\chi^*$ lies \emph{outside} the range of classical interpolation for Katz's $p$-adic $L$-function; the value $\mathscr{L}_{p,\mathfrak{c}}(\chi^*)$
should be regarded as a $p$-adic avatar of the value at $s=1$ of the first derivative of the complex $L$-function $L(\chi,s)$. 

Now we shall describe the content of this paper and our general strategy. As suggested before, our strategy of the proof is inspired by \cite{BDP2} and makes substantial use of results and techniques from \cite{BDP1} and \cite{CH}. The ingredients taken from \cite{BDP1} are generalized Heegner cycles (reviewed in \S\ref{secGHC}), BDP $p$-adic $L$-functions (reviewed in Section \ref{BDPsec}) and a $p$-adic Gross--Zagier formula (reviewed in \S\ref{secGZ}) that explains the relation between generalized Heegner cycles and $p$-adic $L$-functions. A connection with Katz's $p$-adic $L$-function is made possible, as in \cite{BDP2}, by factorization formulas of $L$-functions (see, especially, equality \eqref{cplxfact}) involving theta series and inducing a factorization of BDP $p$-adic $L$-functions into a product of two Katz $p$-adic $L$-functions (this is explained in \S\ref{secfac}). To use these ingredients in the more general setting of Hecke characters of infinity type $(1+\ell,-\ell)$, we are led to compare different constructions of motives of Hecke characters coming from different factorizations of $\chi$. More precisely, generalized Heegner cycles arises from cycles in a generalized Kuga--Sato variety of the form $X_r\defeq W_r\times A^r$, where $r=2\ell$, $A$ is a fixed elliptic curve defined over the Hilbert class field $H$ of $K$ having CM by the ring of integers $\mathcal{O}_K$ of $K$ and $W_r$ is the Kuga--Sato variety over a suitable modular curve of dimension $r+1$ (see \S\ref{secBDPmot}). To be able to construct classes in $H^1(K,V_\chi)$ out of generalized Heegner cycles, we are led to connect the motive of $\chi$ to $X_r$. More precisely, let $M_{E_\chi}(\chi)$ be the motive of $\chi$ with coefficients in $E_\chi$. The construction of this motive, which we review in \S\ref{motives}, builds on results from \cite{DMOS82} and \cite{S86}, where $M_{E_\chi}(\chi)$ is constructed by means of projectors from simpler motives: Artin motives (\S\ref{Artinmotives}), Tate motives (\S\ref{Tatemotives}) and motives of abelian varieties (\S\ref{abmotives}). However, to connect $M_{E_\chi}(\chi)$ with the motive $h^1(X_r)$ of $X_r$ we take a different viewpoint and consider factorizations of $\chi$ of the form 
\begin{equation} \label{introfac}
\chi=\varphi \psi(\chi_A\chi_A^\ast)^\ell,
\end{equation}
where $\chi_A$ is the Hecke character (of weight $(-1,0)$) attached to $A$, $\chi_A^\ast$ is its conjugate, $\varphi$ is a finite order character and $\psi$ is a Hecke character of infinity type $(k-1,0)$ where $k=r+2$ (a similar factorization has been considered in \cite{BDP3}, \cite{BDP4} to address different questions). The motive of $\chi$ can then be decomposed (up to a Tate twist) as the tensor product of the following motives:
\begin{itemize}
\item the Deligne--Scholl motive of the theta series $\theta_\psi$ attached to $\psi$ (\S\ref{secscholl}), which is obtained by taking quotients of the motive of $W_r\otimes E_\psi$ by the action of Hecke operators, where $E_\psi$ is the Hecke field of $\theta_\psi$;  
\item the motive attached to the character $(\chi_A\chi_A^*)^\ell$, which can be realized (again, up to Tate twists) in the \'etale realization of the motive $h^1\bigl(B^\ell\times (B^\ast)^\ell\bigr)$, where $B\defeq\Res_{H/K}(A)$, $B^\ast\defeq\Res_{H/K}(A^\ast)$ and $A^*$ is $A$ as an $H$-variety with the action of $\cO_K$ twisted by the generator of $\Gal(K/\Q)$.
\end{itemize}
Up to Tate twists, one may then realize the motive of $\chi$ as a submotive of the motive $h^1(Z_r)$ of the variety $Z_r\defeq W_r\times B^\ell\times (B^\ast)^\ell$. Observe that $Z_r$ is a $K$-variety whose base change to $H$ can be explicitly related to $X_r$ (this is done in Lemma \ref{lemma1}). Comparing the \'etale realizations of $X_r$ and $Z_r$, we are able (see, especially, Corollary \ref{coroz}) to produce classes $z_{\psi,\vartheta}\in H^1_f\bigl(K,V_{\chi^\ast}(1)\otimes_{\mathcal{E}_\chi} \mathcal{E}_{\psi,\vartheta}\bigr)$ using generalized Heegner cycles, where $\vartheta\defeq(\chi_A\chi_A^\ast)\Norm_K^{-(\ell-1)}$ and $\mathcal{E}_{\psi,\vartheta}$ is the completion of the image of the composite of the coefficient fields of $\psi$ and $\vartheta$ under the embedding $\bar\Q\hookrightarrow\bar\Q_p$. To define these classes, we combine results from \cite{CH} with several results on the relation between \'etale Abel--Jacobi maps for $X_r$ and $Z_r$, discussed in \S\ref{sec-etaleAJ}, \S\ref{AJMAPS} and \S\ref{secAJmap}. Moreover, comparing the de Rham realizations of $X_r$ and $Z_r$ in \S\ref{seccomparisons} and \S\ref{sec:dRHecke}, we define canonical classes $\omega_{\psi,\vartheta}$ in the de Rham cohomology of $Z_r$ and relate them to the de Rham classes for $X_r$ appearing in the $p$-adic Gross--Zagier formula due to Bertolini--Darmon--Prasanna (\cite{BDP1}). In this way, we can connect the value of Katz's $p$-adic $L$-function to the Bloch--Kato logarithm of the classes $z_{\psi,\vartheta}$ evaluated at the classes $\omega_{\psi,\vartheta}$ (see Proposition \ref{propGRF}). Note that for each factorization of the form \eqref{introfac} we obtain, through the procedure outlined before, classes $z_{\psi,\vartheta}$ and $\omega_{\psi,\vartheta}$ that \emph{depend} on the given factorization. In order to be able to compare classes coming form different factorizations and to descend to $H^1\bigl(K,V_\chi^\dagger\bigr)$ and $\mathbf{D}_\dR(V_\chi)$ to get the classes $\bz_\chi$ and $\bomega_\chi$ appearing in the statement of Theorem \ref{introteo}, we need to exploit an argument similar to the corresponding argument in \cite{BDP2}, involving the notion of \emph{good pairs}, which we describe in \S\ref{secgoodpairs} and \S\ref{secgoodpairs2}. Here we need to crucially use the fact that $H^1\bigl(K,V_{\chi^\ast}(1)\bigr)$ is an $\mathcal{E}_\chi$-vector space of dimension $1$. This one-dimensionality follows in our case from the assumption $\mathscr{L}_{p,\mathfrak{c}}(\chi^\ast)\neq 0$ combined with results of Castella--Hsieh (\cite[Theorem B]{CH}). It is expected that one should be able to replace the condition $\mathscr{L}_{p,\mathfrak{c}}(\chi^\ast)\neq0$ with $L^\prime(\chi,1)\neq 0$. Indeed, the implication $L^\prime(\chi,1)\neq 0\Rightarrow \dim_{\mathcal{E}_\chi}H^1_f\bigl(K,V_{\chi^\ast}(1)\bigr)=1$ is predicted by the Bloch--Kato conjecture: it might be obtained (under the non-triviality of $p$-adic Abel--Jacobi maps on Heegner cycles) by a generalization of S.-W. Zhang's formula of Gross--Zagier type for $L$-functions of higher weight modular forms (\cite{Zhang}). The reader is referred to Remark \ref{finalrem} for details on the web of (partly conjectural) implications connecting our results, the Bloch--Kato conjecture and Zhang's work. 

The organization of this article is as follows. In Section \ref{secmotives} we review well-known results on motives of Hecke characters and realize them using the generalized Kuga--Sato motives $h^1(Z_r)$. In Section \ref{sec3} we study the relations between the \'etale and de Rham realizations of $h^1(Z_r)$ and $h^1(X_r)$, while in Section \ref{sec4} we use the construction of motives of Hecke characters $\chi$ of infinity type $(\ell+1,-\ell)$ for $\ell\in\N$ by means of generalized Kuga--Sato motives to obtain canonical \'etale and de Rham classes in the \'etale and de Rham realizations, respectively, of $\chi$. In Section \ref{BDPsec}  we review the necessary material from \cite{BDP2} and \cite{BDP1} on generalized Heegner cycles, Katz and BDP $p$-adic $L$-functions, factorization formulas and $p$-adic Gross--Zagier formulas. Finally, in Section \ref{secGRF} we prove our generalized Rubin formula by using factorizations of the form \eqref{introfac} attached to good pairs and the results on $h^1(Z_r)$ collected before. The paper closes with Appendix \ref{app}, where we collect auxiliary results on $p$-adic Galois representations that are (tacitly) used in the main body of the text.
%\begin{remark} As mentioned before, an approach based on factorizations of the form \eqref{introfac} is also used in \cite{BDP3}, \cite{BDP4}, \cite{BDP5} to address different questions on the construction of points on CM elliptic curves out of generalized Heegner cycles; it might be interesting to study possible relations between our result and these works. 
%\end{remark}
\begin{remark}
When the prime number $p$ is inert or ramified in $K$, a Katz-type $p$-adic $L$-function has been constructed by Andreatta and Iovita in \cite{AI24}. It would be interesting to obtain analogues of the results in this paper in this non-split context: we plan to study this kind of questions in a future project. 
\end{remark}

\subsection{Notation and conventions} \label{notation-subsec}

We conclude this introduction by setting some general notation that will be freely used (sometimes without any explicit warning) in the article. Let $\bar\Q$ be the algebraic closure of $\Q$ inside $\C$ and write $x\mapsto\bar x$ for complex conjugation: this is an element of $G_\Q\defeq\Gal(\bar\Q/\Q)$. For every prime number $\ell$, we fix an algebraic closure $\bar\Q_\ell$ of $\Q_\ell$ and an embedding $\bar\Q\hookrightarrow\bar\Q_\ell$. As a rule, we use roman letters for number fields inside $\bar\Q$ and calligraphic ones for local fields; in particular, if a prime number $p$ is understood, then for a number field $F\subset\bar\Q$ we denote by $\mathcal{F}$ the completion of the image of $F$ in $\bar\Q_p$ under $\bar\Q\hookrightarrow\bar\Q_p$. We let $G_F\defeq\Gal(\bar\Q/F)$ and $G_\mathcal{F}\defeq\Gal(\bar\Q_p/\mathcal{F})$ be the corresponding absolute Galois groups. 

For two finite field extensions $\mathcal{E}$ and $\mathcal{F}$ of $\Q_p$, we write $\mathcal{E}\otimes\mathcal{F}$ for $\mathcal{E}\otimes_{\Q_p}\mathcal{F}$, while for a commutative ring $R\neq \Q_p$ and two $R$-modules $M$ and $N$ we use the full notation $M\otimes_R N$ to indicate the tensor product of $M$ and $N$ over $R$. 

For a number field $F$, we let $\A_F$ denote the adele ring of $F$; for each finite place $\mathfrak{q}$ of $F$, we write $\pi_\mathfrak{q}\in\A_F$ for the adele having a fixed uniformizer at $\mathfrak{q}$ and $1$ elsewhere. 

For a number field $F$, we write $\Norm_F$ the (absolute) norm map on ideals; for a finite extension $F'/F$ of number fields, we denote by $\Norm_{F^\prime/F}$ the relative norm map on ideals. 

For a Hecke character $\phi$, the symbol $E_\phi$ will denote the field of coefficients of $\phi$; thus, in accord with the previous notation, $\mathcal{E}_\phi$ will be the completion of the image of $E_\phi$ under $\bar\Q\hookrightarrow\bar\Q_p$ for a fixed prime number $p$. Given Hecke characters $\phi_1,\dots,\phi_n$, we write $E_{\phi_1,\dots,\phi_n}$ for the composite field $E_{\phi_1}\cdot\ldots\cdot E_{\phi_n}$; thus, $\mathcal{E}_{\phi_1,\dots,\phi_n}$ is again the completion of the image of $E_{\phi_1,\dots,\phi_n}$ under $\bar\Q\hookrightarrow\bar\Q_p$. 

Let $p$ be a prime number and let $F$ be a field of characteristic $0$. By a \emph{$p$-adic Galois representation} of $G_F$ we mean a continuous homomorphism $\rho:G_F\rightarrow\Aut(V)$ for a finite-dimensional $\Q_p$-vector space $V$ equipped with a continuous (right) action $v\mapsto v^g$ of $G_F$; we often simply write $V$ for $\rho$. Let $\chi_\cyc\colon G_F\rightarrow\Z_p^\times$ be (the restriction to $G_F$ of) the $p$-adic cyclotomic character. For a $p$-adic representation $\rho$ and $n\in\Z$, set $\rho(n)\defeq\rho\otimes\chi_\cyc^n$, with the usual convention in terms of duals for $n<0$; in particular, $\Q_p(n)=\chi_\cyc^n$. Moreover, if $V$ is the vector space affording $\rho$, then we often write $V(n)$ for $\rho(n)$. For a $p$-adic representation $V$ of $G_F$, where $F$ is a field of characteristic zero, and a finite order character $\phi$ of $G_F$, we define the twist $V(\phi)\defeq(V\otimes\mathcal{E}_\phi)(\phi)$. Sometimes we abuse notation and, for $p$-adic representations $V$ and $V^\prime$ of $G_F$ with coefficients in $\mathcal E$ and $\mathcal{E}^\prime$ (which means that they take values in $\GL_n(\mathcal{E})$ and $\GL_{n^\prime}(\mathcal{E}^\prime)$ for some integers $n\geq1$ and $n'\geq1$) and characters $\phi$ and $\phi^\prime$ with coefficients in $\mathcal E_\phi\supset \mathcal E$ and $\mathcal{E}^\prime_{\phi^\prime}\supset \mathcal{E}^\prime$, we simply denote by $V(\phi)\rightarrow V^\prime(\phi^\prime)$ the map obtained by applying the required scalar extensions. Explicitly, this is the map 
\begin{equation} \label{notation1}
(V\otimes_\mathcal E\mathcal E_\phi)(\phi)\otimes_{\mathcal E_\phi}\mathcal{E}\mathcal{E}^\prime_{\phi,\phi'}\longrightarrow
(V'\otimes_{\mathcal{E}'}\mathcal{E}'_{\phi'})(\phi')\otimes_{\mathcal{E}'_{\phi'}}\mathcal{E}\mathcal{E}'_{\phi,\phi'},
\end{equation} 
where $\mathcal{E}\mathcal{E}'_{\phi,\phi'}$ is the composite of $\mathcal{E}_\phi$ and $\mathcal{E}'_{\phi'}$. We adopt a similar convention for maps between de Rham modules of such representations, so $\D_{\dR,\mathcal{F}}\bigl(V(\varphi)\bigr)\rightarrow\D_{\dR,F}(V^\prime(\varphi^\prime))$ denotes the map obtained from \eqref{notation1}. 

For a commutative ring with unity $R$ and a finitely generated, free $R$-module $M$, we write $M^{\otimes r}$ for the $r$-fold tensor product over $R$ of $M$ with itself; the symmetric group $S_r$ acts on $M^{\otimes r}$ as permutation of the factors and we let $\Sym^r(M)$ be the subspace of $M^{\otimes r}$ consisting of the invariants of $M^{\otimes r}$ under the action of $S_r$. 

Given a modular form $f$ of even weight $k\geq 2$, level $\Gamma_1(N)$ for some integer $N\geq1$ and character $\varepsilon_f$, we denote by $E_f$ the Hecke field of $f$ (\emph{i.e.}, the number field generated over $\Q$ by the Fourier coefficients of $f$) and for a finite prime $\p$ of $E_f$ we write $V_f=V_{f,\p}$ for the $\p$-adic Galois representation over the completion of $E_f$ at $\p$ attached to $f$ by Deligne (\cite{Del-Bourbaki}). We normalize this representation by requiring that if $\Frob_\ell$ is a \emph{geometric} Frobenius, then the polynomial $\det\bigl(1-\Frob_\ell X\,|\,V_f\bigr)$ is equal to $X^2-a_\ell(f)X-\varepsilon_f(\ell)\ell^{k-1}$ for all primes $\ell\nmid Np$, where $p$ is the residual characteristic of $\p$ and $a_\ell(f)$ is the $\ell$-th Fourier coefficient of $f$. 

Finally, the Hodge--Tate weight of $\chi_\cyc:G_{\Q_p}\rightarrow\Z_p^\times$ is $-1$.
 
\section{Motives of Hecke characters}\label{secmotives}

We briefly review the theory of motives of Hecke characters; in doing this, we follow the standard references \cite{DMOS82} and \cite{S86}. We also describe how to construct these motives by means of generalized Kuga--Sato motives introduced in this section but inspired by analogous constructions in \cite{BDP1}, \cite{BDP2}, \cite{BDP3}, \cite{BDP4}. 

\subsection{Algebraic Hecke characters} \label{Heckechar}

We start by fixing some notation and terminology for (algebraic) Hecke characters. See, \emph{e.g.}, \cite[Section 2]{BDP2}, \cite[\S2.1]{KL}, \cite[Chapter 0]{S86} for more details. Fix a number field $F\subset\bar\Q$, whose degree will be denoted by $d$, and an ideal $\mathfrak{f}$ of the ring of integers $\cO_F$ of $F$; denote by $\Sigma\defeq\bigl\{\phi:F\hookrightarrow\bar{\Q}\bigr\}$ the set of distinct embeddings of $F$ into $\bar\Q$. Let $I_{\mathfrak f}$ stand for the group of (non-zero) fractional ideals of $F$ that are prime to $\mathfrak{f}$ and let $E\subset\bar\Q$ be a number field such that $F\subset E$. An $E$-valued \emph{algebraic Hecke character} of conductor dividing $\mathfrak{f}$ is a group homomorphism 
\[ \chi:I_{\mathfrak{f}}\longrightarrow E^\times \] 
satisfying the following property: for every totally positve $\alpha\in F^\times$ with $\alpha\equiv1\pmod{\mathfrak{f}}$, there is an equality 
\[ \chi\bigl((\alpha)\bigr)=\prod_{\phi\in\Sigma}\phi(\alpha)^{n_\phi}, \]
where $n_\phi\in\Z$ for each $\phi\in\Sigma$. The \emph{infinity type} of $\chi$ is the vector $T\defeq{(n_{\phi})}_{\phi\in\Sigma}\in \Z^{\Sigma}$. The infinity type induces an algebraic homomorphism $T \colon F^\times\rightarrow E^\times$, which will be denoted with the same symbol (see, \emph{e.g.}, \cite[Chapter 0, \S 2]{S86} for the general notion of an \emph{algebraic homomorphism} $F^\times\rightarrow E^\times$). The \emph{conductor} of $\chi$ is the smallest (in the sense of divisibility) ideal $\mathfrak{f}_\chi$ of $\cO_F$ such that $\chi$ can be extended to a character modulo $\mathfrak{f}_\chi$. The \emph{weight} $w$ of $\chi$ is the integer characterized by the property that for every $\phi\in\Sigma$ we have $n_{\phi}+n_{\bar\phi}=w$, where $\bar\phi$ denotes the complex-conjugate embedding of $\phi$ (\emph{i.e.}, $\bar\phi(x)\defeq\overline{\phi(x)}$ for all $x\in F$, where $x\mapsto \bar{x}$ is complex conjugation, \emph{cf.} \S \ref{notation-subsec}); notice that the existence of such a $w\in\Z$ is a consequence of the proof of Dirichlet's unit theorem (see, \cite[Lemma 7]{Raghuram} or \emph{e.g.}, \cite[Chapter 0, \S 3]{S86}). In particular, the norm map $\Norm_F:I_F\rightarrow\Q^\times$ on the group $I_F$ of (non-zero) fractional ideals of $F$ gives rise to an algebraic Hecke character of $F$ of infinity type $(1,\dots,1)$, weight $2$ and conductor $\cO_F$. Moreover, for a Hecke character $\chi$ set $\chi^*(x)\defeq\overline{\chi(x)}$ for all $x\in I_{\mathfrak f}$.

Given, as above, an algebraic Hecke character $\chi\colon I_{\mathfrak{f}}\rightarrow E^\times$ of infinity type $T$, there is a unique group homomorphism 
\[ \chi_\A\colon \A_F^\times\longrightarrow E^\times \] 
such that $\chi_\A^{-1}(\{1\})$ is open, the restriction of $\chi_\A$ to $F^\times$ is the character $a\mapsto \prod_{\phi\in\Sigma}\phi(a)^{n_i}$ and $\chi_\A(\pi_\mathfrak{q})=\chi(\mathfrak{q})$ for all prime ideals $\mathfrak{q}\nmid \mathfrak{f}$, where $\pi_\mathfrak{q}$ is the adele having a fixed uniformizer as its $\mathfrak{q}$-component and $1$ elsewhere. 

Now let $K\subset\bar\Q$ be an imaginary quadratic field; then the infinity type of an algebraic Hecke character of $K$ is a pair of integers $(\ell_1,\ell_2)$ corresponding to the ordering $(\phi,\bar\phi)$ of the embeddings of $K$ into $\bar\Q$ in which $\phi\colon K\hookrightarrow \bar\Q$ is the (set-theoretic) inclusion; the weight of the character is $w\defeq\ell_1+\ell_2$. In this case, the \emph{central character} of $\chi$ is the unique Dirichlet character $\varepsilon_\chi$ such that ${\chi|}_{\A_\Q^\times}=\varepsilon_\chi\Norm_\Q^{w}$. 

\subsection{Representations of Hecke characters} \label{secrep} 

Let $F$ be a number field and let $\chi$ be a Hecke character of $F$ of conductor $\mathfrak{f}_\chi$ with values in $E$ and infinity type $T$. Let $\lambda$ be a prime ideal of $E$ and we denote $E_\lambda$ the completion of $E$ at $\lambda$. Recall that $G_F\defeq\Gal(\bar{F}/F)$ is the absolute Galois group of $F$. There is a compatible system of $\lambda$-adic $G_F$-representations
\[ \rho_{\chi,\lambda}\colon G_F\longrightarrow E_{\lambda}^\times, \]  
whose construction is obtained in two steps. Firstly, the algebraic character $T:F^\times\rightarrow E^\times$ gives rise to a morphism $T_\A:\A_F^\times\rightarrow \A_E^\times$ and we denote by $T_\lambda:\A_F^\times\rightarrow E_\lambda^\times$ the composition of $T_\A$ with the canonical projection $\A_E^\times\longepi E_\lambda^\times$. The \emph{$\lambda$-adic avatar} $\chi_\lambda:\A_F^\times/F^\times\rightarrow E_\lambda^\times$ of $\chi$ is defined by $\chi_\lambda\defeq\chi_\A\cdot T_\lambda^{-1}$. Secondly, compose the character $\chi_\lambda$ with the inverse of the geometrically normalized reciprocity map $\mathrm{rec}_F:\A_F^\times/F^\times\rightarrow G_F$ to obtain the searched-for continuous character $\rho_{\chi,\lambda}: G_F\rightarrow E_{\lambda}^\times$. See, \emph{e.g.}, \cite[\S2.1]{KL} or \cite[Chapter 0, \S 5]{S86} for details on this construction. Observe that the eigenvalues of the polynomial $\det\bigl(X-\rho_{\chi,\lambda}(\Frob_\mathfrak{q})\bigr)$
of a geometric Frobenius $\Frob_\mathfrak{q}$ of $\rho_{\chi,\lambda}$ for all (unramified) prime ideals $\mathfrak{q}\nmid \mathfrak{f}_\chi\Norm_{E}(\lambda)$ 
of $F$ have absolute value equal to $\Norm_F(\mathfrak{q})^{w/2}$, where $w$ is the weight of $\chi$ and $\Norm_E$ and $\Norm_F$ are the absolute norm maps on ideals of $E$ and $F$, respectively.  

\subsection{Representations of CM abelian varieties} \label{secrepab}

Let $A$ be an abelian variety defined over a number field $F$ with complex multiplication by $\Phi$; this means that $\Phi$ is a product of CM fields such that $\dim_{\Q}(\Phi)=2\dim(A)$ and there is an an inclusion $i_A\colon \Phi\hookrightarrow \End_{F}(A)\otimes_\Z\Q$, where $\End_{F}(A)$ is the ring of endomorphisms of $A$ defined over $F$. If $\ell$ is a prime number, then the action of $G_{F}$ on the $\ell$-adic Tate module $T_\ell(A)$ of $A$ gives rise to a continuous character 
$\rho_{A,\ell}\colon G_{F}\rightarrow (\Phi\otimes_\Q\Q_\ell)^\times$. Moreover, recall that there is an isomorphism of $G_{F}$-modules
\[ \Ta_\ell(A)^\vee \simeq H^1_\text{\'et}(\bar{A},\Z_\ell)(1), \] 
where $\star(1)$ on the right denotes Tate twist and $\star^\vee$ on the left denotes $\Q_p$-linear dual. See, \emph{e.g.}, \cite[Chapter 1, \S1]{S86} for details. 

Assume that $\Phi=E$ is a field. Then for any finite prime $\lambda$ of $E$ the action of $G_{F}$ on the $\lambda$-adic Tate module $T_\lambda(A)$ of $A$ gives rise to a continuous character 
\[ \rho_{A,\lambda}:G_{F}\longrightarrow E_{\lambda}^\times \] 
(as before, $E_\lambda$ is the completion of $E$ at $\lambda$). Note that if $\mathfrak{q}$ is a finite prime of $F$ of good reduction for $A$ with $\mathfrak{q}\nmid \Norm_{E}(\lambda)$ and $\Frob_\mathfrak{q}$ is a geometric Frobenius at $\mathfrak q$, then the roots of the polynomial 
\[ \det\bigl(1-\rho_{A,\lambda}(\Frob_\mathfrak{q})X\bigr) \] 
are algebraic numbers having absolute value $\Norm_{F}(\mathfrak{q})^{-1/2}$. The $G_F$-representations $\rho_{A,\lambda}$ thus obtained form a compatible system (see, \emph{e.g.}, \cite[Chapter 1, \S1.2]{S86} for the notion of \emph{compatible system} of Galois representations in the present context). 

\subsection{Motives of Hecke characters} 

Let $\mathrm{Mot}(F)$ denote the Tannakian semisimple category of \emph{motives for absolute Hodge cycles} over a number field $F$; we refer to \cite{DMOS82} and \cite{S86} for the terminology and the main definitions on motives (see, in particular, \cite[Chapter 1, \S2.2]{S86}). In the following lines, we briefly recall what we need in this paper. 

An object $M\in\mathrm{Mot}(F)$ can be identified with a triple $M=(X,\Pi,n)$, where $X$ is a smooth projective variety defined over $F$, $\Pi$ is a projector (\emph{i.e.}, an absolute Hodge cycle) with coefficients in $\Q$ and $n\in\Z$. In $\mathrm{Mot}(F)$, the projector $\Pi$ decomposes as a sum of pairwise orthogonal idempotents $\Pi=\Pi^0+\Pi^1+\dots$; we let $h^i(X)\defeq\Pi^i(X)$ for all $i\in\N$. To make sense of this, recall that in the category $\mathrm{Mot}(F)$ every idempotent comes from a splitting of the objects: see, \emph{e.g.}, \cite[p. 35]{S86} for further details. Given a motive $M=(X,\Pi,n)$, we set $h^i(M)\defeq h^{i+n}(X)$. Following \cite{S86}, we denote by $C^q_\mathrm{AH}(X)$ the $\Q$-vector space of absolute Hodge cycles of codimension $q$ on $X$ over $F$. 

Denote by $\End_\Q(M)$ the $\Q$-algebra of endomorphisms of $M$. A motive $M$ is said to have \emph{coefficients} in a number field $E$ if there exists an embedding of $\Q$-algebras $E\hookrightarrow\End_\Q(M)$. We write $\mathrm{Mot}_E(F)$ for the category of motives defined over $F$ with coefficients in $E$. 

We let $M\otimes N$ denote the tensor product in $\mathrm{Mot}_E(F)$; see \cite[p. 47]{S86}) for details. For an integer $n\geq1$, we write $M^{\otimes n}$ for the tensor product $M\otimes\dots\otimes M$ ($n$ factors). For a field extension $F\prime/F$, there is a base change functor, denoted by $M\mapsto M\times_FF^\prime$, from $\mathrm{Mot}_E(F)$ to $\mathrm{Mot}_E(F^\prime)$. For a field extension $E^\prime/E$, there are an extension of coefficients functor, denoted by $M\mapsto M\otimes_EE^\prime$, from $\mathrm{Mot}_E(F)$ to $\mathrm{Mot}_{E^\prime}(F)$ obtained by tensoring with the unit object, and a restriction of coefficients functor, denoted by $M\mapsto M_{|E}$, from $\mathrm{Mot}_{E^\prime}(F)$ to $\mathrm{Mot}_E(F)$ that is obtained by restricting the action of $E^\prime$ to $E$; again, see \cite[p. 47]{S86} for details. We adopt the following standard notation: given motives $(X,\Pi_X,n_X)$ and $(Y,\Pi_Y,n_Y)$ in $\mathrm{Mot}_E(F)$, we denote by $(X\times _F Y,\Pi_X\Pi_Y,n_X+n_Y)$ their product, meaning that $\Pi_X\Pi_Y$ is the projector $(\Pi_X,\Pi_Y)$ acting as $\Pi_X$ (respectively, $\Pi_Y$) on the first (respectively, second) factor of $X\times _FY$; similar conventions will be adopted also in presence of more than two projectors.  

A motive $M\in \mathrm{Mot}_E(F)$ is equipped for all $i\in\N$ with Betti (\emph{i.e.}, singular) realizations $H^i_\mathrm{Betti}(M_\sigma)$ in $E$-vector spaces, one for each embedding $\sigma:E\hookrightarrow \C$, a de Rham realization $H^i_\mathrm{dR}(M)$ in filtered $E\otimes_\Q F$-modules and for each prime $\lambda$ of $E$ an \'etale realization $H^i_\text{\'et}(M_\lambda)$ in $\lambda$-adic $G_F$-representations; furthermore, these realizations satisfy suitable comparison isomorphisms (\emph{cf.} \cite[Chapter 1, \S2.2.2]{S86}). The \emph{rank} of $M$ is the dimension $r$ of $H^i_\mathrm{Betti}(M_\sigma)$ as an $E$-vector space, which is independent of the choice of $\sigma$. It follows that $H^1_\mathrm{dR}(M)$ is a free $E\otimes_\Q F$-module of rank $r$ and $H^1_\text{\'et}(M_\lambda)$ is a vector space of dimension $r$ over $E_\lambda$.  

In the next definition, $\chi$ is a Hecke character of $F$ with coefficients in $E$ and conductor $\mathfrak{f}_\chi$.

\begin{definition}\label{defmotives} 
A \emph{motive for $\chi$} with coefficients in $E$ is a motive $M\in\mathrm{Mot}_E(F)$ of rank $1$ such that for every prime $\lambda$ of $E$ and all primes $\mathfrak{q}$ of $F$ such that $\mathfrak{q}\nmid \mathfrak{f}_\chi\Norm_E(\lambda)$ the action of $G_F$ on the $1$-dimensional $E_{\lambda}$-vector space $H^1_\text{\'et}(M_\lambda)$ is given by multiplication by $\chi(\mathfrak{q})$.
\end{definition}

If $M_E(\chi)$ is a motive for $\chi$ with coefficients in $E$ and $E^\prime/E$ is a field extension, then we set 
\[ M_{E^\prime}(\chi)\defeq M_E(\chi)\otimes_EE^\prime. \]
Motives for Hecke characters exist and are unique (\emph{cf.} \cite[Chapter 2, Theorem 4.1 and \S5]{S86}): we will construct them out of Tate motives, Artin motives and motives of abelian varieties, as explained below. 
 
\subsection{Tate motives} \label{Tatemotives} 	

Let $\Q(-1)\defeq h^2(\mathds{P}^1)$, where $\mathds{P}^1$ is the projective line; for every prime number $\ell$, the $\ell$-adic \'etale realization of $\Q(-1)$ is the $\Q_\ell$-linear dual $\Hom\bigl(\Q_\ell(1),\Q_\ell\bigr)$ of the $\ell$-adic cyclotomic character $\Q_\ell(1)$. We let $\Q(1)$ denote the dual of $\Q(-1)$, so that the $\ell$-adic \'etale realization of $\Q(1)$ is $\Q_\ell(1)$. See, \emph{e.g.}, \cite[p. 30]{S86} for details and a description of the other realizations of $\Q(1)$ and $\Q(-1)$. For $M\in\mathrm{Mot}_E(F)$ and $n\in\Z$, set $M(n)\defeq M\otimes \Q(n)$ for the $n$-th Tate twist of $M$ (tensor product in the Tannakian category $\mathrm{Mot}_E(F)$). One may also check (\emph{e.g.}, by comparing the eigenvalues of $\Frob_\ell$) that the motive $M_{\Q}(\Norm_K)$ of the norm map is $\Q(-1)$.  
 
\subsection{Motives of abelian varieties} \label{abmotives} 

Let $K$ be an imaginary quadratic field. Let us fix an algebraic Hecke character $\chi$ of $K$ of infinity type $(-1,0)$ with values in $E_\chi$. Let $H_\chi$ be a finite abelian extension of $K$ such that the Hecke character $\nu_\chi\defeq\chi \circ\Norm_{H_\chi/K}$
takes values in $K^\times$, where $\Norm_{H_\chi/K}$ is the relative norm map on ideals. Then, by a result of Casselman (see, \emph{e.g.}, \cite[Theorem 4.1]{GS81} or \cite[Theorem 6]{Sh71}), there exists an elliptic curve $A_\chi$ with complex multiplication by $\cO_K$, defined over $H_\chi$, whose Serre--Tate character is $\nu_\chi$ (see, \emph{e.g.}, \cite{Gross} or \cite[Theorem 10]{ST}). Thus, for every prime number $\ell$, the $\ell$-adic representation of $G_{H_\chi}$ on the $\ell$-adic Tate module $\Ta_\ell(A_\chi)$ of $A_\chi$ is given by the $\ell$-adic realization of $\nu_\chi$. From here on, we fix the isomorphism $\iota:\cO_K\simeq \End_{H_\chi}(A_\chi)$ in such a way that $\iota (x)$ for every $x\in\cO_K$ acts on the cotangent space $\Omega^1(A_\chi/H_\chi)$ of $A_\chi$ as multiplication by $x$ (via $K\hookrightarrow H_\chi$). 

Consider the abelian variety 
\[ B_\chi\defeq\mathrm{Res}_{H_\chi/K}(A_\chi) \] 
obtained from $A_\chi$ by restriction of scalars from $H_\chi$ to $K$. Set $G_\chi\defeq\Gal(H_\chi/K)$. For each $\rho\in G_\chi$, let $A^\rho_\chi\defeq A_\chi\times_{H_\chi,\rho}H_\chi$, where the product is defined with respect to $\rho:H_\chi\rightarrow H_\chi$ (here we prefer the notation already used for motives and write $\times_FF^\prime=\times_{\Spec{F}}\Spec(F^\prime)$ instead of the somewhat more common notation $\otimes_{F}F^\prime$, which we reserve for base change of coefficient fields). Consider the ring 
\[ R_\chi\defeq\bigoplus_{\rho\in G_\chi}\Hom_{H_\chi}(A_\chi,A_\chi^\rho)\rho, \] 
where $\Hom_{H_\chi}$ stands for  the (abelian, additive) group of isogenies defined over $H_\chi$, with multiplication defined by the formula 
\[ (\phi_\rho\cdot\rho)\cdot(\phi_\tau\cdot\tau)\defeq(\phi_\tau^{\rho}\circ\phi_\rho)\cdot(\tau\rho) \] 
(here, for a given $\phi_\tau:A_\chi\rightarrow A_\chi^\tau$, we denote by $\phi_\tau^\rho:A_\chi^{\rho}\rightarrow A_\chi^{\tau\rho}$ the induced canonical map); see, \emph{e.g.}, \cite[p. 10]{Den} for details. Then $\Phi_\chi\defeq R_\chi\otimes_\Z\Q$ is the full endomorphism ring of $B_\chi$: it is a commutative, semisimple $K$-algebra of dimension $h_\chi\defeq[H_\chi:K]$. We may also fix an isomorphism $R_\chi\simeq\End_K(B_\chi)$ in such a way that the square 
\[ \xymatrix{R_\chi\ar[r]^-\simeq & \End_K(B_\chi)\\\cO_K\ar@{^(->}[u] \ar[r]^-\simeq & \End_{H_\chi}(A_\chi),\ar@{^(->}[u]}
\] 
where the vertical arrows are the canonical embeddings (more precisely, the embedding of $\cO_K$ into $R_\chi$ takes 
$x$ to the composition $A_\chi\xrightarrow{\iota (x)} A_\chi\rightarrow A_\chi^\rho$ on the $\rho$-summand), commutes.  
With this normalization, $B_\chi$ becomes an abelian variety defined over $K$ with complex multiplication by $\Phi_\chi$. Denote by
\[ \psi_{B_\chi}\colon \A_K^\times\longrightarrow \Phi_\chi^\times \] 
the Serre--Tate character associated with $B_\chi$ (see, \emph{e.g.}, \cite[p. 196]{GS81} for the notion of Serre--Tate characters in this setting). For every homomorphism $\lambda:\Phi_\chi\rightarrow\C$, we thus obtain a homomorphism $\psi_{B_\chi,\lambda}\defeq\lambda\circ \psi_{B_\chi}:A_K^\times\rightarrow \C^\times$. By \cite[\S4.8]{GS81}, the sets $\bigl\{\psi_{B_\chi,\lambda}\mid\lambda:\Phi_\chi\rightarrow\C\bigr\}$ and $\bigl\{\chi\psi_{0}\mid\psi_0: G\rightarrow\C^\times\bigr\}$ are equal. We conclude that there are an idempotent $e_\chi$ and an identification $E_\chi=e_\chi\Phi_\chi$ such that the square 
\[ \xymatrix{\A_K^\times\ar[r] \ar[d]^-{\chi}& \Phi_\chi^\times\ar@{=}[d]^-{\cdot e_\chi} \\
E_\chi^\times \ar@{=}[r] &e_\chi\Phi_\chi^\times} \]
commutes. 

In the following definition, let the Hecke character $\chi$ have infinity type $(-1,0)$.

\begin{definition}\label{def1} 
The \emph{motive of $\chi$} is $M_{E_\chi}(\chi)\defeq\bigl(h^1(B_\chi),e_\chi,1\bigr)$.
\end{definition}

\begin{remark}\label{remdef1} 
There are equalities 
\[ M_{E_\chi}(\chi)=M_{E_\chi}(\chi\Norm_K)\times K(1)=\bigl(h^1(B_\chi),e_\chi,1\bigr), \]
so Definition \ref{def1} is equivalent to defining the motive $M_{E_\chi}(\chi\Norm_K)$ to be $\bigl(h^1(B_\chi),e_\chi,0\bigr)$. Note that $\chi\Norm_K$ has infinity type $(0,1)$. 
\end{remark}

Now we consider the abelian variety $B_{\chi}^\ast=\mathrm{Res}_{H_\chi/K}(A_\chi^\ast)$, where $A_\chi^\ast=A_\chi$ as before as an $H_\chi$-scheme, now equipped with the isomorphism $\iota ^\ast \colon \mathcal{O}_K\rightarrow\End_{H_\chi}(A_\chi)$ normalized so that the action of $\iota^\ast(x)$ on $\Omega^1(A_\chi/H_\chi)$ is via $\bar{x}$. The same construction as before is then used to define the motive of the character $\chi^\ast\Norm_K$, where $\chi^\ast$ satisfies $\chi^*(x)=\overline{\chi(x)}$ and has infinity type $(1,0)$.

\begin{definition}\label{def2} 
Let the Hecke character $\chi$ have infinity type $(0,-1)$. The \emph{motive of $\chi$} is $M_{E_\chi}(\chi)\defeq\bigl(h^1(B_\chi^*),e_{\chi^*},1\bigr)$.
\end{definition}

\begin{remark} As in Remark \ref{remdef1}, Definition \ref{def2} is equivalent to defining the motive $M_{E_\chi}(\chi^*\Norm_K)$ to be $(h^1(B_\chi^*),e_{\chi^*},0)$. Note that $\chi^*\Norm_K$ has infinity type $(1,0)$.
\end{remark}

Definitions \ref{def1} and \ref{def2} allow us to define motives for all Hecke characters of infinity type $(-1,0)$ and $(0,-1)$, and for all Hecke characters that may be obtained from these characters by applying an integer power of the norm map. To complete the definition for all Hecke characters, we still need to introduce motives of finite order characters, which are discussed in the next subsection.

\subsection{Artin motives} \label{Artinmotives}

To discuss motives of finite order Hecke characters, we need to briefly recall the notion of Artin motives; see \cite[p. 211]{DMOS82} and \cite[Chapter 2, \S2.4.1]{S86} for details. 

Let $F$ be a number field. We denote by $\mathrm{Rep}_F$ the Tannakian category of representations of $G_F$ on finite-dimensional $\Q$-vector spaces. For a variety $X$ defined over $F$ of dimension $0$, the set $X(\bar{F})$ is finite and equipped with a (left) $G_F$-action; we thus obtain a representation $\rho_X=\Q^{X(\bar{F})}$ in $\mathrm{Rep}_F$. It is easy to check that absolute Hodge cycles ${C}^0_\mathrm{AH}(X\times Y)$ for $0$-dimensional $F$-varieties $X$ and $Y$ are mapped to $G_F$-homomorphisms of representations, and if we denote by $\mathrm{Mot}^0(F)$ the category of such varieties equipped with morphisms induced by absolute Hodge cycles, then we obtain a fully faithful functor $\mathrm{Mot}^0(F)\rightarrow \mathrm{Rep}_F$. One may define the category of \emph{Artin motives} as the smallest Tannakian subcategory of $\mathrm{Rep}_F$ containing the image of $\mathrm{Mot}^0(F)$. 

Now let $\varphi$ be a finite order Hecke character (of infinity type $(0,0)$) of $F$ with coefficients in a number field $E$ containing $F$. Via the reciprocity law of (global) class field theory, one associates with $\varphi$ a representation $\rho_\varphi\colon G_F\rightarrow E^\times$ of $G_F$, uniquely characterized by the formula $\rho_\varphi(g)(x)=\varphi(\mathfrak{a}_g)x$ for all $x\in E$, where $\mathfrak{a}_g\in \A_F^\times/F^\times$ is the image of $g\in G_F$ via the geometrically normalized Artin map. The representation $\rho_\varphi$ factors through the Galois group $G_\varphi\defeq\Gal(H_\varphi/K)$ of a suitable abelian extension $H_\varphi$ of $K$, and we still denote by $\rho_\varphi:G_\varphi\rightarrow E^\times$ the map induced by $\rho_\varphi$ on $G_\varphi$. Fix an isomorphism $X_E\defeq F[X]/(f)\simeq E$; then $X$ is a variety defined over $F$ whose set of geometric points $X_E(\bar{F})$ is naturally equipped with a $G_F$-action; this gives rise to a representation $\rho_{X_E}$ of $G_F$ on the finite-dimensional $\Q$-vector space $\Q^{X_E(\bar{F})}$, as before. Let us consider the projector 
\[ e_\varphi\defeq\frac{1}{|G_\varphi|}\cdot\sum_{g\in G_\varphi}\rho_\varphi(g)g \]
acting on $\mathrm{Rep}_F$. The pair $(\rho_{X_E},e_\varphi)$ corresponds to the Artin motive $(X_E,e_\varphi)$; the \emph{motive of $\varphi$} is $(X_E,e_\varphi,0)$, which is a motive defined over $K$ with coefficients in $E$.

 \subsection{Construction of motives of Hecke characters}\label{motives} 
 
Here we construct the motive of Hecke characters of imaginary quadratic fields. Let $\chi$ be a Hecke character of an imaginary quadratic field $K$ of infinity type $(\ell_1,\ell_2)$ with coefficients in $E_\chi$. Let $\Sigma=\{\chi_1,\dots,\chi_t\}$ be a set of $t$ not necessarily distinct Hecke characters of $K$; suppose that $\chi_j$ has coefficients in $E_{\chi_j}$ for $j\in\{1,\dots,t\}$. We assume that we have constructed the motive $M_{E_{\chi_j}}(\chi_j)$ of $\chi_j$ for $j\in\{1,\dots,t\}$. Let $s$ be an integer such that $\chi_1\cdots\chi_t\Norm_K^s$ has the same infinity type as $\chi$, so that we may write 
\begin{equation} \label{factorization}
\chi=\varphi\chi_1\cdots\chi_t\Norm_K^s
\end{equation} 
for some finite order character $\varphi$. Denote by $E_\Sigma$ the composite field of $E_{\chi_1},\dots,E_{\chi_t}, E_\varphi$ and note that $E_\chi\subset E_\Sigma$. For any $\star\in\{\varphi,\chi_1,\dots,\chi_t\}$, set $M_\Sigma(\star)\defeq M_{E_\star}(\star)\otimes_{E_\star}E_\Sigma$. Now define the motive
%\begin{equation}\label{motive1}
\[ M_{{\Sigma}}(\chi)\defeq\bigl(M_{\Sigma}(\chi_1)\otimes \dots\otimes M_{\Sigma}(\chi_t)\otimes M_{{\Sigma}}(\varphi)\bigr)(-s), \]
%\end{equation}
where the tensor product is taken in the (Tannakian) category $\mathrm{Mot}_{E_{\Sigma}}(K)$ of motives over $K$ with coefficients in ${E_{\Sigma}}$. 
Then $M_{{\Sigma}}(\chi)$ is a motive associated with $\chi$ with coefficients in $E_{\Sigma}$. 

Now we descend coefficients to $E_\chi$. Let $H_{\Sigma}$ be a finite extension of $K$ such that $\chi_i\circ\Norm_{H_\Sigma/K}$ and $\varphi\circ\Norm_{H_\Sigma/K}$ take values in $K^\times$, for all $i=1,\dots,t$.  
For $\star\in \{\varphi,\chi_1,\dots,\chi_t\}$,  
let $M_{K}(\star\circ\Norm_{H_\Sigma/K})$ be the motive of $\star\circ\Norm_{H_\Sigma/K}$; these are then motives defined over $H_\Sigma$ with coefficients in $K$, 
and we may consider the motive 
\[ M_{E_\chi}(\chi\circ\Norm_{H_\Sigma/K})\defeq\left(\bigotimes_{\star\in\{\varphi,\chi_1,\dots,\chi_t\}}M_{K}(\star\circ\Norm_{H_\Sigma/K})\!\right)(-s) \]
defined over $H_\Sigma$ with coefficients in $E_\chi$. 
As indicated, this is a motive for $\chi\circ\Norm_{H_\sigma/K}$ with coefficients in $E_\chi$ and there is an isomorphism 
%\begin{equation} \label{iso2}
\[ M_{E_{\Sigma}}(\chi)\times_K H_\Sigma\simeq M_{E_{\chi}}(\chi\circ\Norm_{H_\Sigma/K})\otimes_{E_\chi}E_{\Sigma}. \]
%\end{equation}
in $\mathrm{Mot}_{E_{\Sigma}}({H_\Sigma})$. Then the operator 
\[ e_\Sigma\defeq\frac{1}{[E_\Sigma:E_\chi]}\cdot\sum_{\sigma\in \Gal(E_\Sigma/E_\chi)}\sigma \] 
acting on $M_{E_{\chi}}(\chi\circ\Norm_{H_\Sigma/K})\otimes_{E_\chi}E_{\Sigma}$ defines a projector, denoted by the same symbol, in $\End_{\mathrm{Mot}_{E_{\Sigma}}({H_\Sigma})}\bigl(M_{E_{\Sigma}}(\chi)\times_K H_\Sigma\bigr)$ such that 
\[ e_\Sigma\bigl(M_{E_{\Sigma}}(\chi)\times_K H_\Sigma\bigr)\simeq M_{E_{\chi}}(\chi\circ\Norm_{H_\Sigma/K}) \]
as motives over $H_\Sigma$ with coefficients in $E_\chi$. Since $e_\Sigma$ is clearly defined over $K$, we can define the \emph{motive of $\chi$ over $K$ with coefficients in $E_\chi$ associated with factorization \eqref{factorization}} to be 
\[ M_{E_\chi}^{(\Sigma)}(\chi)\defeq e_\Sigma M_{E_{\Sigma}}(\chi). \]  
By \cite[Chapter 1, Section 5]{S86}, the motive attached to $\chi$ is unique up to isomorphism and is independent of factorization \eqref{factorization}.
%so there is an isomorphism 
%\[\Theta_{\Sigma}: M_{E_\chi}(\chi)\overset\simeq\longrightarrow M_{E_\chi}^{(\Sigma)}(\chi)\]
%in the category $\mathcal{M}_K(E_\chi)$. 
%Our next task is to make this isomorphism explicit in some relevant situations. 

\subsection{Standard factorization} \label{remell}

Fix an elliptic curve $A$ with CM by the ring of integers $\cO_K$ of $K$ and defined over the Hilbert class field $H$ of $K$. Let $\nu_A$ be the Serre--Tate character of $A$ and let $\chi_A$ be a Hecke character of $K$ of weight $-1$ and coefficients in $E_{\chi_A}$ such that $\nu_A=\chi_A\circ\Norm_{H/K}$ (see, \emph{e.g.}, \cite[Theorem 4.1, (v)]{GS81}).

Assume that the infinity type of $\chi_A$ is $(-1,0)$, so the infinity type of $\chi_{A}^*=\chi_{A^*}$ is $(0,-1)$, where $\nu_{A}^\ast=\chi_{A}^\ast\circ\Norm_{H/K}$ and $\nu_{A^\ast}=\nu_A^\ast$ is the Serre--Tate character of $A^\ast$. Therefore, $\chi_A\chi_A^\ast$ has infinity type $(-1,-1)$ and coefficients in $E_{\chi_A}$. 
 
Now let $\chi$ be a Hecke character of $K$ of infinity type $(\ell_1,\ell_2)$. The definition adopted in \cite{S86} of $M_{E_\chi}(\chi)$ 
is obtained via the procedure explained in \S\ref{motives} by means of the factorization $\chi=\varphi\chi_A^{-\ell_1}(\chi_A^\ast)^{-\ell_2}$, which we call \emph{standard factorization} of $\chi$.

\subsection{Deligne--Scholl motives} \label{secscholl} 

The motives of Hecke characters of infinity type $(\ell-1,0)$ for an even integer $\ell\geq2$ can be realized in terms of Deligne--Scholl motives of modular forms, as we recall in this subsection. 

Fix an even integer $k\geq 2$. Let $\psi$ be a Hecke character of an imaginary quadratic field $K$ of infinity type $(k-1,0)$ with coefficients in $E_\psi$ and denote by $\mathfrak{f}_\psi$ the conductor of $\psi$. Recall that the symbol $\varepsilon_\psi$ stands for the central character of $\psi$, as defined in \S\ref{Heckechar}. Set $N\defeq D_K\Norm(\mathfrak{f}_\psi)$. The \emph{theta series} attached to $\psi$ is the cusp form $\theta_\psi\in S_{k}\bigl(\Gamma_0(D_K\Norm(\mathfrak{f}_\psi)),\varepsilon_K\varepsilon_\psi\bigr)$ defined by $\theta_\psi\defeq\sum_{\mathfrak{q}} \psi(\mathfrak{q}) q^{\Norm \mathfrak{q}}$ (see, \emph{e.g.}, \cite[Proposition 3.13]{BDP2} or \cite[\S3.3]{KL}). 

Consider the congruence subgroup 
\[ \Gamma_\psi\defeq\bigl\{\smallmat abcd\in\Gamma_0(N)\mid(\varepsilon_K\varepsilon_\psi)(a)=1\bigr\} \]
and denote by $X_{\Gamma_\psi}$ the compact modular curve of level $\Gamma_\psi$. Set $r\defeq k-2$. Let $W_r$ be the $(k-1)$-dimensional Kuga--Sato variety over the modular curve $X_{\Gamma_\psi}$. Recall that, by definition, $W_r$ is the canonical desingularization of the $r$-fold self-product of the universal elliptic curve $\mathcal{E}_{\Gamma_\psi}\rightarrow X_{\Gamma_\psi}$; this variety is defined over $\Q$. See, \emph{e.g.}, \cite[\S2.1]{BDP1} for details on the construction. We may then consider the \emph{Chow motive} $(W_r,\epsilon_W,0)$ of modular forms of weight $k$ and level $\Gamma_{\psi}$; here $\epsilon_W$ is a projector defined in 
\cite[\S1.1]{Sch90} (see also \cite[\S2]{LV17}, \cite[\S2]{Ne92}). The \emph{Scholl motive of $\theta_\psi$} is the motive 
\[ M(\theta_\psi)\defeq\bigl(W_r\otimes E_\psi,\pi_{\theta_\psi}\epsilon_W,0\bigr), \] 
where the projector $\pi_{\theta_\psi}$ is described, \emph{e.g.}, in \cite[\S4.2.0]{Sch90} as the $\theta_\psi$-eigenspace for the action of the Hecke algebra (see also \cite[\S2]{LV17}, \cite[\S3]{Ne92}). Then $M(\theta_\psi)$ is a Grothendieck (\emph{i.e.}, homological) motive defined over $\Q$ with coefficients in $E_\psi$. 

\begin{remark}
Strictly speaking, the results quoted above are stated for $\Gamma_1(N)$ instead of $\Gamma_\psi$; however, in light of \cite[Remark 2.5]{BDP2}, they work in this more general situation. 
\end{remark}

Let $p\nmid N$ be a prime number. Let $V_{\theta_\psi,p}$ denote Deligne's $p$-adic Galois representation 
\[ \rho_{\theta_\psi,p}:G_\Q\longrightarrow \Aut(V_{\theta_\psi,p})\simeq\GL_2(E_\psi\otimes_\Q\Q_p) \] 
attached to $\theta_\psi$, which is characterized (see, \emph{e.g.}, \cite[Theorem 1.2.4]{Sch90}) by requiring that the \emph{geometric} Frobenius at primes $q\nmid Np$ has characteristic polynomial 
\[ P_q(X)\defeq X^2-a_q(\psi)X+\varepsilon_\psi(q)q^{k-1} \]  
(here we view $a_q(\psi)$ and $\varepsilon_\psi(q)$ as elements of $E_\psi\otimes_\Q\Q_p$). By \cite[Theorem 1.2.1]{Sch90}, there is an isomorphism of $p$-adic Galois representations  
\[ V_{W}\defeq\epsilon_WH^{r+1}_\text{\'et}(\bar W_r,\Q_p) \simeq H^1\bigl(\Gamma_\psi,\Sym^{r}(\Q_p)\bigr) \]
of $G_\Q$, where $\bar{W}_r\defeq W_r\otimes_{\Q}\bar\Q$; from this isomorphism, taking into account the Hecke action, we obtain an isomorphism of $p$-adic Galois representations  
\[ \pi_{\theta_\psi}(V_W\otimes_{\Q}E_\psi)\simeq V_{\theta_\psi,p} \]
(see, \emph{e.g.}, \cite[Theorem 1.2.4]{Sch90}). We may also consider the $p$-adic representation 
\[ V_{\psi,p}\defeq H^1_\text{\'et}\bigl(M_{E_\psi}(\psi)_p\bigr)=\bigoplus_{\mathfrak p\mid p} H^1_\text{\'et}\bigl(M_{E_\psi}(\psi)_\mathfrak p\bigr), \]
where the direct sum is taken over all the finite primes $\p$ of $E_\psi$ above $p$ (thus, there is a splitting $E_\psi\otimes_\Q\Q_p\simeq\oplus_{\mathfrak{p}\mid p}E_{\psi,\mathfrak{p}}$ with $E_{\psi,\mathfrak{p}}$ the completion of $E_\psi$ at $\mathfrak{p}$). For a finite-dimensional representation $\rho:G_K\rightarrow\Aut(V)$, we denote by $\Ind_\Q^K(\rho)\defeq\Ind_{G_\Q}^{G_K}(\rho)$ the induction of $\rho$ from $G_K$ to $G_{\Q}$. By comparing Hecke polynomials, we easily see that 
\begin{equation} \label{induction}
V_{\theta_\psi,p}\simeq\Ind_\Q^K(V_{\psi,p})
\end{equation} 
as $p$-adic representations of $G_\Q$ (note that the eigenvalues of a geometric Frobenius at all unramified primes $\mathfrak{q}$ have absolute value $\Norm_K(\mathfrak{q})^{(k-1)/2}$).   

\begin{proposition} \label{scholl-shappa} 
The splitting $M(\theta_\psi)\times K\simeq M_{E_\psi}(\psi)\oplus  M_{E_\psi}(\psi^*)$ holds in $\mathrm{Mot}_{E_\psi}(K)$. 
\end{proposition}

\begin{proof} To start with, note that $M(\theta_\psi)\times_\Q K$ is a Grothendieck motive over $K$ with coefficients in $E_\psi$, so it defines a motive for absolute Hodge cycles with projectors. By the properties of induction and restriction, it follows from \eqref{induction} that the $\p$-adic \'etale realizations of the motives in the statement of the proposition coincide for all but finitely many prime numbers $p$ and for each prime $\p$ above $p$. Therefore, the compatible systems of Galois representations attached to the two motives in question are the same, which implies that these motives are equal. \end{proof}

Let us denote by $e_K$ (respectively, $e_K^\ast$) the projector in $\End_{\mathrm{Mot}_{E_\psi}({K})}\bigl(M_{E_\psi}(\psi)\oplus M_{E_\psi}(\psi^\ast)\bigr)$ corresponding to the first (respectively, second) summand $M_{E_\psi}(\psi)$( respectively $M_{E_\psi}(\psi^\ast)$). Then, by Proposition \ref{scholl-shappa}, there is an isomorphism  
%\begin{equation}\label{isoscholl1}
\[ e_K\bigl(M(\theta_\psi)\times K\bigr)\simeq M_{E_\psi}(\psi) \]
%\end{equation} 
in $\mathrm{Mot}_{E_\psi}(K)$. In particular, $M_{E_\psi}(\psi)$ can be realized as a Grothendieck motive. It follows that there is an isomorphism 
\begin{equation}\label{isogalpsi}
e_K\pi_{\theta_\psi}(V_W\otimes_{\Q}E_\psi)\simeq V_{\psi,p}
\end{equation}
of $p$-adic representations of $G_K$.

\begin{remark}
As observed in \cite[\S3.4.2]{FL}, $M(\theta_\psi)$ can also be realized as a Chow motive, so the same is true of $M_{E_\psi}(\psi)$.  
\end{remark}

If $\mathcal{E}_{\psi}$ stands for the completion of the image of $E_\psi$ via our fixed embedding $\bar\Q\hookrightarrow\bar\Q_p$, then we simply write $V_{\theta_\psi}$ and $V_\psi$ for the $\mathcal{E}_\psi$-components of $V_{\theta_\psi,p}$ and $V_{\psi,p}$, respectively; the isomorphism class of these representations is independent of the choice of the finite prime $\p$ associated with the embedding $\bar\Q\hookrightarrow\bar\Q_p$. 

\subsection{Norm factorizations} \label{secnorm} 

As in \S\ref{secscholl}, fix an even integer $k\geq2$ and set $\ell\defeq\frac{k-2}{2}$. Let $\chi$ be a Hecke character of infinity type $(1+\ell,-\ell)$ of an imaginary quadratic field $K$. Then there is a factorization, which we call \emph{norm factorization}, of the form $\chi=\psi_\chi\Norm_K^{-\ell}$, where $\psi_\chi$ is a Hecke character of $K$ of infinity type $(k-1,0)$ that only depends on $\chi$. 
In this case, $E_\chi=E_{\psi_\chi}$ and, using the notation introduced in the general construction of \S\ref{motives},  we have $\Sigma=\Sigma_{\psi_\chi}=\{\psi_\chi\}$ and $s=-\ell$. Thus, the motive resulting from this factorization by means of the general construction is $M_{E_{\psi_\chi}}(\chi)\defeq\bigl(W_r\otimes {E_\chi},\epsilon_We_K\pi_{\theta_\psi},\ell\bigr)$.

\subsection{Theta factorizations} \label{secthetafact} 

Now we slightly generalize the construction in the previous subsection. As in \S\ref{secscholl} and \S\ref{secnorm}, fix an even integer $k\geq 2$ and set $\ell\defeq\frac{k-2}{2}$. Let $\psi$ be a Hecke character of  an imaginary quadratic field $K$ of infinity type $(k-1,0)$ and let $\chi$ be a Hecke character of $K$ of infinity type $(1+\ell,-\ell)$. Then there is a factorization, which we call \emph{theta factorization}, of the form $\chi=\psi\vartheta$ for some Hecke character $\vartheta$ of $K$ of infinity type $(-\ell,-\ell)$, and, using the notation in \S\ref{motives}, we let $\Sigma\defeq\Sigma_{\psi,\vartheta}=\{\psi,\vartheta\}$ in this case (note that $s=0$ in the notation of \S\ref{motives}; also, note that a norm factorization considered in \S\ref{secnorm} is, in particular, a theta factorization). Set $e_{\psi,\vartheta}\defeq e_{\Sigma_{\psi,\vartheta}}$ and let $M_{E_{\vartheta}}(\vartheta)\defeq(X_\vartheta,e_\vartheta,\ell_\vartheta)$ be the motive of $\vartheta$ with coefficients in a field $E_\vartheta$. The motive resulting from this factorization by means of the general construction is 
\[ M_{E_{\chi}}(\chi)\defeq\bigl(W_r\otimes {E_\psi}\times X_\vartheta,e_{\psi,\vartheta}e_Ke_\vartheta\pi_{\theta_\psi}\epsilon_W,\ell_\vartheta\bigr). \] 
in the rest of the paper, we will need to be even more specific and factor $\vartheta=\varphi(\chi_A\chi_A^\ast)^\ell$ for some finite order character $\varphi$; in the notation of \S\ref{motives}, we let $\Sigma=\Sigma_{\psi,\vartheta}\defeq\bigl\{\psi,(\chi_A\chi_A^\ast)^\ell,\varphi\bigr\}$ (still, $s=0$ in this case). Set $e_{\psi,\vartheta}\defeq e_{\Sigma_{\psi,\vartheta}}$. Then the motive resulting from this factorization by means of the general construction is 
\begin{equation} \label{motivedesc}
M_{E_{\chi}}(\chi)\defeq\bigl(W_r\otimes {E_\psi}\times X_{E_\varphi}\times B^\ell\times (B^\ast)^\ell ,e_{\psi,\vartheta}e_K\pi_{\theta_\psi}\epsilon_W e_\varphi e_{\chi_A}^\ell e_{\chi_A^\ast}^\ell,0\bigr).
\end{equation}
Moreover, observe that if we define $\varphi_A\defeq(\chi_A\chi_A^\ast)\Norm_K$, then $\varphi_A$ is a finite order character and we may realize the norm factorization in \S\ref{secnorm} as $\chi=(\varphi\varphi_A^\ell)\psi\Norm_K^{-\ell}$, so that $\psi_\chi=(\varphi\varphi_A^\ell)\psi$.  

\section{Realizations of generalized Kuga--Sato motives} \label{sec3}

In order to describe the motive attached to a theta factorization $\chi=\varphi\psi(\chi_A\chi_A^*)^\ell$ as in \S\ref{secthetafact}, in this section we shall study some motives that we call \emph{generalized Kuga--Sato motives}: they are a variation of similar objects previously introduced by Bertolini--Darmon--Prasanna in \cite{BDP1}. 
%Generalized Kuga--Sato motives can be used to construct motives of Hecke characters $\chi$ of infinity type $(\ell+1,-\ell)$ using the formalism in \S\ref{secthetafact}. Indeed, for any fixed theta-factorization $\chi=\psi\vartheta=(\varphi\varphi_A^\ell)\psi\Norm_K^{-\ell}$ if we define $Z_\varphi=Z_r\times X_\varphi$, then, after extending coefficients to $E_\psi$, the motive $M_{E_{\chi}}(\chi)$ can be constructed as the image of a projector acting on $Z_\varphi$ as explained in \eqref{motivedesc}. However, we shall take a slightly different view point by studying the realization of the product variety $B^\ell\times(B^*)^\ell$ instead of considering each factor separately. \Bmu

\subsection{Generalized Kuga--Sato motives} \label{secBDPmot}  

Let $k\geq 2$ be an even integer; as before, set $r\defeq k-2$ and $\ell\defeq r/2$. Recall that $B\defeq\mathrm{Res}_{H/K}(A)$ and $B^*\defeq\mathrm{Res}_{H/K}(A^*)$. Define the $K$-variety 
\[ Z_r\defeq W_r\times B^\ell\times(B^*)^\ell. \]
Let $\Sigma_\ell\defeq \Bmu_2^{\ell}\rtimes S_\ell$, where $\Bmu_2\defeq\{1,-1\}$ and $S_\ell$ is the symmetric group on $\ell$ letters; let $j:\Sigma_\ell\rightarrow \Bmu_2$ be the homomorphism that is the identity on $ \Bmu_2$ and the sign character on $S_\ell$. The group $\Sigma_\ell$ acts on $B^\ell$ and $(B^*)^\ell$ as multiplication by $-1$ for the $ \Bmu_2$ subgroup
and the permutation of the factors for the subgroup $S_\ell$, and thus we obtain an action of $\Sigma_\ell$ both on $B^\ell$ and on $(B^*)^\ell$. Put 
\begin{equation} \label{epsilon}
\epsilon\defeq\frac{1}{2^\ell\ell!}\sum_{\sigma\in\Sigma_\ell}j(\sigma)\sigma.
\end{equation}
View $\epsilon$ as an element of $\Q[\Aut(B^\ell)]$ (respectively, $\Q[\Aut((B^*)^\ell)]$) and call it $\epsilon_B$ (respectively, 
$\epsilon_{B}^*$), then define the projector $\epsilon_Z\defeq\epsilon_W\epsilon_B\epsilon_B^*$.

\begin{definition}\label{genKS} 
The \emph{generalized Kuga--Sato motive} is the Chow motive 
$(Z_r,\epsilon_Z,0)$. 
\end{definition}

\subsection{de Rham realizations} 
We first describe the de Rham realization of generalized Kuga--Sato motives, which also explains the effect of the identity map $A^\times\times(A^*)^\ell\rightarrow A^r$ on de Rham cohomology alluded to before. 

\begin{lemma} \label{lemmaA}  
\begin{enumerate}
\item There are canonical identifications of $K$-vector spaces 
\[ \epsilon_BH^\ell_\dR(B^\ell/K)=\Sym^\ell\bigl(H^1_\dR(B/K)\bigr), \quad\epsilon_B^*H^1_\dR\bigl((B^*)^\ell/K\bigr)=\Sym^\ell(H^1_\dR(B^*/K)\bigr). \]     
\item $\epsilon_BH^j_\dR(B^\ell/K)=0=\epsilon_B^*H^j_\dR(B^\ell/K)$ for $j\neq \ell$. 
\end{enumerate}
\end{lemma}

\begin{proof} Recall from \S\ref{abmotives} that the choice of the isomorphism $\iota:\cO_K\simeq\End_{H}(A)$ was fixed so that 
for $x\in\cO_K$ the endomorphism $\iota(x)$ acts on the cotangent space $\Omega^1(A/H)$ of $A$ as multiplication by $x$. Recall also that $\End_K(B)\simeq R=\bigoplus_{\sigma}\Hom_{H}(A,A^\sigma)\sigma$ and, finally, the embedding $\iota_B:\cO_K\hookrightarrow R$ defined via the composition 
\[ A\xrightarrow{\iota(x)} A\overset{p_A}\longrightarrow A^\sigma \]
on the $\sigma$-component, where $p_A:A^\sigma=A\times_{H,\sigma}H\twoheadrightarrow A$ is projection on the first component.
Since $\iota(-1)$ acts on $\Omega^1(A/H)$ as multiplication by $-1$, it follows from the definitions recalled above that $\iota(-1)^\sigma$ acts as multiplication by $-1$ on 
$\Omega^1(A^\sigma/H)$. Since 
\begin{equation}\label{ResExt}
B\times_K H\simeq \bigoplus_\sigma A^\sigma,\end{equation} we have $\Omega^1(B/H)\simeq\bigoplus_\sigma\Omega^1(A^\sigma/H)$, so $\iota(-1)^\sigma$ acts as multiplication by $-1$ on $\Omega^1(B/H)$. The action of $(-1)$ on $\Omega^1(B/H)$ is the base change of the action of the same element on $\Omega^1(B/K)$, and therefore $(-1)$ also acts as $-1$ on $\Omega^1(B/K)$. Thus,
we conclude that $\iota_B(-1)$ acts as multiplication by $-1$ on $H^1_\dR(B/K)$ and as the identity both on $H^2_\dR(B/K)$ and on $H^0_\dR(B/K)$. Taking into account the action of the subgroup $(\Bmu_2)^\ell$ of $\Sigma_\ell$, we see that the only term that is not annihilated by $\epsilon_B$ in the K\"unneth decomposition 
\[ H^\ast_\dR(B^\ell/K)\simeq \bigoplus_{(i_1,\dots,i_\ell)}H^{i_1}_\dR(B/K)\otimes\dots\otimes H^{i_\ell}_\dR(B/K) \]
(where the direct sum is over all $\ell$-tuples $(i_1,\dots,i_\ell)$ with $i_j\in\{0,1,2\}$ for each $j=1,\dots,\ell$) 
is $H^1_\dR(B/K)^{\otimes\ell}$, so we have 
\[ \epsilon_BH^\ast_\dR(B^\ell/K)\simeq \epsilon_BH^1_\dR(B/K)^{\otimes\ell}. \] 
On the other hand, the action of the subgroup $S_\ell$ of $\Sigma_\ell$ on $H^1_\dR(B/K)^{\otimes\ell}$ corresponds to the permutation action of the factors, and therefore $\epsilon_BH^1_\dR(B/K)^{\otimes\ell}\simeq\Sym^\ell(H^1_\dR(B/K))$. The argument for $B^*$ is the same, taking into consideration that $\iota(\overline{-1})=\iota(-1)$. \end{proof}

Motivated by the previous lemma, we introduce the $K$-vector space  
\[ \mathbf{Sym}^\ell_\dR(B\otimes B^*)\defeq\Sym^\ell\bigl(H^{1}_\dR({B}/K)\bigr)\otimes_{K} \Sym^\ell\bigl(H^{1}_\dR({B}^*/K)\bigr). \]
By Lemma \ref{lemmaA} and \cite[Lemma 2.2]{BDP1}, the de Rham realization of the generalized Kuga--Sato motive $(Z_r,\epsilon_Z,0)$ is $\epsilon_ZH^\bullet_\dR(Z_r/K)\simeq \epsilon_W H^{k-1}_\dR(W_r/K)\otimes_K\mathbf{Sym}^\ell_\dR(B\otimes B^*)$.

\subsection{Étale realizations} \label{sec-etaleAJ} 

Following the approach we took in the de Rham case, we describe the \'etale realizations of Kuga--Sato motives.
Fix a prime number $p\nmid N=D_K\Norm_K(\mathfrak{f}_\psi)$.
%and let $\mathcal{F}$ be finite extension of $\Q_p$ containing the image of $K$ under the fixed embedding $\bar\Q\hookrightarrow\bar\Q_p$, which we simply write $\mathcal{F}\supset K$; later, we will assume that $p$ is split in $K$, so 
%we may just take $\mathcal{F}=\Q_p$. 
If $X$ is a variety defined over a field $F$ of characteristic $0$, then $\bar{X}\defeq X\times_F\bar{F}$ is the base change of $X$ to $\bar F$. 

\begin{lemma} \label{lemmaet} 
\begin{enumerate}
\item There are canonical identifications of $G_K$-representations
\[ \epsilon_BH^\ell_\text{\'et}(\bar{B}^\ell,\Q_p)=\Sym^\ell\bigl(H^1_\text{\'et}(\bar{B},\Q_p)\bigr),\quad\epsilon_B^*H^1_\text{\'et}\bigl((\bar{B}^*)^\ell,\Q_p\bigr)=\Sym^\ell\bigl(H^1_\text{\'et}(\bar{B}^*,\Q_p)\bigr). \]     
\item $\epsilon_BH^j_\text{\'et}(\bar B^\ell,\Q_p)=0=\epsilon_B^*H^j_\text{\'et}\bigl((\bar B^*)^\ell,\Q_p\bigr)$ for $j\neq \ell$. 
\end{enumerate}
\end{lemma}

\begin{proof} Recall that for a smooth variety $X$ defined over a number field $F\subset\C$ there are comparison isomorphisms 
\[ H^\ast_\et(\bar{X},\Q_p)\simeq H^\ast_\mathrm{Betti}\bigl(X(\C),\Z\bigr)\otimes_\Z\Q_p,\quad H^\ast_\mathrm{Betti}\bigl(X(\C),\Z\bigr)\otimes_\Z\C\simeq H^\ast_\mathrm{dR}(X/F)\otimes_F\C, \]
where $H^\ast_\mathrm{Betti}$ denotes Betti (\emph{i.e.}, singular) cohomology; the desired result is an immediate consequence of Lemma \ref{lemmaA}. \end{proof}

As in the de Rham case, we introduce the $G_K$-representation 
\[ \mathbf{Sym}^\ell_\text{\'et}(B\otimes B^\ast)\defeq\Sym^\ell\bigl(H^1_\text{\'et}(\bar{B},\Q_p)\bigr)\otimes_{\Q_p}\Sym^\ell\bigl(H^1_\text{\'et}(\bar{B}^\ast,\Q_p)\bigr). \] 
It follows from Lemma \ref{lemmaet} that 
\begin{equation}\label{isogal}
\epsilon_ZH^\bullet_\text{\'et}(\bar{Z}_r,\Q_p)=\epsilon_Z H^{k-1}_\text{\'et}(\bar{Z}_r,\Q_p) \simeq \epsilon_WH^{k-1}_\text{\'et}(\bar{W}_r,\Q_p)\otimes\boldsymbol{\Sym}^\ell_\text{\'et}(B\otimes B^\ast).
\end{equation}
Observe that
 $(\chi_A\chi_A^*)^\ell\Norm_K^{2\ell}$ is a direct factor of $\mathbf{Sym}^\ell_\text{\'et}(B\otimes B^*)$ as $G_K$-modules; recalling that $\varphi_A=(\chi_A\chi_A^*)\Norm_K$, we obtain a $G_K$-equivariant map 
\begin{equation}\label{isogal2}
(e_{\chi_A}^\ell e_{\chi_A^\ast}^\ell)\cdot\mathbf{Sym}^\ell_\text{\'et}(B\otimes B^\ast)\longrightarrow E_{\varphi_A^\ell}(\varphi_A^\ell\Norm_K^{\ell})=E_{\varphi_A^\ell}(\varphi_A^\ell)(-\ell).
\end{equation}
Let us define the $G_K$-representations $V_{Z}\defeq\epsilon_ZH^{k-1}_\text{\'et}(\bar{Z}_r,\Q_p)$ and $V_W\defeq\epsilon_WH^{k-1}_\text{\'et}(\bar{W}_r,\Q_p)$.
 Composing \eqref{isogal} and \eqref{isogal2}, we get a map 
\begin{equation} \label{galmap}
V_Z\longrightarrow V_W(\varphi_A^\ell)(-\ell)
\end{equation}
in which $V_W(\varphi_{A}^\ell)$ denotes the tensor product of $G_K$-representations $V_W\otimes E_{\varphi_A^\ell}(\varphi_A^\ell)$. 

\subsection{Comparison with generalized Kuga--Sato varieties} \label{seccomparisons}

Now we want to compare the generalized Kuga--Sato motive in Definition \ref{genKS} with a similar object that was considered by Bertolini--Darmon--Prasanna in \cite{BDP1}, namely, the \emph{generalized Kuga--Sato variety} from \cite[\S2.2]{BDP1}. Consider the $H$-variety $X_r\defeq W_r \times A^r$, which is equipped with the projector $\epsilon_X\defeq \epsilon_W\epsilon_A$ defined in \cite[(2.2.1)]{BDP1}. Similarly to \eqref{epsilon}, the projector $\epsilon_A$ in \cite[(1.4.4)]{BDP1} is the image of the projector 
\[ \frac{1}{2^{r}{r}!}\cdot\sum_{\sigma\in\Sigma_{r}}j(\sigma)\sigma \]
in $\Q\bigl[\Aut(A^{r})\bigr]$. Then $(X_r,\epsilon_X,0)$ is a Chow motive defined over $H$.

Comparing $(Z_r,\epsilon_Z,0)$ and $(X_r,\epsilon_X,0)$, we see that the projector $\epsilon_A$ is replaced by the product $\epsilon_B\epsilon_B^\ast$, with the effect that the factors pertaining to $B^\ell$ and those pertaining to $(B^\ast)^\ell$ are not mixed. We put this observation in the next result, 
for which we need to introduce a couple of notation. First, let $\epsilon_A^{(\ell)}$ the image of the 
projector $\epsilon$ in \eqref{epsilon} in $\Q[\Aut(A^{\ell})]$, and similarly let $\epsilon_{A^\ast}^{(\ell)}$ the image of the 
same projector $\epsilon$ in \eqref{epsilon} in $\Q[\Aut((A^\ast)^{\ell})]$. 
We define the $H$-variety $X_r^\ast=W_r\times A^\ell\times (A^\ast)^\ell$ and the motive 
$(X_r^\ast,\epsilon_{X}^\ast,0)$, where the projector $\epsilon_{X}^\ast$ is defined by 
$\epsilon_{X}^\ast=\epsilon_W\epsilon_A^{(\ell)}\epsilon_{A^\ast}^{(\ell)}$. 

\begin{lemma} \label{lemma1}
The motives $\bigl(X_r^\ast,\epsilon_X^\ast,0\bigr)$ and $\bigl(Z_r,e_{\chi_A}^\ell e_{\chi_A^\ast}^\ell\epsilon_Z,0\bigr)$ are isomorphic over $H$. 
\end{lemma}

\begin{proof} Recall from \eqref{ResExt} that $B$ is isomorphic to $\prod_{\sigma}A^\sigma$ over $H$, where the product is taken over all $\sigma\in G\defeq\Gal(H/K)$. By construction, over $H$ the projector $e_{\chi_A}$ coincides with the projection to the factor $A$; alternatively, observe that for every prime $\lambda$ of $E_{\chi_A}$ the $\lambda$-adic Galois representation 
$H^1_\text{\'et}\bigl(M_{E_{\chi_A}}(\chi_A)_\lambda\bigr)$ attached to the motive $M_{E_{\chi_A}}(\chi_A) $ is $\rho_{\chi_A,\lambda}$ and the restriction of $\rho_{\chi_A,\lambda}$ to $H$ is $\nu_A$ by the relation 
$\nu_A=\chi_A\circ\Norm_{H/K}$, therefore the motive $M_{E_{\chi_A}}(\chi_A)\times_K H$ is isomorphic to 
$A$. Similarly, $B^\ast\simeq\prod_{\sigma\in G}(A^\ast)^\sigma$ over $H$ and over $H$ the projector 
$e_{\chi_A^\ast}$ is the projection to the factor $A^\ast$. Therefore, over $H$ there are isomorphisms of motives
\[ \begin{split}
\bigl(Z_r,e_{\chi_A}^\ell e_{\chi_A^\ast}^\ell\epsilon_Z,0\bigr)\times_K H&=\bigl(Z_r\times_K H,e_{\chi_A}^\ell e_{\chi_A^\ast}^\ell\epsilon_W\epsilon_B\epsilon_B^\ast,0\bigr)\\
&\simeq\Bigl(\bigl(W_r\times B^\ell\times (B^\ast)^\ell\bigr)\times_K H,\epsilon_W e_{\chi_A}^\ell e_{\chi_A^\ast}^\ell\epsilon_W\epsilon_B\epsilon_B^\ast,0\Bigr) \\
&\simeq\Bigl(W_r\times_KH\times A^\ell\times (A^\ast)^\ell, \epsilon_W\epsilon_A^{(\ell)}\epsilon_{A^\ast}^{(\ell)},0\Bigr)\\
&\simeq\bigl(X_r^\ast,\epsilon_X^\ast,0\bigr),
\end{split} \] 
where for the penultimate isomorphism we note that, by definition, over $H$ the projector $\epsilon_B$ is equal to $\epsilon_A^{(\ell)}$ 
and the projector $\epsilon_B^\ast$ is equal to $\epsilon_{A^\ast}^{(\ell)}$.  
\end{proof}

%\begin{lemma}\label{lemma2}
%The BDP motive $(X_r,\epsilon_X,0)$ is a submotive of $(Z_r,e_{\chi_A}^\ell e_{\chi_A^*}^\ell\epsilon_Z,0)$ over $H$.
%\end{lemma}

%\begin{proof}
%Thanks to Lemma \ref{lemma1}, it is enough to show that  
%$(X_r^*,\epsilon_X^*,0)$ is a submotive of $(X_r,\epsilon_X,0)$. For this, note that
We want to study the effect of the identity map $A^\ell\times (A^\ast)^\ell\rightarrow A^r$. This map yields a map of $H$-varieties 
$X_r^\ast\rightarrow X_r$. The idempotent $\epsilon_A$ can be written as follows. There is a canonical map $\Sigma_\ell\times\Sigma_\ell\rightarrow\Sigma_r$ defined as follows: given two pairs $(a_1,\sigma_1), (a_1,\sigma_2)\in \Sigma_\ell$, we define their image in $\Sigma_r$ as $(a_1a_1, \sigma_a\sigma_2)$, where $a_1a_2$ is the image of $(a_1,a_2)\in \Bmu_2^{\ell}\times\Bmu_2^\ell$ via the canonical isomorphism $\Bmu_2^{\ell}\times \Bmu_2^\ell\simeq \Bmu_2^r$, 
while $\sigma_1\sigma_2$ is the permutation acting as $\sigma_1$ on the first $\ell$ letters and as $\sigma_2$ on the second $\ell$ letters. Now denote by $\Sigma_\ell^2$  the image of $\Sigma_\ell\times\Sigma_\ell$ in $\Sigma_r$ under this map. In this way, we can thus consider $\epsilon_A^{(\ell)}\epsilon_{A^*}^{(\ell)}$ as an element of $\Q[\Aut(A^r)]$. Fix a system of representatives  $\mathfrak{G}$ of the cosets $\Sigma_r/\Sigma_\ell^2$, whose cardinality is $(2\ell)!/(\ell!)^2$. Define 
\begin{equation} \label{epsilonstar}
\epsilon^\ast\defeq\frac{(\ell!)^2}{(2\ell)!}\cdot\sum_{g\in\mathfrak{G}}g,
\end{equation}
which we view in $\Q\bigl[\Aut(\epsilon_A A^r)\bigr]$. Then 
\begin{equation} \label{factproj}
\epsilon_A=\epsilon^\ast\epsilon_{A}^{(\ell)}\epsilon_{A^\ast}^{(\ell)}.
\end{equation}
Now we use Lemma \ref{lemma1} and equality \eqref{factproj} to provide a comparison between the de Rham cohomology of the generalized Kuga--Sato variety over $H$ and that of the generalized Kuga--Sato motive $(Z_r,\epsilon_Z,0)$ from Definition \ref{genKS}. By \cite[Lemma 1.8]{BDP1}, we know that $\epsilon_A H^*_\dR(A^r/H)\simeq\Sym^r H^1_\dR(A/H)$; moreover, the same argument in \cite{BDP1} shows (as in the proof of Lemma \ref{lemmaA}) that 
\[ \epsilon_A^{(\ell)} H^*_\dR(A^\ell/H)\simeq\Sym^\ell H^1_\dR(A/H),\quad\epsilon_{A^*}^{(\ell)} H^*_\dR\bigl((A^*)^\ell/H\bigr)\simeq\Sym^\ell H^1_\dR(A^*/H). \] 
Observe that the identity map induces canonical isomorphisms 
\[ H^{1,0}_\dR(A/H)\simeq H^{0,1}_\dR(A^*/H),\quad H^{0,1}_\dR(A/H)\simeq H^{1,0}_\dR(A^*/H). \]
In particular, the identity map gives a canonical isomorphism $\iota_A:H^1_\dR(A^*/H)\simeq H^1_\dR(A/H)$ of $H$-vector spaces that interchanges the de Rham filtrations, as explained above. 

By Lemma \ref{epsilonstar}, taking de Rham cohomology over $H$ we get a canonical isomorphism of $H$-vector spaces
\begin{equation} \label{eq*}
\epsilon^\ast e_{\chi_A}^\ell e_{\chi_A^\ast}^\ell\epsilon_W\epsilon_B\epsilon_B^* H^\bullet_\dR(Z_r/H) 
\simeq \epsilon_X H^\bullet_\dR(X_r/H)
\end{equation}
that can be described explicitly as follows. As observed before (\emph{cf.} the proof of Lemma \ref{lemma1}), the motives $\bigl(h^1(B),e_{\chi_A},0\bigr)$ and $\bigl(h^1(A),\mathrm{Id},0\bigr)$ are canonically isomorphic over $H$, so there is a canonical isomorphisms of $H$-vector spaces $e_{\chi_A} H^1_\dR(B/H)\simeq H^1_\dR(A/H)$; similarly, we have also an isomorphism of $H$-vector spaces $e_{\chi_A^*} H^1_\dR(B^*/H)\simeq H^1_\dR(A^*/H)$. Therefore, there is 
a canonical isomorphism of $H$-vector spaces 
\begin{equation} \label{equa*} 
\bigl(e_{\chi_A}^\ell e_{\chi_A^*}^\ell\bigr)\cdot\boldsymbol{\Sym}^\ell_\dR(B\otimes B^*/H)\simeq\boldsymbol{\Sym}^\ell_\dR(A\otimes A^*/H).\end{equation}
The identity map $A\times A^*\rightarrow A\times A$ also induces an isomorphism 
of $H$-vector spaces (switching the de Rham filtrations as explained before) 
\begin{equation}\label{eq***}
\boldsymbol{\Sym}^\ell_\dR(A\otimes A^*/H)\longrightarrow \Sym^\ell H^1_\dR(A/H)\otimes_H \Sym^\ell H^1_\dR(A/H).\end{equation}
Finally, let 
\begin{equation}\label{eq****}
\Sym^\ell H^1_\dR(A/H)\otimes_H \Sym^\ell H^1_\dR(A/H)\longrightarrow  H^1_\dR(A/H)^{\otimes r}\end{equation}
be the map induced by the identity map in each factor.  
It follows from the definition of $\epsilon^*$ (\emph{cf.} \eqref{epsilonstar}) that the composition of 
\eqref{eq***} and \eqref{eq****} induces an isomorphism of $H$-vector spaces 
\begin{equation}\label{equa**}
\epsilon^*\boldsymbol{\Sym}^\ell_\dR(A\otimes A^*/H)\simeq \boldsymbol{\Sym}^r_\dR(A/H)\end{equation}
where $\boldsymbol{\Sym}^r_\dR(A/H)=\Sym^r(H^1_\dR(A/H))$, which interchanges the de Rham filtration as we explained before. Isomorphism \eqref{eq*} is then described as the composition of maps 
\begin{equation}\label{ISODR}
\begin{split}
\epsilon^* e_{\chi_A}^\ell e_{\chi_A^*}^\ell\epsilon_Z H^\bullet_\dR(Z_r/H) &\simeq 
\epsilon_WH^\bullet_\dR(W_r/H)\otimes_H\Bigl(\epsilon^*\boldsymbol{\Sym}^\ell_\dR(A\times A^*/H)\!\Bigr)\\
&\simeq \epsilon_WH^\bullet_\dR(W_r/H)\otimes_H \boldsymbol{\Sym}^r_\dR(A/H)\\
&\simeq \epsilon_XH^\bullet_\dR(X_r/H),
\end{split} \end{equation}
where the first isomorphism follows upon combining Lemma \ref{lemmaA} and isomorphism \eqref{equa*}, while the second is a consequence of \eqref{equa**} and the third is \cite[Proposition 2.4]{BDP1}.

Now, with similar methods, we compare the \'etale realizations of the motives $(Z_r,\epsilon_Z,0)$ and $(X_r,\epsilon_r,0)$. First of all, notice that there is an isomorphism 
\begin{equation}\label{tateiso}
\Ta_p(A)\simeq \Ta_p(A^\ast).
\end{equation} 
of $G_H$-representations. Indeed, since $p$ splits in $K$ as $p=\p\bar{\p}$, we have 
\[ \Ta_p(A)\simeq \Ta_\p(A)\oplus \Ta_{\bar{\p}}(A), \]
where $\Ta_\p(A)$ (respectively, $\Ta_{\bar{\p}}(A)$) is the $\p$-adic (respectively, $\bar{\p}$-adic) Tate module of $A$; moreover, $\Ta_{\bar{\p}}(A)\simeq \Ta_\p(A^*)$ and $\Ta_{\p}(A^*)\simeq \Ta_{\bar{\p}}(A)$, whence isomorphism \eqref{tateiso}. Now define the $G_H$-representation $V_{X}\defeq\epsilon_X H^{k-1}_\text{\'et}(\bar{X}_r,\Q_p)$. Combining \eqref{epsilonstar} and \eqref{tateiso}, we see that there is an isomorphism 
%\begin{equation}\label{isoetale}
\[ \bigl(\epsilon^\ast e_{\chi_A}^\ell e_{\chi_A^\ast}^\ell\bigr)\cdot\boldsymbol{\Sym}^\ell_\text{\'et}(B\otimes B^\ast)\simeq \boldsymbol{\Sym}^r_\text{\'et}(A)=\Sym^{r}\bigl(H^1_\text{\'et}(\bar{A},\Q_p)\bigr) \]
%\end{equation}
of $G_H$-representations (where $\bar{A}\defeq A\times_H\bar{\Q}$), and therefore there is an isomorphism 
 \begin{equation}\label{isogal4}
 %\begin{split}
 (\epsilon^*e_{\chi_A}^\ell e_{\chi_A^\ast}^\ell)V_Z
 %&
 \simeq \epsilon_W H^{k-1}_\text{\'et}(\bar{W}_r,\Q_p)\otimes_{\Q_p}\boldsymbol{\Sym}^r_\text{\'et}(A)
% \\&
 \simeq \epsilon_X H^\bullet_\text{\'et}(\bar{X}_r,\Q_p)
% \\ &
 =V_X
 %\end{split}
 \end{equation}
 of $G_H$-representations, where the second and third isomorphisms come from the \'etale version of \cite[Lemma 1.8]{BDP1} and \cite[\S3.1]{BDP1} (\emph{cf.} also \cite[\S4.2]{CH}). 

\subsection{Constructing de Rham classes}\label{sec:dRHecke}

Since $H^{1,0}(B/K)$ and $H^{1,0}(B^\ast/K)$ are free $\Phi$-modules of rank $1$, applying the idempotents $e_{\chi_A}$ and $e_{\chi_A^\ast}$ to chosen $\Phi$-bases of $H^{1,0}(B/K)$ and of $H^{1,0}(B^\ast/K)$, we obtain generators $\omega_B$ of $e_{\chi_A}H^{1,0}(B/K)$ and $\eta_B$ of $e_{\chi_A^\ast}H^{1,0}(B^*/K)$. Let us denote by $p_j\colon B^\ell\rightarrow B$ and $p_j^\ast\colon (B^\ast)^\ell\rightarrow B^*$ the $j$-th projections.  
In order not to confuse the various projectors $\epsilon_B^\ast$ and $\epsilon^\ast$ and the projection $p_j^\ast$ already defined with pull-back maps, we use the symbol $(-)^{\boldsymbol{\ast}}$ to denote the pull-back maps; for example, $(p_j^\ast)^{\boldsymbol{\ast}}$ is the pull-back of $p_j^\ast$. With this notation, we define a class 
\[ \tilde\omega_B^\ell\tilde\eta_B^\ell=(\tilde\omega_B\tilde\eta_B)^\ell\defeq(p_1)^{\boldsymbol{\ast}}\omega_B\wedge \dots\wedge (p_\ell)^{\boldsymbol{\ast}}\omega_B\wedge(p_1^*)^{\boldsymbol{\ast}}\eta_B\wedge \dots\wedge (p_\ell^*)^{\boldsymbol{\ast}}\eta_B. \]
One can easily check that $(\epsilon_B)^{\boldsymbol{\ast}}\tilde\omega_B^\ell\tilde\eta_B^\ell=\tilde\omega_B^\ell\tilde\eta_B^\ell$, so $\tilde\omega_B^\ell\tilde\eta_B^\ell\in \mathbf{Sym}^\ell_\dR(B\otimes B^\ast)$. 
Applying $e_{\chi_A}^\ell e_{\chi_A^\ast}^\ell$, we thus obtain a class 
\[ \omega_B^\ell\eta_B^\ell=e_{\chi_A}^\ell e_{\chi_A^\ast}^\ell(\tilde\omega_B^\ell\tilde\eta_B^\ell)\in(e_{\chi_A}^\ell e_{\chi_A^\ast}^\ell)\cdot\mathbf{Sym}^\ell_\dR(B\otimes B^\ast). \]
Let us write ${\omega}_{\theta_\psi}\in \mathrm{Fil}^{k-1}\bigl(\epsilon_W H^{k-1}_\mathrm{dR}(W_r/K)\bigr)$ for the differential form attached to ${\theta_\psi}$ as in \cite[Corollary 2.3]{BDP1}. We may then define a class 
\[ \omega_{\theta_\psi}\otimes\omega_B^\ell\eta_B^\ell\in e_{\chi_A}^\ell e_{\chi_A^\ast}^\ell\epsilon_ZH^\bullet_\dR(Z_r/K)\simeq \epsilon_WH^{k-1}_\mathrm{dR}(W_r/K)\otimes_K \Bigl((e_{\chi_A}^\ell e_{\chi_A^\ast}^\ell)\cdot\mathbf{Sym}^\ell_\dR(B\otimes B^\ast)\!\Bigr). \]
Now we compare the de Rham cohomology class $\omega_{\theta_\psi}\otimes\omega_B^\ell\eta_B^\ell$ with a similar class arising from the motive $(X_r,\epsilon_X,0)$. Fix two classes $\omega_A\in H^{1,0}_\mathrm{dR}(A/H)$ and $\eta_A\in H^{0,1}_\mathrm{dR}(A/H)$ as in \cite[(1.4.2)]{BDP1} satisfying $\langle \omega_A,\eta_A\rangle=1$, where $\langle\cdot,\cdot\rangle$ is the algebraic cup product pairing on the de Rham cohomology of $A$. Following \cite[equation (1.4.6)]{BDP1}, we define the class $\omega_A^\ell\eta_A^\ell\in\Sym^{2\ell}\bigl(H^1_\mathrm{dR}(A/H)\bigr)=\boldsymbol{\Sym}^r_\dR(A/H)$ as 
\[ \omega_A^\ell\eta_A^\ell\defeq(\epsilon_A)^{\boldsymbol{\ast}}\bigl((p_1)^{\boldsymbol{\ast}}\omega_A\wedge\dots\wedge(p_\ell)^{\boldsymbol{\ast}}\omega_A \wedge (p_{\ell+1})^{\boldsymbol{\ast}}\eta_A\wedge\dots
\wedge(p_{2\ell})^{\boldsymbol{\ast}}\eta_A\bigr). \]
It follows that  
\begin{equation} \label{omegaetaA}
\omega_A^\ell\eta_A^\ell=\frac{(\ell!)^2}{(2\ell)!}\cdot\sum_{I}(p_1)^{\boldsymbol{\ast}}\varpi_{1,I}\wedge\dots\wedge 
(p_r)^{\boldsymbol{\ast}}\varpi_{r,I},
\end{equation}
where the sum is taken over all the subsets $I$ of $\{1,\dots,2\ell\}$ of cardinality $\ell$, while $\varpi_{i,I}=\omega_A$ if $i\in I$ and $\varpi_{i,I}=\eta_A$ if $i\not\in I$. 
%Since $H^{0,1}_\mathrm{dR}(A/H)$ is equal to $H^{1,0}_\mathrm{dR}(A^*/H)$, this class actually lies in the subspace 
%\[\Sym^{\ell}(H^{1,0}_\mathrm{dR}(A/H))\otimes_H \Sym^\ell(H^{1,0}_\mathrm{dR}(A^*/H))
%\subset\Sym^{\ell}(H^1_\mathrm{dR}(A/H))\otimes_H \Sym^{\ell}(H^1_\mathrm{dR}(A^*/H)). 
%\] 
Noting that $k-1+2\ell=2r+1$, we define a class 
\[ \omega_{\theta_\psi}\otimes\omega_A^\ell\eta_A^\ell\in \epsilon_W H^{k-1}_\mathrm{dR}(W_r/H)\otimes_H\boldsymbol{\Sym}^r_\dR(A/H)\simeq\mathrm{Fil}^1\epsilon_XH^{2r+1}_\mathrm{dR}(X_r/H)(r), \]
where, as before, ${\omega}_{\theta_\psi}$ is the differential form attached to ${\theta_\psi}$, base changed to $H$. Recall that we chose non-zero classes $\omega_B\in e_{\chi_A}H^{1,0}(B/K)$ and $\eta_B\in e_{\chi_A^*}H^{1,0}(B^\ast/K)$. Comparing \eqref{epsilonstar} and \eqref{omegaetaA} with the equality $(\epsilon_B)^{\boldsymbol{\ast}}\tilde\omega_B^\ell\tilde\eta_B^\ell=\tilde\omega_A^\ell\tilde\eta_A^\ell$, we see that we can make these choices in such a way that they are compatible under isomorphism \eqref{equa**}; namely, we can assume that this isomorphism over $H$ takes $\omega_B$ to $\omega_A$ and $\eta_B$ to $\eta_A$. Let us fix such choices once and for all. It follows from the previous construction that the image of the class
\[ \omega_{\theta_\psi}\otimes\omega_B^\ell\eta_B^\ell\in \epsilon_W H^{k-1}_\dR(W_r/H)\otimes_H\mathbf{Sym}^\ell_\dR(B\otimes B^\ast/H) \]
via isomorphism \eqref{ISODR} is the class
\[ {\omega}_{\theta_\psi}\otimes\omega_A^\ell\eta_A^\ell\in  \epsilon_W H^{k-1}_\dR(W_r/H)\otimes_H\mathbf{Sym}^r_\dR(A/H). \] 
In other words, $\epsilon^*(\omega_{\theta_\psi}\otimes\omega_B^\ell\eta_B^\ell)=\omega_{\theta_\psi}\otimes\omega_A^\ell\eta_A^\ell$.

\subsection{\'Etale Abel--Jacobi maps}\label{AJMAPS}

For any field extension $F/K$, there is an \'etale Abel--Jacobi map 
\begin{equation}\label{AJZ1}
{\Phi_{\et,Z}^{(F)}}:
{\CH^{k-1}_0(Z_r)(F)\longrightarrow H^1_f\bigl(F,V_Z(k-1)\bigr)}
\end{equation} 
that we compare with the \'etale Abel--Jacobi map for the motive $(X_r,\epsilon_r,0)$. For this, assume that $F$ contains $H$ and let 
\begin{equation}\label{PhiX}
{\Phi_{\text{\'et},X}^{(F)}}:\CH^{k-1}_0(X_r)(F) \longrightarrow H^1_f\bigl(F,V_X(k-1)\bigr) 
\end{equation}
be the \'etale Abel--Jacobi map. The projector $\epsilon^*$ introduced in \eqref{epsilonstar} induces using \eqref{isogal4} a commutative diagram 
\begin{equation}\label{diagram}
\xymatrix@C=29pt{
\CH^{k-1}_0(Z_r)(F)\ar[r]^-{\Phi_{\text{\'et},Z}^{(F)}}\ar[d] & H^1_f\bigl(F,V_Z(k-1)\bigr)\ar[d] \ar[r]& 
H^1_f\bigl(F,V_W\otimes\boldsymbol{\Sym}^\ell_\text{\'et}(B\otimes B^\ast)(k-1)\bigr)\ar[d]
\\
\CH^{k-1}_0(X_r)(F)\ar[r]^-{\Phi_{\text{\'et},X}^{(F)}} & H^1_f\bigl(F,V_X(k-1)\bigr) \ar[r]&  
H^1_f\bigl(F,V_W\otimes \boldsymbol{\Sym}^r_\text{\'et}(A)(k-1)\bigr)
} 
\end{equation}
in which the vertical arrows are multiplication by the projector $\epsilon^\ast e_{\chi_A}^\ell e_{\chi_A^*}^\ell$. 

Now we rewrite the Abel--Jacobi map above in terms of de Rham cohomology. Fix a finite field extension $\mathcal{F}$ of $\Q_p$ containing the image of $H$ under the fixed embedding $\bar\Q\hookrightarrow \bar\Q_p$. By composition, one defines the dotted map (of $\Q$-vector spaces) in the diagram 
\[ \xymatrix@C=35pt@R=40pt{
\CH^{k-1}_0(X_r)(\mathcal{F})\ar[r]^-{\Phi^{(\mathcal{F})}_{\text{\'et},X}} \ar@{-->}[rrd]^-{\log_{X,\mathcal{F}}}& H^1_f\bigl(\mathcal{F},V_X(k-1)\bigr)\ar[r]^-{\mathrm{comp}} & 
\frac{\epsilon_X H^{2r+1}_\dR(X_r/\mathcal{F})(r+1)}{\mathrm{Fil}^0\epsilon_X H^{2r+1}_\dR(X_r/\mathcal{F})(r+1)}\ar[d]^\simeq\\
&& \bigl(\mathrm{Fil}^{1}\epsilon_XH^{2r+1}_\mathrm{dR}(X_r/\mathcal{F})(r)\bigr)^\vee,} \] 
where $H^1_f\bigl(\mathcal{F},V_X(k-1)\bigr)$ denotes the finite part of the cohomology group $H^1\bigl(\mathcal{F},V_X(k-1)\bigr)$ corresponding to crystalline extensions, $\mathrm{comp}$ is the \'etale-de Rham comparison isomorphism (see, \emph{e.g.}, \cite[\S3.2 and \S3.3]{BDP1}) and the vertical isomorphism follows from Poincar\'e duality upon noticing that  $\mathrm{Fil}^{r}\epsilon_XH^{2r+1}_\mathrm{dR}(X_r/\mathcal{F})(r)$ and $\mathrm{Fil}^{0}\epsilon_XH^{2r+1}_\mathrm{dR}(X_r/\mathcal{F})(r+1)$ are exact annihilators of each other under the Poincar\'e pairing.

\section{Realizations of Hecke characters} \label{sec4}

The goal of this section is to use generalized Kuga--Sato motives to describe the (\'etale and de Rham) realizations of motives of Hecke characters of imaginary quadratic fields $K$ of infinity type $(\ell+1,-\ell)$ and obtain \'etale Abel--Jacobi maps with values in the (\'etale and de Rham) realizations of such a motive. Fix a prime number $p$ that is coprime with the conductor of $\chi$ and splits in $K$. Let $\chi$ be a Hecke character of $K$ of infinity type $(\ell+1,-\ell)$ for some integer $\ell\geq 0$ and fix a theta factorization $\chi=\psi\vartheta=\varphi\psi(\chi_A\chi_A^*)^\ell=(\varphi\varphi_A^\ell)\psi\mathrm{N}_K^{-\ell}$.

\subsection{\'Etale realizations} 

Here we make more explicit the description of the \'etale realization of the motive of $\chi$ from \S \ref{motivedesc}. 

The $p$-adic Galois representation $\rho_{\theta_\psi}:G_\Q\rightarrow\GL_2(\mathcal{E}_\psi)$ of the theta series $\theta_\psi$ of $\psi$ is, as recalled before, the induction of the $p$-adic representation attached to $\psi$.  Applying the projectors $e_K$ and $e_K^*$, and combining with \eqref{galmap}, we thus obtain maps of $G_K$-representations 
%\begin{equation}\label{isogalbase1}
\[ \begin{split}
   \pi_\chi^\et:V_Z(\varphi)\longrightarrow V_W(\varphi\varphi_A^\ell)(-\ell)&\longrightarrow
V_{\theta_\psi}(-\ell)(\varphi\varphi_A^\ell)\\&\longrightarrow
e_KV_{\theta_\psi}(-\ell)(\varphi\varphi_A^\ell)\simeq V_\chi(-2\ell), 
  \end{split} \]
%\end{equation} 
and 
%\begin{equation}\label{isogalbase1star}
\[ \begin{split}
   \pi_{\chi^*}^\et:V_Z(\varphi^*)\longrightarrow V_W(\varphi^*\varphi_A^\ell)(-\ell)&\longrightarrow
V_{\theta_\psi}(-\ell)(\varphi^*\varphi_A^\ell)\\&\longrightarrow
e_K^*V_{\theta_\psi}(-\ell)(\varphi^*\varphi_A^\ell)\simeq V_{\chi^*}(-2\ell). 
   \end{split} \]
%\end{equation} 
The reader will notice the abuse of notation in the lines above and explained at the end of the introduction to this paper, since we tacitly understood the coefficient rings everywhere; taking care of the coefficients, the maps above read 
\[ \pi_\chi^\et:V_Z(\varphi)\otimes_{\mathcal{E}_\varphi}\mathcal{E}_{\psi,\varphi,\varphi_A^\ell}\longrightarrow V_\chi(-2\ell)\otimes_{\mathcal{E}_\chi}\mathcal{E}_{\psi,\varphi,\varphi_A^\ell} \]
and
\[ \pi_{\chi^*}^\et:V_Z(\varphi^*)\otimes_{\mathcal{E}_\varphi}\mathcal{E}_{\psi,\varphi,\varphi_A^\ell}\longrightarrow V_{\chi^*}(-2\ell)\otimes_{\mathcal{E}_\chi}\mathcal{E}_{\psi,\varphi,\varphi_A^\ell}. \]
Let $H_\varphi$ be the field cut out by $\varphi$ over $H$. The restriction of the character $\varphi$ to $G_{H_\varphi}$ is trivial, while the restriction of $\varphi_A$ to $G_{H_\varphi}$ takes values in $K^\times$. As before, we obtain maps of $G_{H_\varphi}$-representations  
\begin{equation} \label{isogal6}
\pi_\chi^\et:V_Z \otimes \mathcal{E}_{\psi,\vartheta}\longrightarrow V_\chi(-2\ell)\otimes_{\mathcal{E}_\chi}\mathcal{E}_{\psi,\vartheta}
\end{equation}
and
%\begin{equation} \label{isogal6star}
\[ \pi_{\chi^*}^\et:V_Z \otimes \mathcal{E}_{\psi,\vartheta}\longrightarrow V_{\chi^*}(-2\ell)\otimes_{\mathcal{E}_\chi}\mathcal{E}_{\psi,\vartheta} \]
% \end{equation}
with smaller coefficients.

\begin{remark}\label{remcoeff}
On the one hand, since $\mathcal{E}_\chi$ is not contained in $\mathcal{E}_\psi$, it is necessary to consider coefficient fields containing both $\mathcal{E}_\chi$ and $\mathcal{E}_\psi$ in \eqref{isogal6}; on the other hand, since  $\mathcal{E}_{\psi,\vartheta}$ is the composite field of these two fields, the choice of coefficient field in the above diagram can not be improved in general. 
\end{remark}

\subsection{de Rham realizations} \label{sec:dR} 

In this subsection, we make more explicit the description of the de Rham realization of the motive of $\chi$ from  \eqref{motivedesc}. For any Hecke character $\phi$, let us set $H^1_\dR(\phi)\defeq H^1_\dR\bigl(M_{E_\phi}(\phi)\bigr)$, which is a free $K\otimes_{\Q}E_\phi$-module of rank $1$. 

The $K\otimes_{\Q} {E}_{\varphi}$-module $H_\dR^0\bigl({E_\varphi}(\varphi)\bigr)$ is free of rank $1$: we fix a generator $\delta_{\varphi}$ of it. From \eqref{motivedesc}, we obtain an isomorphism of filtered $K$-vector spaces 
\[ \epsilon_ZH^{2r+1}_\dR(Z_r/K)\simeq\bigl(\epsilon_WH^{k-1}_\dR(W_r/K)\bigr)\otimes_K
H^0_\dR\bigl(\mathcal{E}_\varphi(\varphi)\bigr)\otimes_K\mathbf{Sym}^\ell_\dR\Bigl(\!\bigl(e_{\chi_A}^\ell e_{\chi_A^*}^\ell\bigr)B\otimes B^*)\!\Bigr). \]
The element $\omega_{\theta_\psi}\otimes\delta_\varphi\otimes\omega_B^\ell\eta_B^\ell$ gives rise to a class in $\mathrm{Fil}^{r+1}\bigl(\epsilon_ZH^{2r+1}_\dR(Z_r/K)\bigr)$. By \eqref{motivedesc}, the image of $\omega_{\theta_\psi}\otimes\delta_\varphi\otimes\omega_B^\ell\eta_B^\ell$ via the projector $e_{\psi,\vartheta}e_K\pi_{\theta_\psi}e_\varphi $ is a generator of $H^1_\dR(\chi)$. If we do not apply the last projector, $\omega_{\theta_\psi}\otimes\delta_\varphi\otimes\omega_B^\ell\eta_B^\ell$ is a generator of the $K\otimes_\Q E_{\psi,\varphi,\varphi_A^\ell}$-module 
\[ H^1_\dR(\chi)\otimes_{E_\chi}E_{\psi,\varphi,\varphi_A^\ell}=H^1_\dR\Bigl(M_{E_{\psi,\varphi,\varphi_A^\ell}}(\chi)\!\Bigr). \] 
Moreover, $\omega_{\theta_\psi}\otimes\omega_B^\ell\eta_B^\ell$ is also a generator of the $K\otimes_\Q E_{\psi,\varphi_A^\ell}$-module 
\[ H^1_\dR\bigl(\psi\varphi_A^\ell\bigr)\otimes_{E_{\psi\varphi_A^\ell}}E_{\psi,\varphi_A^\ell}=H^1_\dR\Bigl(M_{E_{\psi,\varphi_A^\ell}}\bigl(\psi\varphi_A^\ell\bigr)\!\Bigr). \] 
Now let $H_\varphi$ be the extension of $H$ cut out by $\varphi$. Then $\chi=\psi\varphi_A^\ell$ as Hecke characters of $H_\varphi$; furthermore, since $\varphi_A^\ell$ takes values in $K^\times$ as a Hecke character of $H_\varphi$, we conclude that  
$\omega_{\theta_\psi}\otimes\omega_B^\ell\eta_B^\ell$ is a generator of the free $H_\varphi\otimes_\Q E_\chi$-module $H^1_{\dR}\bigl(M_{E_{\psi}}(\chi)\times H_\varphi\bigr)$. Note that any generator $\omega$ of $H^1_\dR(\chi)$ as $K\otimes_\Q E_\chi$-module also generates $H^1_{\dR}\bigl(M_{{E_{\chi}}}(\chi)\times H_\varphi\bigr)$ as $H_\varphi\otimes_{\Q} E_{\chi}$-module. Since $E_{\psi,\chi}=E_{\psi,\vartheta}$, we conclude that
\begin{equation} \label{scalar}
\omega_{\theta_\psi}\otimes\omega_B^\ell\eta_B^\ell=\lambda\omega
\end{equation}
in $H^1_\dR\bigl(M_{E_{\psi,\vartheta}}(\chi)\times H_\varphi\bigr)$ for some $\lambda\in H_\varphi\otimes_\Q E_{\psi,\vartheta}$. 

Now we use de Rham modules of Galois representations to study de Rham realizations. Recall that the Hodge--Tate weight of the cyclotomic character is normalized to be $-1$. There is an isomorphism of filtered $\mathcal{E}_\psi$-vector spaces
\[ \mathbf{D}_\dR(V_{\theta_\psi})\simeq \mathbf{D}_\dR(V_{ \psi})\oplus \mathbf{D}_\dR(V_{ \psi^*}). \]  
The Hodge--Tate weights of $\theta_\psi$ are $0$ and $k-1$, $V_{\psi}$ has a unique Hodge--Tate weight $k-1$ and $V_{\psi^*}$ has a unique Hodge--Tate weight $0$. Mreover, $V_\chi$ has Hodge--Tate weight $\ell+1$. Thus, taking care of the twists, we get an isomorphism of filtered $\mathcal{E}_{\psi,\vartheta}$-vector spaces 
\[ \mathbf{D}_\dR\Bigl(V_{\theta_\psi}\bigl(\varphi\varphi_A^\ell\bigr)(\ell)\!\Bigr)\simeq\mathbf{D}_{\mathrm{dR}}^{\ell+1}(V_{\chi}\otimes_{\mathcal{E}_\chi}\mathcal{E}_{\psi,\vartheta})\oplus\mathbf{D}_\dR^{-\ell}\bigl(V_{\psi^*\varphi\varphi_A^\ell\chi_\mathrm{cyc}}(\ell)\otimes_{\mathcal{E}_\chi}\mathcal{E}_{\psi,\vartheta}\bigr). \]
Then, applying the projector $e_K$, we obtain a map of $1$-dimensional $\mathcal{E}_{\psi,\vartheta}$-vector spaces  
\begin{equation} \label{isoDR}
\mathbf{D}_\dR\Bigl(e_KV_{\theta_\psi}\bigl(\varphi\varphi_A^\ell\bigr)\!\Bigr)=\mathbf{D}_\dR^{\ell+1}\Bigl(e_KV_{\theta_\psi}\bigl(\varphi\varphi_A^\ell\bigr)\!\Bigr)\simeq\mathbf{D}_{\mathrm{dR}}^{\ell+1}(V_{\chi}\otimes_{\mathcal{E}_\chi}\mathcal{E}_{\psi,\vartheta}).
\end{equation}
Now we define the dotted map by requiring the diagram of $\mathcal{E}_{\psi,\varphi,\varphi_A^\ell}$-vector spaces 
\begin{equation} \label{isoDR1}
\xymatrix@C=40pt@R=35pt{
\mathbf{D}_\dR^{r+1}\bigl(V_Z(\varphi)\bigr)\otimes_{\mathcal{E}_\varphi}\mathcal{E}_{\psi,\varphi,\varphi_A^\ell}\ar@{-->}[dd]^-{\pi_\chi^\dR}\ar[r]^-{\eqref{galmap}}& 
\mathbf{D}_\dR^{r+1}\bigl(V_W(\varphi\varphi_A^\ell)(-\ell)\bigr)\otimes_{\mathcal{E}_{\varphi\varphi_A^\ell}}\mathcal{E}_{\psi,\varphi,\varphi_A^\ell}\ar[d]^-{\eqref{isogalpsi}}
\\
&\mathbf{D}_\dR^{r+1}\bigl(V_{\theta_\psi}(\varphi\varphi_A^\ell)(-\ell)\bigr)\otimes_{\mathcal{E}_{\psi,\vartheta}}\mathcal{E}_{\psi,\varphi,\varphi_A^\ell}\ar[d]^-\simeq \\
\mathbf{D}_\dR^{\ell+1}(V_\chi)\otimes_{\mathcal{E}_{\chi}}\mathcal{E}_{\psi,\varphi,\varphi_A^\ell}&\mathbf{D}_\dR^{\ell+1}\bigl(V_{\theta_\psi}(\varphi\varphi_A^\ell)\bigr)\otimes_{\mathcal{E}_{\psi,\vartheta}}\mathcal{E}_{\psi,\varphi,\varphi_A^\ell}
\ar[l]_-{\eqref{isoDR}}}
\end{equation}
to commute. Recall the comparison isomorphism 
\[ H^1_\dR(\chi)\otimes_{E_\chi}\mathcal{E}_\chi\simeq \mathbf{D}_\dR(V_\chi). \]
Every generator $\bomega_\chi$ of $H^1_\dR(\chi)$ as a $K\otimes_\Q E_\chi$-module gives a generator, denoted by ${\omega}_\chi$, of the $1$-dimensional $\mathcal{E}_\chi$-vector space $\mathbf{D}_\dR(V_\chi)$. For any field extension $\mathcal{F}$ of $\Q_p$, we may then consider the element ${\omega}_\chi\otimes 1$ in the $1$-dimensional $\mathcal{F}$-vector space 
$\mathbf{D}_{\dR}(V_\chi)\otimes\mathcal{F}\simeq \mathbf{D}_{\dR,\mathcal{F}}(V_\chi)$. 

The comparison isomorphism of $\mathcal{E}_\varphi$-vector spaces 
\[ \epsilon_ZH^{2r+1}_\dR(Z_\varphi/\Q_p)\otimes_{E_\varphi}\mathcal{E}_\varphi\simeq\mathbf{D}_\dR\bigl(V_Z(\varphi)\bigr) \] 
combined with the map $\pi^\dR_\chi$ in \eqref{isoDR1} induces a map of of $\mathcal{E}_{\psi,\varphi,\varphi_A^\ell}$-vector spaces 
\[ \pi^\dR_\chi:\mathrm{Fil}^{r+1}\Bigl(\epsilon_ZH^{2r+1}_\dR(Z_\varphi/\Q_p)\otimes_{E_{\varphi}}\mathcal{E}_{\varphi,\varphi,\varphi_A^\ell}\Bigr)\longepi\mathbf{D}_\dR^{\ell+1}(V_\chi)\otimes_{\mathcal{E}_{\chi}}\mathcal{E}_{\psi,\varphi,\varphi_A^\ell}, \]
still denoted by the same symbol. Then we may consider the element $\pi_\chi^\dR\bigl(\omega_{\theta_\psi}\otimes\delta_\psi\otimes\omega_B^\ell\eta_B^\ell\bigr)$. For any field extension $\mathcal{F}/\Q_p$, we still denote by $\pi_\chi^\dR\bigl(\omega_{\theta_\psi}\otimes\delta_\psi\otimes\omega_B^\ell\eta_B^\ell\bigr)$ the image of this class in the $\mathcal{F}\otimes\mathcal{E}_{\psi,\varphi,\varphi_A^\ell}$-module 
\[ \bigl(\mathbf{D}_\dR^{\ell+1}(V_\chi)\otimes\mathcal{F}\bigr)\otimes_{\mathcal{E}_{\chi}}\mathcal{E}_{\psi,\varphi,\varphi_A^\ell}=\mathbf{D}_{\dR,\mathcal{F}}^{\ell+1}(V_\chi)\otimes_{\mathcal{E}_{\chi}}\mathcal{E}_{\psi,\varphi,\varphi_A^\ell}. \]
Therefore, $\pi_\chi^\dR\bigl(\omega_{\theta_\psi}\otimes\delta_\psi\otimes\omega_B^\ell\eta_B^\ell\bigr)=\lambda\cdot({\omega}_\chi\otimes1)$ for some $\lambda\in \mathcal{F}\otimes\mathcal{E}_{\psi,\varphi,\varphi_A^\ell}^\times$. 

Similarly as in the case of \'etale realizations, let $\mathcal{F}_\varphi$ be the completion of the field $H_\varphi$ cut out over $H$ by $\varphi$. The restriction of the character $\varphi$ to $G_{\mathcal{F}_\varphi}$ is trivial, while the restriction of $\varphi_A$ takes values in $\Q_p^\times$. As before, we may then define the dotted map (using again the same symbol) by requiring that the diagram of $\mathcal{E}_{\psi,\vartheta}$-vector spaces 
\begin{equation}\label{isoDR2}
\xymatrix@C=40pt@R=35pt{
\mathbf{D}_{\dR,\mathcal{F}_\varphi}^{r+1}(V_Z)\otimes\mathcal{E}_{\psi,\vartheta}\ar@{-->}[dd]^-{\pi_\chi^\dR}\ar[r]^-{\eqref{galmap}}& 
\mathbf{D}_{\dR,\mathcal{F}_\varphi}^{r+1}\bigl(V_W(\varphi_A^\ell)(-\ell)\bigr)\otimes\mathcal{E}_{\psi,\vartheta}\ar[d]^-{\eqref{isogalpsi}}
\\
&\mathbf{D}_{\dR,\mathcal{F}_\varphi}^{r+1}\bigl(V_{\theta_\psi}(\varphi_A^\ell)(-\ell)\bigl)\ar[d]^-\simeq \\
\mathbf{D}_{\dR,\mathcal{F}_\varphi}^{\ell+1}(V_\chi)\otimes_{\mathcal{E}_{\chi}}\mathcal{E}_{\psi,\vartheta}&\mathbf{D}_\dR^{\ell+1}\bigl(V_{\theta_\psi}(\varphi_A^\ell)\bigr)
\ar[l]_-{\eqref{isoDR}}}
\end{equation} 
be commutative.

\begin{remark} 
As observed in Remark \ref{remcoeff}, the coefficient field $\mathcal{E}_{\psi,\vartheta}$ in diagram \eqref{isoDR2} cannot be improved upon in general. 
\end{remark}

Again, as before, from the comparison theorem and diagram \eqref{isoDR2} we obtain a map 
%\begin{equation}\label{eqomega1}
\[ \pi^\dR_\chi:\mathrm{Fil}^{r+1}\Bigl(\epsilon_ZH^{2r+1}_\dR(Z_r/\mathcal{F}_\varphi)\otimes\mathcal{E}_{\varphi,\vartheta}\Bigr)\longepi\mathbf{D}_{\dR,\mathcal{F}_\varphi}^{\ell+1}(V_\chi)\otimes_{\mathcal{E}_{\chi}}\mathcal{E}_{\psi,\vartheta}. \]
%\end{equation}
Now let $\mathcal{F}_{\psi,\vartheta}$ stand for the composite field of $\mathcal{E}_{\psi,\vartheta}$ and $\mathcal{F}_\varphi$. 
There is an isomorphism of $1$-dimensional $\mathcal{F}_{\psi,\vartheta}$-vector spaces 
\begin{equation} \label{inv01}
\mathbf{D}_{\dR,\mathcal{F}_{\psi,\vartheta}}^{\ell+1}(V_\chi)\overset\simeq\longrightarrow 
\mathbf{D}_{\dR,\mathcal{F}_{\psi,\vartheta}}^{\ell+1}(V_\chi\otimes_{\mathcal{E}_\chi}\mathcal{E}_{\psi,\vartheta})^{\mathcal{E}_{\psi,\vartheta}}
\end{equation}
(see the results in \S \ref{appendixA2}). Associated with the chosen embedding $\bar\Q\hookrightarrow\bar\Q_p$ is a comparison isomorphism of $1$-dimensional $\mathcal{E}_\chi$-vector spaces 
\[ \mathrm{comp}:H^1_\dR(\chi)\otimes_K\Q_p\overset\simeq\longrightarrow \D_{\dR}^{\ell+1}(V_\chi). \]
We may then consider the injective $\mathcal{E}_\chi$-linear map $\rho_\dR$ given by the composition
\begin{equation} \label{classcoh}
\begin{split}
H^1_\dR(\chi)\otimes_K\Q_p\xrightarrow[\simeq]{\mathrm{comp}}\D_{\dR}^{\ell+1}(V_\chi)&\longmono \D_{\dR}^{\ell+1}(V_\chi)\otimes_{\mathcal{E}_\chi}\mathcal{F}_{\psi,\vartheta}\\&\overset\simeq\longrightarrow\D_{\dR,\mathcal{F}_{\psi,\vartheta}}^{\ell+1}(V_\chi)\xrightarrow[\simeq]{\eqref{inv01}}\mathbf{D}_{\dR,\mathcal{F}_{\psi,\vartheta}}^{\ell+1}(V_\chi\otimes_{\mathcal{E}_\chi}\mathcal{E}_{\psi,\vartheta})^{\mathcal{E}_{\psi,\vartheta}},
\end{split}
\end{equation}
where the second arrow is the canonical injection $x\mapsto x\otimes1$ and all the other maps are isomorphisms. Recall the generator $\bomega_\chi$ of the $K\otimes_\Q E_\chi$-module $H^1_\dR(\chi)$. Therefore, in the $1$-dimensional $\mathcal{F}_{\psi,\vartheta}$-vector space on the right in \eqref{inv01} we have the class 
\begin{equation} \label{omegachi}
\omega_{\chi}\defeq\rho_\dR(\bomega_\chi)\end{equation} 
obtained from $\bomega_\chi$ via the map from \eqref{classcoh}, 
and the class 
\begin{equation}\label{omegapsi}
\omega_{\psi,\vartheta}=\pi_\chi^\dR\bigl(\omega_\psi\otimes\omega_B^\ell\eta_B^\ell\bigr)
\end{equation}
obtained from $\omega_\psi\otimes\omega_B^\ell\eta_B^\ell$ using the map in $\pi^\dR_\chi$ in 
\eqref{isoDR2}.  
The next lemma compares the two. Before that, we recall the notion of generalized Gauss sum. 

For a finite order Hecke character $\varphi:\A_K^\times\rightarrow E^\times$, let $H_\varphi$ be the finite extension of $K$ cut out by $\varphi$. The $E$-vector space 
\[ E\{\varphi\}\defeq\bigl\{x\in E H_\varphi\mid\text{$x^g=\varphi(g)x$ for all $g\in \Gal(E H_\varphi/E)$}\bigr\} \] 
is $1$-dimensional; we denote by $\mathfrak{g}(\varphi)$ a generator of it. Then $\mathfrak{g}(\varphi)$ belongs to $E\{\varphi\}\cap (E H_\varphi)^\times$ and is well defined up to elements of $E^\times$ (see, \emph{e.g.}, \cite[Chapter 2, \S3]{S86} or \cite[\S2B]{BDP2} for a more complete description of $\mathfrak{g}(\varphi)$). We call $\mathfrak{g}(\varphi)$ a \emph{generalized Gauss sum} for $\varphi$. 
  
\begin{lemma} \label{gaussum}
$\omega_{\chi}\equiv\mathfrak{g}(\varphi)\cdot\omega_{\psi,\vartheta}\pmod{E_{\psi,\vartheta}^\times}$.
\end{lemma}

\begin{proof} Set $\mathcal{G}\defeq\Gal(\mathcal{F}_{\psi,\vartheta}/\mathcal{F}_{\varphi})$. There is a semilinear action of $\mathcal G$ on the de Rham modules $\mathbf{D}_{\dR,\mathcal{F}_{\psi,\vartheta}}(V_\chi)$ and $\mathbf{D}_{\dR,\mathcal{F}_{\psi,\vartheta}}\bigl(V_{\theta_\psi}(\varphi_A^\ell)\bigr)$, and these two actions are related as follows: for any $\sigma\in\Gal(\mathcal{F}_{\psi,\vartheta}/\mathcal{F}_{\varphi})$, if $\omega\star\sigma$ denotes the action on $\mathbf{D}_{\dR,\mathcal F_{\psi,\vartheta}}(V_{\chi})$ and $\omega\star_\varphi\sigma$ the action on $\mathbf{D}_{\dR,\mathcal F_{\psi,\vartheta}}\bigl(V_{\theta_\psi} (\varphi_A^\ell)\bigr)$, then 
\begin{equation} \label{eq01}
\omega\star \sigma=[\varphi(\sigma)]\omega\star_\varphi\sigma.
\end{equation}
Let $\mu \in (H_{\varphi} E_{\psi,\vartheta} )^\times $ be the scalar in \eqref{scalar} satisfying $\omega_{\psi,\vartheta}=\mu \omega_{\chi}$. The equality $\omega_{\psi,\vartheta}\star_\varphi\sigma=\omega_{\psi,\vartheta}$ holds for all $\sigma\in \mathcal{G}$ because $\omega_{\psi,\vartheta}$ is defined over $\mathcal{F}_\varphi$. On the other hand, we have 
\begin{equation} \label{eq02}
\omega_\psi\star_\varphi\sigma=(\mu \omega_{\psi,\vartheta})\star_\varphi\sigma=\mu ^\sigma(\omega_{\psi,\vartheta}\star_\varphi\sigma)=\mu ^\sigma[\varphi^{-1}(\sigma)]\omega_{\psi,\vartheta}\star \sigma=\mu ^\sigma\varphi^{-1}(\sigma)\omega_{\psi,\vartheta}\star\sigma,
\end{equation}
where the second equality follows from the semilinearity of the Galois action, the third from \eqref{eq01} and the fourth because $\omega_{\psi,\vartheta}$ is $\mathcal{E}_{\psi,\vartheta}$-invariant in the sense of \eqref{inv01}. Furthermore, since $\omega_{\psi,\vartheta}$ too is defined over $\mathcal{F}_\varphi$, we have also $\omega_{\psi,\vartheta}\star \sigma=\omega_{\psi,\vartheta}$. Therefore, replacing in the last equation at the end of \eqref{eq02}, we obtain $\omega\star_\varphi\sigma=\mu ^\sigma(\varphi)^{-1}(\sigma)\omega_{\psi,\vartheta}$. Combining this equality with the previous relations, we get
 \[\mu \omega_{\psi,\vartheta}=\omega_\psi=\omega_\psi\star_\varphi\sigma=\mu ^\sigma\varphi^{-1}(\sigma)\omega_{\psi,\vartheta},\] 
whence $\mu=\mu ^\sigma\varphi^{-1}(\sigma)$. Since $\mu\in H_\varphi E_{\psi,\vartheta}$, we conclude that $\mu =\mathfrak{g}(\varphi)$. \end{proof}  

\subsection{\'Etale Abel--Jacobi maps}\label{secAJmap} 

Let $\mathcal{F}$ be a finite extension of $\Q_p$ containing the image of $H_\varphi$ under the fixed embedding $\bar\Q\hookrightarrow\bar\Q_p$. The \'etale Abel--Jacobi map for $X_r$ sits in the commutative diagram  
{\small{\begin{equation}\label{D}
\xymatrix{
\CH^{r+1}_0(X_r/\mathcal{F})({\mathcal{F}})\otimes_{\Q}\mathcal{E}_{\chi}
\ar[d]^-{\Phi_{\et,X}^{({\mathcal{F}})}}\ar@/_{6.2pc}/[ddd]^{\log_{X,\mathcal{F}}}
\\
H^1_f\bigl(\mathcal F,V_X(r+1)\otimes \mathcal{E}_{\chi}\bigr)\ar[r]^-{\eqref{galmap}}\ar[d]^-{\mathrm{comp}}& 
H^1_f\bigl(\mathcal{F},V_{\psi^\ast(\chi_A\chi_A^\ast)^\ell}(1)\bigr)\ar[r]^-{e_K^\ast}\ar[d] & 
H^1_f\bigl(\mathcal F,V_{\chi^*}(1)\bigr)\ar[d]\ar@/^{4.5pc}/[dd]_{\log_{\chi}}
\\
\frac{\D_{\dR,\mathcal F}(V_X(r+1)\otimes\mathcal{E}_{\chi})}{\D^{0}_{\dR,\mathcal F}(V_X(r+1)\otimes\mathcal{E}_{\chi})}\ar[r]^-{\eqref{galmap}}\ar[d]_-\simeq^-{\delta_{X}}
&\frac{\D_{\dR,\mathcal F}(V_{\theta_{\psi^*}(\chi_A\chi_A^*)^\ell}(1))}{\D^{0}_{\dR,\mathcal F}(V_{\theta_{\psi^*}(\chi_A\chi_A^*)^\ell}(1))}\ar[r]^-{e_K^*}\ar[d]_-\simeq^-{\delta_{\psi^*}} &\frac{\D_{\dR,\mathcal F}(V_{\chi^*}(1))}{\D^{0}_{\dR,F}(V_{\chi^*}(1))}\ar[d]_-\simeq^-{\delta_{\chi^*}}
\\
\Bigl(\mathbf{D}_{\dR,\mathcal F}^1\bigl(V_X(r)\otimes\mathcal{E}_{\chi}\bigr)\!\Bigr)^\vee\ar[r]^-{\lambda_X}
& \Bigl(\mathbf{D}_{\dR,\mathcal F}^1\bigl(V_{\theta_{\psi^*}(\chi_A\chi_A^*)^\ell}\bigr)\!\Bigr)^\vee\ar[r]^-{\lambda_{\psi^*}}
& \Bigl(\D^1_{\dR,\mathcal F}\bigl(V_{\chi^*}(1)^\vee\bigr)\!\Bigr)^\vee
}\end{equation}}}
\!\!in which the bottom horizontal lines $\lambda_X$ and $\lambda_{\psi^\ast}$ are defined by means of the isomorphisms $\delta_{X}$, $\delta_{\psi^\ast}$, $\delta_{\chi^*}$ from $p$-adic Hodge theory and the indicated middle horizontal arrows. Notice that $\chi^\ast=\psi^\ast(\chi_A\chi_A^\ast)^\ast$ as $G_{\mathcal{F}}$-representations and that $\mathcal{E}_\chi=\mathcal{E}_\psi=\mathcal{E}_{\psi,\vartheta}$ because $\varphi$ is trivial and $\chi_A$ takes values in $\Q_p$.  

\begin{definition} \label{defSD}
A Hecke character $\chi\in \Sigma_\mathrm{crit}(\mathfrak{c})$ is \emph{self-dual} if $\chi \chi^\ast = \Norm_K$. 
\end{definition}

Let us assume that $\chi$ is self-dual; then $\chi^\ast\chi_\cyc=\chi^{-1}$, so $V_{\chi^\ast}(1)^\vee\simeq V_\chi$. Thus, the bottom right corner of diagram \eqref{D} can be replaced by the dual of $\mathbf{D}_{\dR,\mathcal{F}}(V_\chi)$. 

Now we describe the two maps $\lambda_X$ and $\lambda_{\psi^\ast}$ in greater details. Hecke theory implies that $V_{\theta_\psi(\chi_A\chi_A^\ast)^\ell}$ is a direct factor of $V_X(r)\otimes\mathcal{E}_\psi$; we denote by 
\begin{equation} \label{iota-theta-eq}
\iota_{\theta_\psi}:V_{\theta_\psi(\chi_A\chi_A^\ast)^\ell}\longmono  V_X(r)\otimes\mathcal{E}_\psi 
\end{equation}
a section of the projection map \eqref{galmap}. It follows that $\D_{\dR,\mathcal{F}}\bigl(V_{\theta_\psi(\chi_A\chi_A^\ast)^\ell}\bigr)$ is a direct factor of $\D_{\dR,\mathcal{F}}\bigl(V_X(r)\otimes \mathcal{E}_\psi\bigr)$ and, with a slight abuse of notation, we write
\begin{equation} \label{iota-theta-eq2}
\iota_{\theta_\psi}\colon \D_{\dR,\mathcal{F}}\bigl(V_{\theta_\psi(\chi_A\chi_A^\ast)^\ell}\bigr)\longmono\mathbf{D}_{\dR,F}\bigl(V_X(r)\otimes\mathcal{E}_\chi\bigr)
\end{equation}
also for the map induced by the injection $\iota_{\theta_\psi}$ from \eqref{iota-theta-eq}. Observe that there are splittings
\[ \begin{split}
    \mathbf{D}_{\dR,\mathcal{F}}\bigl(V_{\theta_\psi(\chi_A\chi_A^\ast)^\ell}\bigr)&=
\mathbf{D}_{\dR,\mathcal{F}}(V_{\theta_\psi})\otimes\D_{\dR,\mathcal{F}}\bigl(V_{(\chi_A\chi_A^\ast)^\ell}\bigr)\\
  &\simeq\Bigl(\mathbf{D}_{\dR,\mathcal{F}}^{k-1}(V_{\theta_\psi})\otimes\D_{\dR,F}^{-\ell}\bigl(V_{(\chi_A\chi_A^\ast)^\ell}\bigr)\!\Bigr)\oplus\Bigl(\mathbf{D}_{\dR,\mathcal{F}}^0(V_{\theta_\psi})\otimes\D_{\dR,\mathcal{F}}^{-\ell}(V_{(\chi_A\chi_A^\ast)^\ell})\!\Bigr)\\
  &
\simeq \mathbf{D}_\dR^{\ell+1}\bigl(V_{ \psi(\chi_A\chi_A^*)^\ell}\bigr)\oplus \mathbf{D}_{\dR,\mathcal{F}}^{-\ell}\bigl(V_{ \psi^\ast(\chi_A\chi_A^\ast)^\ell}\bigr)\\
&\simeq\mathbf{D}_\dR^{\ell+1}(V_{ \chi})\oplus \mathbf{D}_{\dR,\mathcal{F}}^{-\ell}(V_{\chi^*})
   \end{split} \]
and that taking duals interchanges the two factors in the last line; taking filtrations, we thus conclude that 
\[ \begin{split}
   \mathbf{D}^1_{\dR,\mathcal{F}}\bigl(V_{\theta_\psi(\chi_A\chi_A^*)^\ell}\bigr)^\vee&=\mathbf{D}_{\dR,\mathcal{F}}^{-\ell}(V_{\chi^*})\simeq\D^{\ell+1}_{\dR,\mathcal{F}}(V_\chi)^\vee=\D_{\dR,\mathcal{F}}(V_\chi)^\vee.
 \end{split} \] 
Composing this isomorphism with the dual of the map $\iota_{\theta_\psi}$ from \eqref{iota-theta-eq2}, we obtain a map 
\[ \mathbf{D}_{\dR,\mathcal F}^1\bigl(V_X(r)\otimes\mathcal{E}_{\chi}\bigr)^\vee\longrightarrow\D_{\dR,\mathcal{F}}(V_\chi)^\vee \] 
that, under the previous identifications, coincides with the projection to the direct factor $\D_{\dR,\mathcal{F}}^{-\ell}(V_{\chi^*})$. 
On the other hand, the composition of \eqref{galmap} and $e_K^*$ can also be described as the projection to $\mathbf{D}_{\dR,\mathcal{F}}^{-\ell}(V_{\chi^*})$; thus, we conclude that the composition $\lambda_{\psi^*}\circ\lambda_X$ is the dual of the injection 
\[ \iota_\chi:\D_{\dR,\mathcal{F}}(V_\chi)\overset{\tilde e_K}\longmono \D_{\dR,\mathcal{F}}\bigl(V_{\theta_\psi(\chi_A\chi_A^*)^\ell}\bigr)\xrightarrow{i_{\theta_\psi}}\D_{\dR,\mathcal{F}}\bigl(V_X(r)\otimes \mathcal{E}_\chi\bigr), \]
where $\tilde e_K$ is induced by the section of the projection $V_{\theta_\psi(\chi_A\chi_A^*)^\ell}\twoheadrightarrow V_\chi$ induced by $e_K$. 
 
Let us consider now the following commutative diagram, in which the pairings are the canonical ones and 
the map $e_\chi$ is induced by the composition of $e_K$ with the map in $\eqref{galmap}$:
\[ \xymatrix@R=35pt@C=-2pt{\mathbf{D}_{\dR,\mathcal{F}}^1\bigl(V_X(r)\otimes\mathcal{E}_\chi\bigr)\ar[d]^-{e_\chi}&\times&\mathbf{D}_{\dR,\mathcal{F}}^1\bigl(V_X(r)\otimes\mathcal{E}_\chi\bigr)^\vee\ar[d]^-{\iota_\chi^\vee}
\ar[rrrrrr]&&&&&& \mathcal{F}\otimes \mathcal{E}_\chi\ar@{=}[d]\\
\D^1_{\dR,\mathcal{F}}(V_{\chi})\ar@/^{1pc}/[u]^-{\iota_\chi}&\times & \D^1_{\dR,\mathcal{F}}(V_{\chi})^\vee
\ar[rrrrrr]&&&&&& \mathcal{F}\otimes \mathcal{E}_\chi.} \]
By examining the relevant filtration, we see that 
\[ \omega_{\theta_\psi}\otimes\omega_A^\ell\eta_A^\ell=(e_\chi\circ \iota_\chi)\bigl(\omega_{\theta_\psi}\otimes\omega_A^\ell\eta_A^\ell\bigr). \]
Thus, we conclude that for all $\xi\in\D^1_{\dR,\mathcal{F}}(V_{\chi})^\vee$ there is an equality 
\begin{equation} \label{accpair}
\xi\bigl(\omega_{\theta_\psi}\otimes\omega_A^\ell\eta_A^\ell\bigr)=\iota^\vee_\chi(\xi)\Bigl(e_\chi\bigl(\omega_{\theta_\psi}\otimes\omega_A^\ell\eta_A^\ell\bigr)\!\Bigr).
\end{equation}
In particular, let us fix $\Delta\in \CH^{r+1}_0(X_r/\mathcal{F})({\mathcal{F}})\otimes_{\Q}\mathcal{E}_{\chi}$ and denote by $\Delta_{\chi}$ the image of $\Delta$ in $H^1_f\bigl(\mathcal F,V_{\chi^*}(1)\bigr)$ via the composition of the maps $\Phi^{(F)}_{\et,X}$, \eqref{galmap} and $e_K^*$ in diagram \eqref{D}. Then it follows from \eqref{accpair} and \eqref{D} that 
\begin{equation} \label{accpair2}
\log_{X,\mathcal{F}}(\Delta)\bigl(\omega_{\theta_\psi}\otimes\omega_A^\ell\eta_A^\ell\bigr)=\log_{\chi}(\Delta_{\chi})\Bigl(e_\chi\bigl(\omega_{\theta_\psi}\otimes\omega_A^\ell\eta_A^\ell\bigr)\!\Bigr).
\end{equation}
This equality will play a key role in the proof of our main result in \S \ref{sec:GRFII}.

\section{$p$-adic $L$-functions} \label{BDPsec}

Given a Hecke character $\chi$ of an imaginary quadratic field $K$, we denote by $L(\chi,s)$ the Hecke $L$-function of $\chi$; see, \emph{e.g.}, \cite[Chapter 0, \S6]{S86} for a description of the Euler factors, keeping in mind that we fix once and for all an embedding $E_\chi\hookrightarrow \C$, so that $L(\chi,s)$ is a $\C$-valued function. An integer $n$ is said to be \emph{critical} for $L(\chi,s)$ if the Gamma factors on both sides of the functional equation of $L(\chi,s)$ (for which, see for example \cite[Chapter 0, \S6]{S86}) do not have poles at $n$. 

\subsection{Critical region} \label{sec:criticalregion} 

Let $\chi$ be a Hecke character of $K$. We say that $\chi$ is \emph{critical} if $s=0$ is a critical point for $L(\chi,s)$. 
Fix an integral ideal $\mathfrak{c}$ of $K$. Denote by $\Sigma(\mathfrak{c})$ the set of all Hecke characters of $K$ of conductor dividing $\mathfrak{c}$ and let $\Sigma_\mathrm{crit}(\mathfrak{c})$ be the subset of $\Sigma(\mathfrak{c})$ consisting of critical characters. 
One can write $\Sigma_{\mathrm{crit}}(\mathfrak{c})$ as a disjoint union 
\[ \Sigma_{\mathrm{crit}}(\mathfrak{c} )= \Sigma_{\mathrm{crit}}^{(1)}(\mathfrak{c} )\amalg \Sigma_{\mathrm{crit}}^{(2)}(\mathfrak{c}) \] 
where
\begin{itemize}
\item $\Sigma_{\mathrm{crit}}^{(1)}(\mathfrak{c})$ is the subset of characters of infinity type $(\ell_1,\ell_2)$ with $\ell_1\leq 0$ and $\ell_2\geq 1$;
\item $\Sigma_{\mathrm{crit}}^{(2)}(\mathfrak{c})$ is the subset of characters of infinity type $(\ell_1,\ell_2)$ with $\ell_1\geq 1$ and $\ell_2\leq 0$. 
\end{itemize}
If $\mathfrak{c}=\bar{\mathfrak{c}}$, then $\Sigma_{\mathrm{crit}}^{(1)}(\mathfrak{c} )$ and $\Sigma_{\mathrm{crit}}^{(2)}(\mathfrak{c})$ are interchanged by $\nu \mapsto \nu^\ast$.

The set $\Sigma_{\mathrm{crit}}(\mathfrak{c})$ is endowed with a natural $p$-adic compact open topology described in \cite[Section 5.2]{BDP1}. We write $\widehat{\Sigma}_{\mathrm{crit}}(\mathfrak{c})$ for the corresponding completion of $\Sigma_{\mathrm{crit}}(\mathfrak{c})$, in which both $\Sigma_{\mathrm{crit}}^{(1)}(\mathfrak{c})$ and $\Sigma_{\mathrm{crit}}^{(2)}(\mathfrak{c})$ are dense. 

\subsection{Complex periods of Hecke characters}\label{persec}

As in \S \ref{remell}, fix an elliptic curve $A$ with CM by $\cO_K$, defined over the Hilbert class field $H$ of $K$. Choose also a regular differential $\omega_A\in \Omega^1(A/H)$ and a cycle $\gamma\in H_1\bigl(A(\C),\Q\bigr)$. Set
\begin{equation} \label{period1}
\Omega_\infty(\chi_A)\defeq\frac{1}{2\pi i}\int_\gamma\omega_A.
\end{equation}
Equality \eqref{period1} fixes periods of Hecke characters, as follows. Let $\chi$ be a Hecke character of infinity type $(\ell_1,\ell_2)$ and consider the standard factorization $\chi=\varphi\chi_A^{-\ell_1}(\chi_A^\ast)^{-\ell_2}$ introduced in \S\ref{remell}. Then define 
\begin{equation} \label{period3}
\Omega_\infty(\chi)\defeq\mathfrak{g}(\varphi)\cdot(2\pi i)^{\ell_2}\cdot\Omega_\infty(\chi_A)^{\ell_1-\ell_2}.
\end{equation}
It follows that
\[ \frac{L(\chi^{-1},0)}{\Omega_\infty(\chi^\ast)}\in E_\chi \] 
for all integral ideals $\mathfrak{c}$ of $K$ and all $\chi\in\Sigma_\mathrm{crit}(\mathfrak{c})$ with $\ell_1>\ell_2$ (see, \emph{e.g.}, \cite[Theorem 2.12]{BDP2} or \cite[Chapter 2, Theorem 2.1]{S86}).

\subsection{$p$-adic periods of Hecke characters}
Let $p$ be a prime number. We view $H$ as a subfield of $\bar\Q_p$ via our fixed embedding $\bar\Q\hookrightarrow\bar\Q_p$. Assume that $p$ splits in $K$ and that the base change $A_{\C_p}$ of $A$ to $\C_p$ via $\bar\Q\hookrightarrow\bar\Q_p$
has good (ordinary) reduction at the maximal ideal of $\cO_{\C_p}$. Choose an isomorphism $\varphi_p:\hat{A}\overset\simeq\longrightarrow\hat{\mathds{G}}_m$ over $\C_p$ between the formal group $\hat{A}$ of $A$ and the formal multiplicative group $\hat{\mathds{G}}_m$, then define a period $\Omega_p(A)\in\C_p^\times$ by the equality $\omega_A=\Omega_p(A)\cdot \varphi_p^*(dt/t)$, where $dt/t$ is the canonical differential on $\hat{\mathds{G}}_m$. 
This choice again fixes $p$-adic periods of all Hecke characters. Indeed, let $\chi$ be a Hecke character of infinity type $(\ell_1,\ell_2)$; as before, consider the factorization $\chi=\varphi\chi_A^{-\ell_1}(\chi_A^*)^{-\ell_2}$ for a finite order character $\varphi$, then define 
\begin{equation} \label{period4}
\Omega_p(\chi)\defeq\mathfrak{g}(\varphi)\cdot\Omega_p(A)^{\ell_1-\ell_2}\in\C_p^\times.
\end{equation} 
%In the next proposition, we record an immediate consequence of this definition.

%\begin{proposition} \label{lemmaperiods1}
%For every character $\chi$ and every $\ell\in\Z$, one has $\Omega_p(\chi)=\Omega_p\bigl(\chi(\chi_A\chi_A^*)^\ell\bigr)$. 
%\end{proposition}

%\begin{proof} Writing, as above, $\chi=\varphi\chi_A^{-\ell_1}(\chi_A^*)^{-\ell_2}$, one has $\chi(\chi_A\chi_A^*)^\ell=\varphi\chi_A^{\ell-\ell_1}(\chi_A^*)^{\ell-\ell_2}$ for the same finite order character $\varphi$, so we get $\Omega_p\bigl(\chi(\chi_A\chi_A^*)^\ell\bigr)=\mathfrak{g}(\varphi)\cdot\Omega_p(A)^{\ell_1-\ell_2}=\Omega_p(\chi)$. \end{proof}

\subsection{Katz's $p$-adic $L$-function} \label{katz-subsec}

Let us fix an ideal $\mathfrak{c}$ of $\cO_K$ and recall the $p$-adic completion $\widehat{\Sigma}_\mathrm{crit}(\mathfrak{c})$ introduced in \S\ref{sec:criticalregion}. Assume from now on that $p$ splits in $K$. In \cite{Katz}, Katz proved the existence of a $p$-adic analytic function 
\[ \mathscr{L}_{p,\mathfrak{c}} \colon \widehat{\Sigma}_{\text{crit}}(\mathfrak{c} ) \longrightarrow \C_p \] 
that is characterized by the following interpolation property (see, \emph{e.g.}, \cite[Theorem 3.1]{BDP2}).

\begin{theorem}[Katz] 
If $\chi \in \Sigma_\mathrm{crit}^{(2)}(\mathfrak{c})$ has infinity type $(\ell_1,\ell_2)$, then 
\begin{equation}\label{Katz}
    \frac{\mathscr{L}_{p,\mathfrak{c}}(\chi)}{\Omega_p(A)^{\ell_1-\ell_2}}=
    \biggl( \frac{\sqrt{D_K}}{2\pi}\biggr)^{\ell_2}\cdot(\ell_1-1)!\cdot\bigl(1-\chi(\mathfrak{p})/p\bigr)\cdot\bigl(1-\chi^{-1}(\bar{\mathfrak{p}})\bigr)\cdot\frac{ L_{\mathfrak{c}}(\chi^{-1},0)}{\Omega_\infty(A)^{\ell_1-\ell_2}}.
\end{equation}
\end{theorem} 

\begin{proof} See, \emph{e.g.}, \cite[Chapter II, Theorem 4.14]{de87}. \end{proof}

Here $L_\mathfrak{c}(\chi^{-1},s)$ is the Hecke $L$-function of $\chi^{-1}$ with the Euler factors at primes dividing $\mathfrak{c}$ removed; in particular, if $\chi$ has conductor $\mathfrak{c}$, then $L(\chi^{-1},s)=L_\mathfrak{c}(\chi^{-1},s)$. Formula \eqref{Katz} should be interpreted as asserting that the right-hand side belongs to $\bar\Q$ and its image in $\bar\Q_p$ via the fixed embedding $\bar\Q\hookrightarrow\bar\Q_p$ is equal to the left-hand side.  
  
Recall from Definition \ref{defSD} that a Hecke character $\chi\in \Sigma_\mathrm{crit}(\mathfrak{c})$ is self-dual if $\chi \chi^\ast = \Norm_K.$ Denote by $\Sigma_\mathrm{sd}(\mathfrak{c})$ the subset of $\Sigma_\mathrm{crit}(\mathfrak{c})$ consisting of self-dual characters of conductor $\mathfrak{c}$. From the definition, a self-dual Hecke character is necessarily of infinity type $(1+j,-j)$ for some $j\in \Z$. Moreover, the conductor of a self-dual Hecke character is invariant under complex conjugation.
Define $\Sigma_\mathrm{sd}^{(i)}(\mathfrak{c} )\colon \!\! =\Sigma_\mathrm{sd}(\mathfrak{c} )\cap \Sigma_\mathrm{crit}^{(i)}(\mathfrak{c} )$ for $i\in\{1,2\}$, so that 
\begin{itemize}
\item $\Sigma_\mathrm{sd}^{(1)}(\mathfrak{c} )$ is the set of self-dual characters of infinity type $(1+j,-j)$ 
with $j\leq -1$;  
\item $\Sigma_\mathrm{sd}^{(2)}(\mathfrak{c} )$ is the set of self-dual characters of infinity type $(1+j,-j)$ 
with $j\geq 0$. 
\end{itemize}

\begin{remark} If $\chi\in \Sigma_\mathrm{sd}^{(2)}(\mathfrak{c})$, then 
\begin{equation} \label{katz-cong}
\frac{\mathscr{L}_{p,\mathfrak{c}}(\chi)}{\Omega_p(\chi)}\equiv \frac{L_\mathfrak{c}(\chi^{-1},0)}{\Omega_\infty(\chi)}\pmod{E_{\chi}^\times}.
\end{equation}
To check this relation, let $(1+j,-j)$ be the infinity type of $\chi$. Since the algebraic number 
\[ \Bigl(i{\sqrt{D_K}}\Bigr)^{-j}\cdot j!\cdot\bigl(1-\chi(\mathfrak{p})/p\bigr)\cdot\bigl(1-\chi^{-1}(\bar{\mathfrak{p}})\bigr) \] 
belongs to $E_\chi^\times$, formula \eqref{Katz} yields the congruence 
\begin{equation} \label{congruence-eq}
\frac{\mathscr{L}_{p,\mathfrak{c}}(\chi)}{\Omega_p(A)^{2j+1}}\equiv\frac{ L_{\mathfrak{c}}(\chi^{-1},0)}{(2\pi i)^{-j}\Omega_\infty(A)^{2j+1}}\pmod{E_\chi^\times}. 
\end{equation}
On the other hand, by definition, $\Omega_\infty(\chi)=\mathfrak{g}(\varphi)\cdot\Omega_\infty(A)^{2j+1}\cdot(2\pi i)^{-j}$ (equality \eqref{period3}) and $\Omega_p(\chi)=\mathfrak{g}(\varphi)\cdot\Omega_p(A)^{2j+1}$ (equality \eqref{period4}). Therefore, congruence \eqref{congruence-eq} gives
\[ \frac{\mathfrak{g}(\varphi)\cdot\mathscr{L}_{p,\mathfrak{c}}(\chi)}{\Omega_p(\chi)}\equiv\frac{ \mathfrak{g}(\varphi)\cdot L_{\mathfrak{c}}(\chi^{-1},0)}{\Omega_\infty(\chi)}\pmod{E_\chi^\times}, \]
which immediately gives congruence \eqref{katz-cong}. 
\end{remark}

%    By \cite[Remark 3.7]{BDP2}, we know that the central character of a self-dual character is either the trivial character $1$ or the non-trivial quadratic character $\varepsilon_K$ attached to $K$. We can define regions
%    \[
%        \Sigma_\mathrm{sd}^+(\mathfrak{c} )
%        \colon =\{\chi \in \Sigma_{\text{sd}}(\mathfrak{c}) \mid \varepsilon_\chi=1 \},\]
%        \[\Sigma_\mathrm{sd}^-(\mathfrak{c} )
%        \colon =\{\chi \in \Sigma_{\text{sd}}(\mathfrak{c}) \mid \varepsilon_\chi=\varepsilon_K \}.
%    \]
%Finally, put $\Sigma_\mathrm{sd}^{(1)}(\mathfrak{c} )^{\pm }=\Sigma_\mathrm{sd}^\pm(\mathfrak{c} )\cap \Sigma_\mathrm{sd}^{(1)}(\mathfrak{c} )$ and $\Sigma_\mathrm{sd}^{(2)}(\mathfrak{c} )^{\pm }=\Sigma_\mathrm{sd}^\pm(\mathfrak{c} )\cap \Sigma_\mathrm{sd}^{(2)}(\mathfrak{c} )$.
% 

\subsection{BDP $p$-adic $L$-function}

Let $K$ be an imaginary quadratic field of discriminant $-D_K$ and let $\theta_\psi=\sum_{n\geq 1}a_n(\psi)q^n\in S_{k}\bigl(\Gamma_0(D_K\Norm(\mathfrak{f}_\psi)),\varepsilon_K\varepsilon_\psi\bigr)$ be the theta series attached to the Hecke character $\psi$ of $K$ of infinity type $(k-1,0)$. Then $E_\psi\subset \bar\Q$ is the Hecke field of $\theta_\psi$. For a Hecke character $\chi$ of $K$, denote by $L(\theta_\psi,\chi,s)$ the complex $L$-function of $\theta_\psi$ twisted by $\chi$.
%; we remark that the Euler factors at primes ${\mathfrak{p} \nmid \mathfrak{f}_\chi N}$ are 
%\[
%    L(f,\chi,s) \defeq \prod_{\mathfrak{p} \nmid \mathfrak{f}_\chi N} (1 - \chi(\mathfrak{p}) \alpha_{\Norm \mathfrak{p}}(f) p^{-s})^{-1}
%    (1 - \chi(\mathfrak{p}) \alpha_{\Norm \mathfrak{p}}(f) p^{-s})^{-1} \prod_{\mathfrak{p} \mid  \mathfrak{f}_\chi N} (\text{see Jaquet})^{-1}
%\]
From now on, we will impose

\begin{assumption} \label{ass2}
    The conductor $\mathfrak{f}_\psi$ of $\psi$ is a cyclic ideal of $\cO_K$ of norm prime to $D_K$.
\end{assumption}

As a consequence of Assumption \ref{ass2}, we see that the pair $(\theta_\psi,K)$ satisfies the \emph{Heegner hypothesis relative to $N\defeq D_K\Norm(\mathfrak{f}_\psi)$}: there exists a cyclic ideal $\mathfrak{N}$ of $\cO_K$ of norm $N$, \emph{i.e.}, $\cO_K/\mathfrak{N} \simeq \Z /N\Z$. See, \emph{e.g.}, \cite[Lemma 3.14]{BDP1} for details. Fix such an ideal $\mathfrak{N}$ and write $\mathfrak{N}_{\varepsilon_K\varepsilon_\psi}$ for the unique ideal dividing $\mathfrak{N}$ of norm $N_{\varepsilon_K\varepsilon_\psi}$, where $N_{\varepsilon_K\varepsilon_\psi}$ is the norm of the conductor of the character $\varepsilon_K\varepsilon_\psi$ of $\theta_\psi$.  

\begin{definition} \label{defcc} 
A Hecke character $\phi$ of type $(\ell_1, \ell_2)$ is \emph{central critical} for $\theta_\psi$ if $\ell_1+\ell_2 = k$ and $\varepsilon_\phi = {\varepsilon_K\varepsilon_\psi}$. 
\end{definition}

Let $\Sigma_\mathrm{cc}(\psi)$ be the set of central critical characters of $\theta_\psi$. For a given odd integer $c$ prime to $ N=D_K\Norm(\mathfrak{f}_\psi)$, define a subset $\Sigma_{\mathrm{cc}}(\psi,c) \subset \Sigma_\mathrm{cc}(\psi)$ consisting of those characters $\phi$ such that 
\begin{itemize}
\item $\mathfrak{f}_\phi = (c)\cdot \mathfrak{N}_{\varepsilon_K\varepsilon_\psi}$;
\item the local sign $\epsilon_q(\theta_\psi,\phi^{-1})$ is $1$ for all prime numbers $q$. 
\end{itemize}
We will consider the following subsets of $\Sigma_{\mathrm{cc}}(\psi,c)$:
\begin{itemize} 
\item $\Sigma_{\mathrm{cc}}^{(1)}(\psi,c)$ is the subset of $\Sigma_{\mathrm{cc}}(\psi,c)$ consisting of those characters having infinity type $(k+j, -j)$ with $1-k\leq j\leq -1$;
\item $\Sigma_{\mathrm{cc}}^{(2)}(\psi,c)$ is the subset of $\Sigma_{\mathrm{cc}}(\psi,c)$ consisting of those characters having infinity type $(k+j, -j)$ with $j\geq 0$.
\end{itemize} 
Fix a prime number $p$ splitting in $K$ and not dividing $Nc$. The set $\Sigma_{\mathrm{cc}}^{(2)}(\psi,c)$ is dense in the completion $\widehat{\Sigma}_{\mathrm{cc}}(\psi,c)$ of ${\Sigma}_{\mathrm{cc}}(\psi,c)$ with respect to the $p$-adic compact-open topology ${\Sigma}_{\mathrm{cc}}(\psi,c)$ (see, \emph{e.g.}, \cite[\S 5.2]{BDP1} for details). Recall that $E_{\psi,\phi}$ is the composite of $E_\psi$ and $E_\phi$; denote by $E_{\psi,\phi,{\varepsilon_K\varepsilon_\psi}}$ the composite of $E_{\psi,\phi}$ with the abelian extension of $\Q(\varepsilon_K\varepsilon_\psi)$ cut out by $\varepsilon_K\varepsilon_\psi$. 
%Finally, set:
%\begin{equation}\label{Fchipsi}F_{\psi,\phi}=E_{\psi,\phi,\varepsilon_K\varepsilon_\psi}\left(\sqrt{-D_K\Norm(\mathfrak{f}_\psi})\right).\end{equation}

The results of \cite{BDP1} show that there exists a unique $p$-adic analytic function 
\[ L_p(\theta_\psi,-):\widehat{\Sigma}_{\mathrm{cc}}(\psi,c)\longrightarrow\C_p, \] 
which we shall call \emph{BDP $p$-adic $L$-function}, satisfying the following interpolation property.
 
\begin{theorem}[Bertolini--Darmon--Prasanna] 
For all characters $\phi \in {\Sigma}_{\mathrm{cc}}^{(2)}(\psi,c)$, the equality
\begin{equation} \label{BDP}
\frac{L_p(\theta_\psi,\phi)}{\Omega_p(A)^{2(k+2j)}}=\frac{E_p(\psi)C(\psi,\phi,c)}{\omega(\psi,\phi)}\cdot\frac{L(\theta_\psi,\phi^{-1},0)}{\Omega_\infty(A)^{2(k+2j)}}
\end{equation}
holds true.
\end{theorem}

\begin{proof} Combine \cite[Theorem 5.5]{BDP1} and \cite[Proposition 5.10]{BDP1}. \end{proof}

In the statement of the previous theorem, the notation not introduced above is as follows:  
\begin{itemize}
\item $E_p(\psi)\defeq\bigl(1-\phi^{-1}(\bar{\mathfrak{p}}) a_p(\psi)+ \ \phi^{-2}(\bar{\mathfrak{p}}) {\varepsilon_K\varepsilon_\psi}(p)p^{k-1}\bigr)^2$; 
\item $\omega(\psi,\phi)\in E_{\psi,\phi,{\varepsilon_K\varepsilon_\psi}}^\times$ is defined in \cite[equation (5.1.11)]{BDP1};
\item if $\omega_K$ is the number of roots of unity in $K$, then 
\[ C(\psi,\phi,c)\defeq\frac{j!\cdot(k+j-1)!\cdot\omega_K}{2}\cdot\biggl(\frac{2\pi}{c\sqrt{D}}\biggr)^{\!k+2j-1}\cdot\prod_{q\mid c}\frac{q-{\varepsilon_K\varepsilon_\psi}(q)}{q-1}. \]
 \end{itemize}
As before, equality \eqref{BDP} must be understood as stating that the right-hand side belongs to $\bar\Q$ and its image in $\bar\Q_p$ via the fixed embedding $\bar\Q\hookrightarrow\bar\Q_p$ is equal to the left-hand side.  
 
 \subsection{A factorization formula} \label{secfac}
 
Now we want to recall the factorization formula in \cite[\S 3D]{BDP2} expressing the BDP $p$-adic $L$-function as a product of two Katz $p$-adic $L$-functions. Fix a Hecke character $\psi$ of infinity type $(k-1,0)$ with $k\geq 2$ and consider its associated theta series
$\theta_\psi = \sum_{\mathfrak{q}} \psi(\mathfrak{q}) q^{\Norm \mathfrak{q}}$, as in \S\ref{secscholl}. Write $a_n(\psi)$ for the $n$-th Fourier coefficient of $\theta_\psi$, so that $\theta_\psi=\sum_{n\geq 1}a_n(\psi)q^n$, and keep Assumption \ref{ass2} in force. 

If $\phi \in \Sigma_{\mathrm{cc}}(\psi,c)$ has infinity type $(k+\ell,-\ell)$, then $\psi^{-1}\phi$ has infinity type $(1+\ell,-\ell)$ and ${\psi^*}^{-1}\phi$ has infinity type $(k+\ell,1-(k+\ell))$. By \cite[Lemma 3.15]{BDP2}, there is a factorization 
\begin{equation}\label{cplxfact}
L(\theta_\psi, \phi^{-1}, s)= L(\psi\phi^{-1},s) \cdot L(\psi^{\ast} \phi^{-1},s).
\end{equation}
Put $\mathfrak{d}_K\defeq\sqrt{-D_K}$. By \cite[Lemma 3.16]{BDP2}, we know that
\begin{itemize}
\item if $\phi \in \Sigma^{(1)}_{\mathrm{cc}}(\psi,c)$, then $ \psi^{-1} \phi \in \Sigma_{\mathrm{sd}}^{(1)}(c \mathfrak{d}_K)^- $ and 
$\psi^{\ast -1} \phi \in \Sigma_{\mathrm{sd}}^{(2)}\bigl(c \mathfrak{d}_K\Norm(\mathfrak{f}_\psi)\bigr)^-$;
\item if $\phi \in \Sigma^{(2)}_{\mathrm{cc}}(\psi,c)$, then $\psi^{-1} \phi,\psi^{\ast -1} \phi\in\Sigma_{\mathrm{sd}}^{(2)}\bigl(c \mathfrak{d}_K\Norm(\mathfrak{f}_\psi)\bigr)^-$. 
\end{itemize}  
The following factorization formula can be found in \cite{BDP2}. 

\begin{theorem}[Bertolini--Darmon--Prasanna] \label{BDP2-thm}
For all $\phi \in \Sigma_{\mathrm{cc}}(\psi,c)$ of type $(k+\ell,-\ell)$, the equality
\begin{equation} \label{eq**} 
L_p(\theta_\psi,\phi) = \frac{\omega_K}{2c^{k+2j-1}}\cdot\omega(\theta_\psi,\phi)^{-1}\cdot\prod_{q\mid c} \frac{q-\varepsilon_K(q)}{q-1}\cdot\mathscr{L}_{p,\:c\mathfrak{d}_K}(\psi^{-1}\phi)\cdot\mathscr{L}_{p, \: c\mathfrak{d}_K\Norm(\mathfrak{f}_\psi)}(\psi^{\ast-1}\phi)
\end{equation}
holds true.
\end{theorem}

\begin{proof} This is \cite[Theorem 3.17]{BDP2}. \end{proof}
        
\begin{remark}
Equality \eqref{eq**} is proved using factorization formula \eqref{cplxfact} and then comparing the relevant interpolation formulas \eqref{Katz} and \eqref{BDP} for $p$-adic $L$-functions at characters in $\Sigma^{(2)}_{\mathrm{cc}}(\mathfrak{c},\mathfrak{N},\varepsilon)$; the conclusion for all characters in the $p$-adic completion $\widehat\Sigma_{\mathrm{cc}}(\psi,c)$ of $\Sigma_{\mathrm{cc}}(\psi,c)$ (therefore, including characters in $\Sigma^{(1)}_{\mathrm{cc}}(\mathfrak{c},\mathfrak{N},\varepsilon)$) follows by a density argument. 
\end{remark}

The next result is an easy consequence of Theorem \ref{BDP2-thm}.

\begin{corollary} 
For all $\phi \in \Sigma_{\mathrm{cc}}(\mathfrak{c},\mathfrak{N}, \varepsilon)$, the congruence
\begin{equation} \label{factorization-modulo-E}
L_p(\theta_\psi,\phi)\equiv\mathscr{L}_{p, \:c \mathfrak{d}_K}(\psi^{-1}\phi)\cdot\mathscr{L}_{p, \:c \mathfrak{d}_K\Norm(\mathfrak{f}_\psi)}(\psi^{\ast-1}\phi)\pmod{E_{\psi,\phi,\varepsilon_K\varepsilon_\psi}^\times}
\end{equation}
holds true.
\end{corollary} 

\begin{proof} The factors appearing on the right-hand side of \eqref{eq**} belong to $E_{\psi,\phi,\varepsilon_K\varepsilon_\psi}^\times$, whence the claim. \end{proof}
 
\subsection{Generalized Heegner cycles} \label{secGHC} 

We offer a quick review of generalized Heegner cycles \emph{\`a la} Bertolini--Darmon--Prasanna (\cite{BDP1}). Let $\psi$ and $K$ be as in Section \ref{BDPsec} and keep Assumption \ref{ass2} in force. Recall that $p$ splits in $K$, $c$ is an odd integer prime to $N=D_K\Norm(\mathfrak{f}_\psi)$ and $p\nmid Nc$. Moreover, assume $N\geq 5$ and, as before, set $r\defeq k-2$ and $\ell\defeq r/2$.  

Recall the elliptic curve $A$ with CM by $\cO_K$ defined over $H$ that was fixed in \S\ref{remell} and \S\ref{persec}, and also the ideal $\mathfrak{N}\subset\cO_K$ such that $\cO_K/\mathfrak{N}\simeq\Z/N\Z$. Pick $t_A\in A[\mathfrak{N}]$; the pair $(A,t_A)$ is defined over an abelian extension $H_\mathfrak{N}$ of $K$, independent of the choice of $t_A$, and so defines a point in $X_{\Gamma_\psi}(H_{\mathfrak{N}})$. Let $\cO_c=\Z+c\cO_K$ be the order of $K$ of conductor $c$ and define $A_c\defeq\C/\mathcal{O}_c$; then $A_c$ is the image of an isogeny $\varphi_c:A\rightarrow A_c$ of degree $c$, defined over the ring class field $H_c$ of $K$ of conductor $c$. For every ideal $\mathfrak{a}$ of $\cO_c$ relatively prime to $\mathfrak{N}\cap\mathcal{O}_c$, let $\varphi_{\mathfrak{a},c}:A_c\rightarrow A_\mathfrak{a}\defeq A_c/A_c[\mathfrak{a}]$ be the canonical isogeny (\emph{cf.} \cite[equation (1.4.7)]{BDP1} and \cite[\S5.1]{BDP1}). Then $\varphi_\mathfrak{a}\defeq\varphi_{\mathfrak{a},c}\circ\varphi_c:A\rightarrow A_\mathfrak{a}$ is an isogeny of degree $c\Norm(\mathfrak{a})$; observe that the pair $\bigl(A_\mathfrak{a},\varphi_\mathfrak{a}(t_A)\bigr)$ determines a point in $X_{\Gamma_\psi}(H_{c,\mathfrak{N}})$ along with an embedding $A_\mathfrak{a}^r\hookrightarrow W_r$, where $H_{c,\mathfrak{N}}$ is the composite of the fields $H_c$ and $H_\mathfrak{N}$. We may then define $\Gamma_{\varphi_\mathfrak{a}}\defeq\mathrm{Graph}(\varphi_\mathfrak{a})\subset A\times A_\mathfrak{a}$ to be the graph of $\varphi_\mathfrak{a}$ and  
\[ \Upsilon_{\varphi_\mathfrak{a}}\defeq\mathrm{Graph}(\varphi_\mathfrak{a})^r \subset (A\times A_\mathfrak{a})^r\subset A^r\times W_r\subset X_r. \] 
It turns out that the \emph{generalized Heegner cycle} 
\[ \Delta_{\varphi_\mathfrak{a}}\defeq\epsilon_X\Upsilon_{\varphi_\mathfrak{a}}\in \mathrm{CH}_0^{r+1}(X_r)(H_{c,\mathfrak{N}}) \] is a homologically trivial cycle with $\Q$-coefficients of codimension $r+1$ in $X_r$, defined over $H_{c,\mathfrak{N}}$ (see, \emph{e.g.}, \cite[\S2.3]{BDP1} for details).  
 
\subsection{A $p$-adic Gross--Zagier formula} \label{secGZ} 

We continue with the setting of \S\ref{secGHC}. Let $\widehat{K}=\A_{K,\text{fin}}$ denote the ring of finite adeles of $K$ and define 
$\widehat{\cO}_c\defeq\cO_c\otimes_\Z\widehat{\Z}$, where $\widehat{\Z}$ is the profinite completion of $\Z$. For every $\cO_c$-ideal prime to $\mathfrak{N}_{\varepsilon_K\varepsilon_\psi}c$, let $a=(a_\mathfrak{q})_\mathfrak{q}\in \widehat{K}^\times$ satisfy $\mathfrak{a}=a\widehat{\cO}_c\cap K$; then $a$ is necessarily a $p$-adic unit at all primes $\mathfrak{q}\,|\,\mathfrak{N}_{\varepsilon_K\varepsilon_\psi}c$.  

Fix a character $\phi\in\Sigma_{\mathrm{cc}}^{(1)}(\psi,c)$ of infinity type $(\ell+1,\ell+1)$ (\emph{cf.} Definition \ref{defcc}). 
As explained, \emph{e.g.}, in \cite[p. 1093]{BDP1}, we view $\phi$ as a character on $\cO_c$-ideals $\mathfrak{a}$ prime to $\mathfrak{N}_{\varepsilon_K\varepsilon_\psi}c$ by setting $\phi(\mathfrak{a})\defeq\prod_{\mathfrak{q}\nmid \mathfrak{N}  _{\varepsilon_K\varepsilon_\psi}c}\phi_\mathfrak{q}(a_\mathfrak{q})$, where $a=(a_\mathfrak{q})_\mathfrak{q}\in \widehat{K}^\times$ is chosen to satisfy $\mathfrak{a}=a\widehat{\cO}_c\cap K$, as before. In particular (\emph{cf.} \cite[equation (4.1.8)]{BDP1}), there is an equality 
\[ \phi\bigl((\alpha)\bigr)=\Norm_K^{\ell+1}(\alpha)\varepsilon_K\varepsilon_\psi(\alpha) \] 
for every $\alpha\in K^\times$ that is a unit at all primes dividing $\mathfrak{N}_{\varepsilon_K\varepsilon_\psi}c$, where $\varepsilon_K\varepsilon_\psi(\alpha)$ is a shorthand for $\varepsilon_K\varepsilon_\psi\bigl(\alpha\pmod{\mathfrak{N}_{\varepsilon_K\varepsilon_\psi}}\bigr)$. By a fundamental result of Bertolini--Darmon--Prasanna, the value $L_p(\theta_\psi,\phi)$ can be described as follows.

\begin{theorem}[Bertolini--Darmon--Prasanna] \label{thm5.11} 
If $\phi \in \Sigma_{\mathrm{cc}}^{(1)}(\psi,c)$ has infinity type $(\ell+1,\ell+1)$, then 
\begin{equation}\label{BDPTHM}
{L_p(\theta_\psi,\phi)} = \frac{E_p(\theta_\psi)}{c^\ell \ell!} \cdot \Biggl( \sum_{[\mathfrak{a}] \in \Pic(\mathcal{O}_c)} \phi^{-1}(\mathfrak{a})\Norm_K(\mathfrak{a})  \log_{X,\mathcal{F}}(\Delta_{\varphi_\mathfrak{a}} )\bigl(\omega_{\theta_\psi}\wedge \omega_A^\ell\eta_A^{\ell}\bigr)\!\Biggr)^2
\end{equation}
for an explicit constant $E_p(\theta_\psi)$.
\end{theorem}

\begin{proof} This is \cite[Theorem 5.13]{BDP1}. \end{proof}
      
\section{A generalized Rubin formula} \label{secGRF}

In this section we prove the main result of this paper.

\subsection{Good pairs} \label{secgoodpairs}  

Let $\mathfrak{c}=\bar{\mathfrak{c}}$ be an ideal of $\cO_K$ invariant under complex conjugation and let $\chi\in\Sigma_\mathrm{sd}(\mathfrak{c})$. By \cite[Remark 3.7]{BDP2}, we know that the central character $\varepsilon_\chi$ of a self-dual character $\chi$ is either trivial or equal to the non-trivial quadratic character $\varepsilon_K$ attached to $K$. Let $\mathfrak{d}_K$ be as in \S \ref{secfac} and note that if $\varepsilon_\chi=\varepsilon_K$, then $\mathfrak{d}_K\,|\,\mathfrak{c}$.

In the assumption below, $k\geq2$ is an even integer; as before, $r\defeq k-2$ and $\ell\defeq r/2$.

\begin{assumption}\label{ass}
Let $\mathfrak{c}=\bar{\mathfrak{c}}$ be an ideal of $\cO_K$ invariant under complex conjugation and let $\chi \in \Sigma_{\mathrm{sd}}(\mathfrak{c})$ be a Hecke character of infinity type $(1+\ell,-\ell)$. With notation as above, assume that $\chi$ satisfies the following  conditions:
\begin{itemize}
\item $\varepsilon_\chi=\varepsilon_K$;
\item the sign $\omega_\chi$ of the functional equation for the $L$-function $L(\chi,s)$ is $-1$; 
\item $\mathfrak{d}_K\parallel \mathfrak{c}$; thus, $\mathfrak{c}=(c)\mathfrak{d}_K$ for a unique positive integer $c$ prime to $D_K$. 
\end{itemize} 
\end{assumption}
For the rest of the paper,  let us fix a character $\chi$ satisfying Assumption \ref{ass} and a prime number $p\nmid cD_K$ that splits in $K$.

Let $\psi$ be a Hecke character of $K$ of infinity type $(k-1,0)$ and consider the theta factorization \begin{equation}\label{factorization1}\chi=\psi\vartheta.\end{equation}  
For any Dirichlet character $\phi$, let $H_\phi$ denote the composite field of $K$ and the cyclotomic extension of $\Q$ cut out by $\phi$. Since $\varepsilon_\chi=\varepsilon_K$, we have $\varepsilon_\vartheta=\varepsilon_K\varepsilon_\psi$, so $H_{\varepsilon_\psi}=H_{\varepsilon_\vartheta}$.
%; also, note that $E_\vartheta=E_\psi$.  

\begin{definition} \label{goodfact} 
The pair $(\psi,\vartheta)$ is \emph{good for $\chi$} if factorization \eqref{factorization1} holds and 
\begin{enumerate}
\item the conductor $\mathfrak{f}_\psi$ of $\psi$ is a cyclic $\cO_K$-ideal prime to $D_Kp$;
\item $\vartheta^{-1}\Norm_K \in \Sigma_{\mathrm{cc}}^{(1)}(\psi,c)$;
\item $L_{\mathfrak{c}\Norm(\mathfrak{f}_\psi)}(\psi^\ast\vartheta \Norm_K^{-1},0)\ne 0$. 
\end{enumerate}
If $(\psi,\vartheta)$ is a good pair for $\chi$, then \eqref{factorization1} is a \emph{good factorization} of $\chi$.  
\end{definition}
If $\chi=\psi\vartheta$ is a good factorization in the sense of Definition \ref{goodfact}, then, since $\chi$ is self-dual, there is an equality 
%\begin{equation}\label{GP1} 
$\chi^\ast = 
%\chi^{-1}\Norm_K=
%\psi^{-1}\varphi^{-1}(\chi_A\chi_A^*)^{-\ell}\Norm_K=
\psi^{-1}\vartheta^{-1}\Norm_K
%\end{equation}
$
and 
%\begin{equation}\label{GP1+2/3}
$
\chi^*\in \Sigma_{\mathrm{sd}}^{(1)}(\mathfrak{c}).
$
%\end{equation}
Furthermore, we have $\psi^{\dagger}\vartheta^{-1} \Norm_K\in \Sigma_{\mathrm{sd}}^{(2)}\bigl(\mathfrak{c}\Norm(\mathfrak{f}_\psi)\bigr)$ (recall the notation $\psi^\dagger=(\psi^*)^{-1}$), so from property (3) in Definition \ref{goodfact} and Theorem \ref{Katz} it follows that 
\begin{equation} \label{GP3}
\mathscr{L}_{p, \mathfrak{c}\Norm(\mathfrak{f}_\psi)}(\psi^{\dagger}\vartheta^{-1}\Norm _K)\neq 0.
\end{equation}

\begin{remark}
Condition (3) in Definition \ref{goodfact}, and therefore the non-vanishing in \eqref{GP3} as well, holds generically in the region of interpolation because of the Heegner hypothesis and the hypothesis that the root number of $L(\chi,s)$ is $-1$, by Assumption \ref{ass}. More precisely, assume that $(\psi,\vartheta)$ satisfies the first two conditions in Definition \ref{goodfact}. Since 
$\vartheta^{-1}\Norm_K \in \Sigma_{\mathrm{cc}}^{(1)}(\psi,c)$ by (2), the sign of the functional equation of $L(\theta_\psi,\vartheta N_K^{-1},s)$ is $-1$, and from the factorization 
\begin{equation} \label{facteq}
L(\theta_\psi,\vartheta\Norm_K^{-1},s)=L(\psi\vartheta\Norm_K^{-1},s)L(\psi^*\vartheta\Norm_K^{-1},s)=L(\chi\Norm_K^{-1},s)L(\psi^\ast\vartheta\Norm_K^{-1},s)
\end{equation} 
recalled above (\emph{cf.} \cite[Lemma 3.15]{BDP2}) it follows that the sign of the functional equation of $L(\psi^\ast\vartheta^{-1}\Norm_K^{-1},s)$ is $+1$. Thus, condition (3) is expected to hold generically. 
\end{remark}

\begin{lemma}\label{lemma3.3} 
    If $\chi=\psi\vartheta$ and $\mathfrak{f}_\psi$ is a cyclic $\mathcal{O}_K$-ideal prime to $D_Kpc$, then $\vartheta^{-1}\Norm_K\in\Sigma^{(1)}_\mathrm{cc}(\psi,c)$.
\end{lemma}
\begin{proof} 
%We first check that the conditions in Definition \ref{defcc} of central critical characters for $\theta_\psi$ are satisfied.  
The infinity type of $\vartheta^{-1}\Norm_K$ is $(\ell+1,\ell+1)$, and $2\ell+2=k$; moreover from 
$\varepsilon_\chi=\varepsilon_\psi\varepsilon_\vartheta$ and $\varepsilon_\chi=\varepsilon_K$, we see that $\varepsilon_{\vartheta^{-1}\Norm_K}=\varepsilon_{\vartheta^{-1}}=\varepsilon_K\varepsilon_\psi$,
so $\vartheta^{-1}\Norm_K\in \Sigma_\mathrm{cc}(\psi)$. 
For the further conditions required to the subset $\Sigma_\mathrm{cc}(\psi,c)$ of $\Sigma_\mathrm{cc}(\psi)$, note that from $\mathfrak{f}_\chi=(c)\mathfrak{d}_K$ and $(\mathfrak{f}_\psi,(c)\mathfrak{d}_K)=1$, we have $\mathfrak{f}_\vartheta=(c)\mathfrak{d}_K\mathfrak{f}_\psi$; moreover, the conductor of the character 
$\varepsilon_K\varepsilon_\psi$ of $\theta_\psi$ is $D_K\Norm_K(\mathfrak{f}_\psi)$, and therefore the condition $\mathfrak{f}_{\vartheta^{-1}\Norm_K}=(c)\mathfrak{N}_{\varepsilon_K\varepsilon_\psi}$ is satisfied 
taking $\mathfrak{N}_{\varepsilon_K\varepsilon_\psi}=\mathfrak{d}_K\mathfrak{f}_\psi$. 
Finally, the local root numbers are all equal to $1$ because $D_K\mid N_{\varepsilon_K\varepsilon_\psi}=\Norm_K(\mathfrak{N}_{\varepsilon_K\varepsilon_\psi})$, so the set $S(\theta_f)$ in \cite[p. 1094]{BDP1} is empty.
\end{proof}

Let $S_\chi$ denote the set of good pairs for $\chi$. For each $(\psi,\vartheta)\in S_\chi$, denote by $E_{\psi,\vartheta}$ the composite field of their coefficient fields $E_\psi$ and $E_\vartheta$. Note that, thanks to the factorization $\chi=\psi\vartheta$, the composite field $E_{\chi,\psi}$ of $E_\chi$ and $E_\psi$ is equal to the composite field $E_{\psi,\vartheta}$; also, $E_\chi\subset E_{\psi,\vartheta}$; finally, if we write $\vartheta=\varphi(\chi_A\chi_A^*)^\ell=(\varphi\varphi_A^\ell)\Norm_K^{-\ell}$, then $E_\vartheta=E_{\varphi\varphi_A^\ell}$, and therefore $E_{\psi,\vartheta}=E_{\psi,\varphi\varphi_A^\ell}$ and $\varepsilon_\vartheta=\varepsilon_{\varphi\varphi_A^\ell}$.   
Recall also the field $E_{\psi,\vartheta,\varepsilon_K\varepsilon_\psi}$ introduced in \S\ref{BDPsec}. 
Define $F_{\psi,\vartheta}$ to be the composite field of $E_{\psi,\vartheta}$ and the field extensions $H_{\varphi\varphi_A^\ell}$ and $H_{(\varphi\varphi_A^\ell)^\ast}$ of $K$ cut out by the finite order characters $\varphi\varphi_A^\ell$ and $(\varphi\varphi_A^\ell)^*$, respectively. 
Since $\varepsilon_\vartheta=\varepsilon_K\varepsilon_\psi$, we see that 
$E_{\psi,\vartheta,\varepsilon_K\varepsilon_\psi}\subset F_{\psi,\vartheta}$. 
Note that $\Gal(H_{c,\mathfrak{N}}/H_c)$ acts faithfully on $A_c[\mathfrak{N}]$, and therefore can be canonically identified with a subgroup of $\Z/N\Z$. Let $H_{c,\psi}=H_{c,\mathfrak{N}}^{\ker(\varepsilon_K\varepsilon_\psi)}$ denote the subfield of
$H_{c,\mathfrak{N}}$ fixed by $\ker(\varepsilon_K\varepsilon_\psi)$; then $H_{c,\psi}$ is the composite field of $H_c$ and the extension $\Q(\varepsilon_K\varepsilon_\psi)$ of $\Q$ cut out by $\varepsilon_K\varepsilon_\psi$ (note that $H_{c,\psi}$ is a cyclotomic extension of $H_c$). 
Also, since $\varepsilon_\vartheta=\varepsilon_{\varphi\varphi_A^\ell}$ 
and $\varepsilon_\vartheta=\varepsilon_{K}\varepsilon_\psi$, $H_{\varphi\varphi_A^\ell}$ and $H_{(\varphi\varphi_A^\ell)^\ast}$ both contain $\Q(\varepsilon_K\varepsilon_\psi)$, and therefore by considering the conductors we see that both of them are contained in $H_{c,\psi}$, 
and therefore also in $H_{c,\mathfrak{N}}$.   
% We have then the following diagram of fields
%\begin{equation}\label{fielddiagram}\xymatrix{
%&&& H_{c,\mathfrak{N}}\\
%F_{\psi,\vartheta}
%&&H_{c,\psi}\ar@{-}[ur] &\\
%E_{\psi,\vartheta,\varepsilon_K\varepsilon_\psi}\ar@{-}[u]&H_{\varphi\varphi_A^\ell}H_{(\varphi\varphi_A^\ell)^*}\ar@{-}[ul]\ar@{-}[ur]& \\
%E_{\psi,\vartheta}\ar@{-}[u]&  H_{(\varphi\varphi_A^\ell)^*}\ar@{-}[u]&H_{\varphi\varphi_A^\ell}\ar@{-}[ul]\ar@{-}[uu]& H_c\ar@{-}[uuu]\\
%&K\ar@{-}[u]\ar@{-}[ul]\ar@{-}[ur]\ar@{-}[rru]&
%}\end{equation}
%which will be useful to visualize the various fields involved. 

\subsection{Properties of good pairs}\label{secgoodpairs2}

Our next goal is to show the existence of good pairs and the equality $\bigcap_{(\psi,\vartheta)\in S_\chi}F_{\psi,\vartheta}=E_\chi$. 
To prove this result, we follow the strategy in \cite{BDP2} closely. However, in light of the more complicated setting we are working in, we decided to split the proof into several lemmas, using which one can finally mimic the arguments in \cite{BDP2} to prove Proposition \ref{Goodpairs}.
%
%========================
%
%{\color{red} WE CAN JUST STATE THEOREM \ref{Goodpairs} AND SAY THAT THE PROOF IS THE SAME AS IN BDP. BUT BEFORE BETTER TO CHECK IT! IF EVERYTHING IS OK, THEN WE CAN JUST ELIMINATE THIS SECTION AND PUT THE THEOREM IN THE END OF THE PREVIOUS SECTION.} 
%
%========================

We first recall an analytic result due to Greenberg. Let $\ell$ be a prime number that splits in $K$ and write $K_\infty^{(\ell)}$ for the anticyclotomic $\Z_\ell$-extension of $K$.

\begin{lemma}[Greenberg] \label{greenberg}
Let $\lambda\in \Sigma_{\mathrm{sd}}(\mathfrak{f}_\psi)$ be a self-dual Hecke character of $K$. Assume that the sign $\omega_{\lambda}$ in the functional equation of $L(\lambda,s)$ is $+1$. Then there are infinity many finite order characters $\phi$ of $\Gal\bigl(K_\infty^{(\ell)}/K\bigr)$ such that $L(\lambda,\phi,k/2)\neq 0$. 
\end{lemma}

\begin{proof} By \cite[Theorem 1]{Gr85} and the proof of \cite[Proposition 1]{Gr85}, Katz's $\ell$-adic $L$-function $\mathscr{L}_{\ell,\mathfrak{f}_\lambda}$ does not vanish identically on any open $\ell$-adic neighborhood $U$ of $\lambda$ in $\widehat{\Sigma}_{\text{crit}}(\mathfrak{f}_\lambda)$. Fix an $\ell$-adic neighborhood $U$ of $\lambda$ containing a dense subset of points of the form $\lambda\phi$, where $\phi$ is a finite order character of $\Gal(K^{(\ell)}_\infty/K)$. The restriction of  $\mathscr{L}_{\ell,\mathfrak{f}_\lambda}$ to $U$ is described by a power series with $\ell$-adically bounded coefficients, so by the Weierstrass preparation theorem $\mathscr{L}_{\ell,\mathfrak{f}_\lambda}$ has only finitely many zeros in $U$. The result follows. \end{proof}
 
Let $\bar{S}_\chi$ be the set of pairs satisfying the first two conditions in the definition of a good pair, hence $S_\chi \subset \bar{S}_\chi$. 

\begin{lemma} \label{GPlemma1}  
The set $\bar{S}_\chi$ is non-empty.
\end{lemma}

\begin{proof} Simply observe that, for any Hecke character $\psi$ of infinity type $(k-1,0)$ and cyclic conductor prime to $\mathfrak{c}$, setting $\vartheta\defeq\chi\psi^{-1}$ produces a pair $(\psi,\vartheta)\in\bar{S}_\chi$ by Lemma \ref{lemma3.3}. On the other hand, the existence of Hecke characters $\psi$ as above is well known. \end{proof}

In the statement of the next lemma, we implicitly use Lemma \ref{GPlemma1}.

\begin{lemma} \label{GPlemma2} 
Given $(\psi,\vartheta)\in \bar{S}_\chi$, there exist $(\psi_1,\vartheta_1), (\psi_2,\vartheta_2) \in S_\chi$ such that $F_{\psi_1,\vartheta_1} \cap F_{\psi_2,\vartheta_2} \subset F_{\psi,\vartheta}$. 
\end{lemma}

\begin{proof} Pick a prime number $\ell$ that splits in $K$ as $\ell\cO_K=\lambda\bar{\lambda}$, is relatively prime to the class number of $K$ and to the conductor $\mathfrak{f}_\star$ of $\star$ for $\star\in\{\psi,\vartheta\}$ and is unramified in $F_{\psi,\vartheta}$. Denote by $K_\infty^{(\lambda)} = \bigcup \limits_{n \geq 0} K_n^{(\lambda)}$ (respectively, $K_\infty^{(\bar\lambda)} = \bigcup \limits_{n \geq 0} K_n^{(\bar\lambda)}$)  the unique $\Z_\ell$-extensions of $K$ that is unramified outside $\lambda$ (respectively, $\bar\lambda$), where the subfields $K_n^{(\lambda)}$ (respectively, $K_n^{(\bar\lambda)}$) satisfy $[K_n^{(\lambda)} : K]=\ell^n$ (respectively, $[K_n^{(\bar\lambda)}:K]=\ell^n$). The condition that $\ell$ does not divide the class number of $K$ implies that the fields $K_n^{(\lambda)}$ and $K_n^{(\bar\lambda)}$ are totally ramified at $\lambda$ and $\bar{\lambda}$, respectively. If $\alpha$ is any character of $\Gal\bigl(K_\infty^{(\lambda)}/K\bigr)$, then the pair $(\psi_1,\vartheta_1) =(\psi\alpha,\vartheta \alpha^{-1})$ still belongs to $\bar{S}_\chi$, with $\mathfrak{f}_\psi$ replaced by $\mathfrak{f}_\psi\lambda^n$ for a suitable $n\geq 0$. Furthermore, there is an equality
\[ L\bigl(\psi_1^\ast\vartheta_1\Norm_K^{-1},0\bigr)=L\bigl(\psi^\ast\vartheta\Norm_K^{-1}(\alpha^\ast/\alpha),0\bigr). \]
The character $\alpha^\ast/\alpha$ is an anticyclotomic character of $K$ of $\ell$-power order and conductor, and all such characters arise in this way. Since $(\psi,\vartheta)$ satisfies the first two conditions in Definition \ref{goodfact}, we see as in \eqref{facteq} that the sign of the functional equation of the $L$-function of $\psi^*\vartheta$ is $+1$. By Lemma \ref{greenberg}, there exists a choice of $\alpha$ for which $L\bigl(\psi^\ast\vartheta\Norm_K^{-1} (\alpha^\ast/\alpha),0\bigr)\neq 0$. The corresponding pair $(\psi_1,\vartheta_1)$ belongs to $S_\chi$ and satisfies
\[ F_{\psi_1,\vartheta_1} \subset F_{\psi,\vartheta,\ell,n} \defeq F_{\psi,\vartheta}\cdot\Q(\zeta_{\ell^n})\cdot K_n^{(\lambda)}\cdot K_n^{(\bar\lambda)}\subset\bar\Q \]
for some integer $n\geq1$. Note that the extension $F_{\psi,\vartheta,\ell,n}/F_{\psi,\vartheta}$ has degree dividing $d\ell^M$ for an integer $M\geq0$ and a divisor $d$ of $\ell-1$. Repeating this construction with a different prime number $\ell^\prime$ in place of $\ell$ such that $\ell\nmid(\ell^\prime-1)$ yields a second pair $(\psi_2,\chi_2) \in S_\chi$ and a corresponding extension $F_{\psi,\vartheta,\ell^\prime,n^\prime}$ for an integer $n^\prime\geq1$, which satisfies $F_{\psi_2,\vartheta_2} \subset F_{\psi,\vartheta,\ell^\prime,n^\prime}$ and whose degree over $F_{\psi,\vartheta}$ is of the form $d^\prime(\ell^\prime)^{M^\prime}$ for $d^\prime\,|\,(\ell^\prime-1)$ and an integer $M^\prime\geq 0$. Set 
\[ F_{(\ell,\ell^\prime)}\defeq F_{\psi,\vartheta,\ell,n} \cap F_{\psi,\vartheta,\ell^\prime,n^\prime} \]
Considering the degrees of these field extensions, we see that the degree of $F_{(\ell,\ell^\prime)}/F_{\psi,\vartheta}$ divides $\ell-1$, and then $F_{(\ell,\ell^\prime)}\subset F_{\psi,\vartheta}\cdot\Q(\zeta_\ell)$. Since $\ell$ is unramified in $F_{\psi,\vartheta}$, the extension $F_{(\ell,\ell^\prime)}/F_{\psi,\vartheta}$ must be totally ramified at the primes above $\ell$. On the other hand, $F_{(\ell,\ell^\prime)}/F_{\psi,\vartheta}$ is a subextension of $F_{\psi,\vartheta,\ell^\prime, n^\prime} /F_{\psi,\vartheta}$, so $F_{(\ell,\ell^\prime)}/F_{\psi,\vartheta}$ is unramified at the primes above $\ell$ as well, and hence $F_{(\ell,\ell^\prime)}=F_{\psi,\vartheta}$. 
It follows that $F_{\psi_1,\vartheta_1} \cap F_{\psi_2,\vartheta_2} \subset F_{\psi,\vartheta}$.\end{proof} 

The next lemma shows that the field $F_{\psi,\vartheta}$ can be replaced by $E_{\psi,\vartheta}$.

\begin{lemma} \label{GPlemma3} 
For all $(\psi,\vartheta)\in \bar{S}_\chi$, there exists a finite-order character $\alpha$ of $G_K$ such that the pair $(\psi\alpha,\vartheta\alpha^{-1})$ belongs to $\bar{S}_\chi$ and $F_{\psi,\vartheta} \cap F_{\psi\alpha,\vartheta\alpha} \subset E_{\psi,\vartheta}.$
\end{lemma}

\begin{proof} As before, write $\vartheta=\varphi\varphi_A^\ell\Norm_K^{-\ell}$. For simplicity, set $\beta\defeq\varphi\varphi_A^\ell$. The image of $\beta$ is cyclic, isomorphic to $\Z/n\Z$ for some integer $n\geq 1$. Pick $\alpha$ satisfying the following conditions: 
\begin{enumerate}
    \item[(i)] $\alpha$ has order $n$ and is ramified only at the prime $\lambda$ of $K$;
    \item[(ii)] $\lambda$ is prime to the conductor of $\beta$ and $\beta^\ast$; 
    \item[(iii)] $\ell$ is unramified in $F_{\psi,\vartheta}$.
\end{enumerate}
Let $H_{\beta\alpha}$ be the field cut out by $\alpha$. Conditions (i) and (ii) above imply that
\begin{enumerate}
\item[(iv)] the extension $H_{\beta\alpha}/K$ is totally ramified at $\lambda$ and unramified at $\lambda^\ast$, while the extension $H_{\beta^\ast\alpha^\ast}/K$ is unramified at $\lambda$ and totally ramified at $\lambda^\ast$.
\end{enumerate}
To further simplify our notation, set $M_\beta\defeq H_\beta H_{\beta^\ast}$, $M_{\beta\alpha}\defeq H_{\beta\alpha}H_{\beta^\ast\alpha^\ast}$. It follows from (iii) and (iv) that
\begin{enumerate}
\item[(v)] $E_{\psi,\vartheta}M_\beta/E_{\psi,\vartheta}$ is unramified at all primes above $\ell$;
\item[(vi)] any subextension of $E_{\psi,\vartheta}M_{\beta\alpha}/E_{\psi,\vartheta}$ is ramified at some prime above $\lambda$ or $\lambda^\ast$. 
\end{enumerate}
Therefore, $E_{\psi,\vartheta}M_\beta \cap E_{\psi,\vartheta}M_{\beta\alpha}=E_{\psi,\vartheta}$. On the other hand, $E_{\psi,\vartheta}M_\beta=E_{\psi,\vartheta}H_{\beta}H_{\beta^\ast}=F_{\psi,\vartheta}$. Finally, since $\alpha$ has order $n$, there are equalities $E_{\psi\alpha,\vartheta\alpha^{-1}}=E_{\psi,\vartheta}$ and 
\[ E_{\psi,\vartheta}M_{\beta\alpha}=E_{\psi,\vartheta}H_{\beta\alpha}H_{\beta^\ast\alpha^\ast}=E_{\psi\alpha,\vartheta\alpha^{-1}}H_{\beta\alpha}H_{\beta^\ast\alpha^\ast}=F_{\psi\alpha,\vartheta\alpha^{-1}}. \] 
This concludes the proof of the lemma. \end{proof} 

The following proposition is the main result of this subsection.

\begin{proposition} \label{Goodpairs}
The set $S_\chi$ of good pairs for $\chi$ is non-empty and $\bigcap_{(\psi, \vartheta) \in S_\chi} F_{\psi,\vartheta}=E_\chi.$ 
\end{proposition}

\begin{proof} Using Lemmas \ref{GPlemma2} and \ref{GPlemma3}, proceed as in the proof of \cite[Proposition 3.28]{BDP2}. \end{proof}

\begin{remark}
Results on good pairs in the vein of those described here and in \cite[\S 3E]{BDP2} can also be found in \cite[\S 5.1]{Cas-CM}, to which the reader is referred for details.
\end{remark}

\subsection{A reciprocity law for generalized Heegner cycles} \label{galoisdescent}
 
As customary, we shall denote by $\mathrm{rec}_K\colon \widehat{K}^\times/K^\times\rightarrow \Gal(K^\mathrm{ab}/K)$ the geometrically normalized reciprocity map of (global) class field theory. Set $G_{c,\mathfrak{N}}\defeq\Gal(H_{c,\mathfrak{N}}/K)$; in particular, $G_{c,\mathfrak{N}}$ is a quotient of $ \Gal(K^\mathrm{ab}/K)$. If $\mathfrak{a}$ is an $\cO_c$-ideal prime to $\mathfrak{N}_{\varepsilon_K\varepsilon_\psi}c$ and $a\in \widehat{K}^\times$ (a unit at all primes dividing $\mathfrak{N}_{\varepsilon_K\varepsilon_\psi}c$) satisfies $\mathfrak{a}=a\widehat{\mathcal{O}}_c\cap K$, then we define $\sigma_\mathfrak{a}\defeq{\mathrm{rec}_K(a^{-1})|}_{H_{c,\mathfrak{N}}}\in G_{c,\mathfrak{N}}$. 

For any field $F\supset H$, recall the \'etale Abel--Jacobi map
\[ {\Phi_{\text{\'et},X}^{(F)}}:\CH^{k-1}_0(X_r)(F)\longrightarrow H^1\bigl(F,V_X(k-1)\bigr)\longrightarrow H^1\bigl(F,V_W\otimes \boldsymbol{\Sym}^r_\text{\'et}(A)(r+1)\bigr) \]
 introduced in \eqref{PhiX} (here we use the symbol ${\Phi_{\text{\'et},X}^{(F)}}$ also for the composition of the two maps in the bottom line of \eqref{diagram}; we simplify, as before, our notation by writing, from now on, $\otimes$ in place of $\otimes_{\Q_p}$). With this notation, define $\xi_{X,\mathfrak{a}}\defeq{\Phi_{\text{\'et},X}^{(H_{c,\mathfrak{N}})}}(\Delta_{\varphi_{\mathfrak{a}}})$ and 
\begin{equation} \label{BDPX}
\xi_{X}\defeq\sum_{[\mathfrak{a}] \in\Pic(\cO_c)} \vartheta(\mathfrak{a})\xi_{X,\mathfrak{a}}\in 
H^1\Bigl(H_{c,\mathfrak{N}},V_{W}\otimes \boldsymbol{\Sym}_\text{\'et}^{2\ell}(A)(2\ell+1)\otimes  \mathcal{E}_\vartheta\Bigr).
\end{equation} 
To further lighten our notation, set 
\[ \boldsymbol{e}\defeq e_{\chi_A}e_{\chi_A^\ast},\quad\boldsymbol{e^\ast}\defeq\epsilon^\ast e_{\chi_A}e_{\chi_A^\ast}. \] 
Recall the notation for $\Gamma_{\varphi_\mathfrak{a}}$ and $\Delta_{\varphi_\mathfrak{a}}$ from \S\ref{secGHC}. There is a canonical embedding $A\subset B$ of varieties over $H$, and we may define  $\Xi_{\varphi_\mathfrak{a}}$ to be the image of $\Upsilon_{\varphi_\mathfrak{a}}$ via the maps 
\[ \Upsilon_{\varphi_\mathfrak{a}}\subset A_\mathfrak{a}^r\times A^r\subset A_\mathfrak{a}^r\times B^r\simeq A_\mathfrak{a}^r\times B^\ell\times (B^*)^\ell. \]
Here we use the identification $B^{2\ell}=B^\ell\times B^\ell\simeq B^\ell\times (B^*)^\ell$. Then put $\Lambda_{\varphi_\mathfrak{a}}\defeq\epsilon_Z\Upsilon_{\varphi_\mathfrak{a}}.$

Moreover, recall the \'etale Abel--Jacobi map  
\[ {\Phi_{\text{\'et},Z}^{(F)}}:\CH^{k-1}_0(Z_r)(F)\longrightarrow H^1(F,V_Z(k-1))\longrightarrow H^1\bigl(F,V_W\otimes\boldsymbol{\Sym}^\ell_\text{\'et}(B\otimes B^\ast)(2\ell+1)\bigr) \]
introduced in \eqref{AJZ1} and appearing in the upper horizontal line of \eqref{diagram}, where $F$ is any field containing $K$; still write ${\Phi_{\text{\'et},Z}^{(F)}}$ for the composition of the two maps above. As before, define $\xi_{Z,\mathfrak{a}}\defeq{\Phi_{\text{\'et},Z}^{(H_{c,\mathfrak{N}})}}(\Lambda_{\varphi_{\mathfrak{a}}})$ and
%\begin{equation} \label{BDPZ}
\[ \xi_{Z}\defeq\sum_{[\mathfrak{a}] \in\Pic(\cO_c)} \vartheta(\mathfrak{a})\xi_{Z,\mathfrak{a}}\in 
H^1\Bigl(H_{c,\mathfrak{N}},V_{W}\otimes \boldsymbol{\Sym}_\text{\'et}^\ell(B\otimes B^\ast)(2\ell+1)\otimes_{\Q_p}\mathcal{E}_\vartheta\Bigr). \]
%\end{equation} 
\begin{lemma} \label{lemma6.10}
$\boldsymbol{e^\ast}(\xi_{Z,\mathfrak{a}})=\xi_{X,\mathfrak{a}}$.  
\end{lemma}

\begin{proof} Immediate from diagram \eqref{diagram}. \end{proof}

In order to study the Hecke action on $\xi_Z$, we adapt the arguments in the proof of \cite[Proposition 4.5]{CH}. Let us fix a theta factorization $\chi=\psi\vartheta=(\varphi\varphi_A^\ell)\psi\Norm_K^{-\ell}$ of a Hecke character $\chi$ of infinity type $(1+\ell,-\ell)$ and recall the map \eqref{isogal2}, which we rewrite as 
\begin{equation} \label{isogal5} 
\boldsymbol{e}_{\varphi_A^\ell}:\boldsymbol{e}\cdot\mathbf{Sym}^\ell_\text{\'et}(B\otimes B^*)(\ell)\longrightarrow V_{\varphi_A^\ell}.
\end{equation}
In the next result, we identify finite order Hecke characters $\phi$ with their associated Galois characters, so that we write $\phi(\mathfrak{a})=\phi(\sigma_\mathfrak{a})$. 

\begin{proposition} \label{LemmaCH} 
The equality
\[ \bigl(\mathrm{id}\otimes\boldsymbol{e}_{\varphi_A^\ell}\bigr)\xi_{Z,\cO_c}^{\sigma_\mathfrak{a}}=(\varphi_A^\ell\chi_\cyc^{-\ell})(\sigma_\mathfrak{a})\bigl(\mathrm{id}\otimes \boldsymbol{e}_{\varphi_A^\ell}\bigr)\xi_{Z,\mathfrak{a}} \]
holds for all $\cO_c$-ideals $\mathfrak{a}$.  
\end{proposition}

\begin{proof} To begin with, we gather some preliminaries, following \cite{CH}. By the theory of complex multiplication, $A_\mathfrak{a}=A_c^{\sigma_\mathfrak{a}}$ and the isogeny $\lambda_\mathfrak{a}:A_c\twoheadrightarrow A_\mathfrak{a}=A_c^{\sigma_\mathfrak{a}}$ is characterized by the equality 
\begin{equation} \label{CH0}
\lambda_\mathfrak{a}(x)=\sigma_\mathfrak{a}(x)
\end{equation} 
for all $x\in A[m]$ with $\bigl(m,\Norm_K(\mathfrak{a})\bigr)=1$ (\emph{cf.} \cite[Proposition 15, p. 42]{de87}; see also \cite[p. 591]{CH}). Therefore, for all such points there are equalities 
\[ \lambda_\mathfrak{a}\circ\varphi_c(x)=\sigma_\mathfrak{a}\bigl(\varphi_c(x)\bigr)=\varphi_c^{\sigma_\mathfrak{a}}\bigl(\sigma_\mathfrak{a}(x)\bigr)=\varphi_c^{\sigma_\mathfrak{a}}\bigl(\lambda_{\mathfrak{a}\cO_K}(x)\bigr). \] 
It follows that 
\begin{equation} \label{CH1/2}
\lambda_\mathfrak{a}\circ\varphi_c=\varphi_c^{\sigma_\mathfrak{a}}\circ\lambda_{\mathfrak{a}\cO_K}.
\end{equation} 
Furthermore, observe that the Serre--Tate character $\nu_B$ of $B$ is equal to 
\[ \bigl(\lambda_{\mathfrak{a}\cO_K}^\rho\bigr)_{\rho\in\Gal(H/K)}\in \bigoplus_{\rho\in\Gal(H/K)}\Hom(A^\rho,A^{\rho\sigma_\mathfrak{a}}). \] 
To check this, note that, by \cite[(1.5)]{Ru81}, for all $x\in A^\rho[m]$ and for all integers $m$ prime to $\Norm_K(\mathfrak{a})$ and the conductor of $\chi_A$ the map $A^\rho[m]\rightarrow A^{\rho\sigma_\mathfrak{a}}[m]$ such that $x\mapsto \sigma_\mathfrak{a}(x)$ equals the $\rho$-component of the Serre--Tate character $\nu_B(\mathfrak{a})$, which is just $\lambda_{\mathfrak{a}\mathcal{O}_K}^\rho:A^\rho[m]\rightarrow A^{\rho\sigma_\mathfrak{a}}[m]$, by equality \eqref{CH0}. 

Now we proceed with the proof. We use the symbol $\Gamma_{\varphi_\mathfrak{a}}$ for the image of $\Gamma_{\mathfrak{a}}$ in $W_r\times B$, while we write $\Gamma^\ast_{\varphi_\mathfrak{a}}$ to denote the image of $\Gamma_{\mathfrak a}$ in $W_r\times B^\ast$. As in the proof of \cite[Proposition 4.5]{CH}, we observe that (writing $\mathrm{id}$ for the identity map on $W_r$) there is an equality 
\[ (\mathrm{id}\otimes \lambda_{\mathfrak{a}\cO_K})_*\Gamma_{\varphi_\mathfrak{a}}=\biggl\{\Bigl(\lambda_\mathfrak{a}\bigl(\varphi_c(z)\bigr),{\lambda}_{\mathfrak{a}\cO_K}(z)\!\Bigr)\;\Big|\;z\in A\biggr\} \]
of subsets of $W_r\times B$; thus, thanks to equality \eqref{CH1/2}, we conclude that $(\mathrm{id}\otimes\lambda_{\mathfrak{a}\cO_K})_*\Gamma_{\varphi_\mathfrak{a}}=\Gamma^{\sigma_\mathfrak{a}}_{c}$, where we set $\Gamma_c\defeq\Gamma_{\varphi_{\mathcal{O}_c}}$. Analogously, there is an equality
\[ (\mathrm{id}\otimes \lambda_{\mathfrak{a}\cO_K})_*\Gamma_{\varphi_\mathfrak{a}}^*=\biggl\{\Bigl(\lambda_\mathfrak{a}\bigl(\varphi_c(z)\bigr),{\lambda}_{\mathfrak{a}\cO_K}(z)\!\Bigr)\;\Big|\;z\in A^*\biggr\} \]
of subsets of $W_r\times B^*$; thanks to equality \eqref{CH1/2}, we see that $(\mathrm{id}\otimes \lambda_{\mathfrak{a}\mathcal{O}_K})_*\Gamma_{\varphi_\mathfrak{a}}^*= (\Gamma^*_c)^{\sigma_\mathfrak{a}}$, where again we set $\Gamma^*_c\defeq\Gamma^*_{\varphi_{\cO_c}}$. Therefore, writing $\Lambda_c$ for $\Lambda_{\varphi_{\cO_c}}$, we obtain an equality  
\[ (\mathrm{id}\otimes \lambda_{\mathfrak{a}\mathcal{O}_K})_*\Lambda_{\varphi_\mathfrak{a}}=\Lambda_c^{\sigma_\mathfrak{a}}. \]
Since \'etale Abel--Jacobi maps commute with Galois actions, we have 
\[ {\Phi_{\text{\'et},Z}^{(H_{c,\mathfrak{N}})}}\bigl((\mathrm{id}\otimes \lambda_{\mathfrak{a}\cO_K})_\ast\Lambda_{\varphi_\mathfrak{a}}\bigr)={\Phi_{\text{\'et},Z}^{(H_{c,\mathfrak{N}})}}(\Lambda_c^{\sigma_\mathfrak{a}})={\Phi_{\text{\'et},Z}^{(H_{c,\mathfrak{N}})}}(\Lambda_c)^{\sigma_\mathfrak{a}}=\xi_{Z,\cO_c}^{\sigma_\mathfrak{a}}. \]
On the other hand, we also have 
%\begin{equation}\label{CH3}
\[ {\Phi_{\text{\'et},Z}^{(H_{c,\mathfrak{N}})}}\circ(\mathrm{id}\otimes \lambda_{\mathfrak{a}\cO_K}{)}_*=(\mathrm{id}\otimes \lambda_{\mathfrak{a}\cO_K}{)}_*\circ{\Phi_{\text{\'et},Z}^{(H_{c,\mathfrak{N}})}} \]
%\end{equation}
because of the functoriality of the \'etale Abel--Jacobi map, and therefore we see that 
\begin{equation}\label{CH2}
(\mathrm{id}\otimes \lambda_{\mathfrak{a}\cO_K}{)}_\ast(\xi_{X,\mathfrak{a}})=\xi_{Z,\cO_c}^{\sigma_\mathfrak{a}}.
\end{equation}
Since the map $\boldsymbol{e}_{\varphi_A^\ell}$ in \eqref{isogal5} is $G_K$-equivariant, for all $x\in\boldsymbol{e}\cdot\mathbf{Sym}^\ell_\text{\'et}(B\otimes B^\ast)(\ell)$ and all $\sigma\in G_K$ there is an equality 
\begin{equation} \label{inv}
\boldsymbol{e}_{\varphi_A^\ell}\Bigl(\bigl(\sigma_\mathfrak{a}\otimes\chi_\cyc^\ell(\sigma_\mathfrak{a})\bigr)(x)\Bigr)=\varphi_A^\ell(\sigma)\boldsymbol{e}_{\varphi_A^\ell}(x).
\end{equation}
Therefore, one has
\[ \begin{split}
(\mathrm{id}\otimes\boldsymbol{e}_{\varphi_A^\ell})(\xi_{Z,\mathcal{O}_c}^{\sigma_\mathfrak{a}})&=
\boldsymbol{e}_{\varphi_A^\ell}\bigl((\mathrm{id}\otimes \lambda_{\mathfrak{a}\mathcal{O}_K})_*(\xi_{X,\mathfrak{a}})\bigr) \text{ (by equality \eqref{CH2})} \\
&=\boldsymbol{e}_{\varphi_A^\ell}\bigl(\sigma_\mathfrak{a}(\xi_{Z,\mathfrak{a}})\bigr) \text{ (by equality \eqref{CH0})}\\
&=\chi_\cyc^{-\ell}(\sigma_\mathfrak{a})\boldsymbol{e}_{\varphi_A^\ell}\Bigl(\!(\sigma_\mathfrak{a}\otimes\chi_\cyc^\ell(\sigma_\mathfrak{a}))(\xi_{Z,\mathfrak{a}})\!\Bigr)\\
&=\chi_\cyc^{-\ell}(\sigma_\mathfrak{a})\varphi_A^\ell(\sigma_\mathfrak{a})\bigl(\mathrm{id}\otimes \boldsymbol{e}_{\varphi_A^\ell}\bigr)(\xi_{Z,\mathfrak{a}}) \text{ (by equality \eqref{inv})},
\end{split} \]
and the proof is complete. \end{proof}

Fix a theta factorization $\chi=\psi\vartheta=\varphi\varphi_A^\ell\psi\Norm_K^{-\ell}$. Define 
\begin{equation} \label{xipsivartheta}
\xi_{\psi,\vartheta}\in H^1_f\bigl(H_{c,\mathfrak{N}},V_{\chi^*}(1)\otimes_{\mathcal{E}_\chi}\mathcal{E}_{\psi,\vartheta}\bigr)
\end{equation} 
to be the image of $\boldsymbol{e}^*\xi_Z$ under the composition 
\[ V_W\otimes \boldsymbol{\Sym}_\text{\'et}^\ell(B\otimes B)(2\ell+1)\otimes\mathcal{E}_{\psi,\vartheta}\xrightarrow{\boldsymbol{e}_{\psi^*}\otimes\boldsymbol{e}_{\varphi_A^\ell}} V_{\psi^*\varphi_A^\ell}(1)\simeq V_{\chi^*}(1)\otimes_{\mathcal{E}_\chi}\mathcal{E}_{\psi,\vartheta}, \]
where $\boldsymbol{e}_{\psi^*}:V_W\twoheadrightarrow V_{\psi^*}$ is the projection map. 

\begin{proposition} \label{PropCH}
The equality $\xi_{\psi,\vartheta}^{\sigma_\mathfrak{a}}=\varphi^{-1}(\sigma_\mathfrak{a})\xi_{\psi,\vartheta}$ holds for all $\cO_c$-ideals $\mathfrak{a}$. 
\end{proposition}

\begin{proof} There are equalities 
\[ \begin{split}
\xi_{\psi,\vartheta}^{\sigma_\mathfrak{a}}&=\sum_{\mathfrak{b}\in\Pic(\cO_c)}\vartheta(\mathfrak{b})\xi_{\psi,\vartheta,\mathfrak{b}}^{\sigma_\mathfrak{a}}=\sum_{\mathfrak{b}\in\Pic(\cO_c)}\bigl(\varphi\varphi_A^\ell\bigr)(\sigma_\mathfrak{b})\Norm_K^{-\ell}(\mathfrak{b})\xi_{\psi,\vartheta,\mathfrak{b}}^{\sigma_\mathfrak{a}}\\
&=\sum_{\mathfrak{b}\in\Pic(\cO_c)}\bigl(\varphi\varphi_A^\ell\bigr)(\sigma_\mathfrak{b})
\Norm_K^{-\ell}(\mathfrak{b})\bigl(\varphi_A^\ell\chi_\cyc^{-\ell}\bigr)(\sigma_\mathfrak{a})\xi_{\psi,\vartheta,\mathfrak{ab}}\\
&=\sum_{\mathfrak{ab}\in\Pic(\cO_c)}\bigl(\varphi\varphi_A^\ell\bigr)(\sigma_\mathfrak{a}^{-1})\bigl(\varphi\varphi_A^\ell\bigr)(\sigma_\mathfrak{ab})\Norm_K^\ell(\mathfrak{a})\Norm_K^{-\ell}(\mathfrak{ab})\bigl(\varphi_A^\ell\chi_\cyc^{-\ell}\bigr)(\sigma_\mathfrak{a})\xi_{\psi,\vartheta,\mathfrak{ab}}\\
&=\varphi^{-1}(\sigma_\mathfrak{a})\xi_{\psi,\vartheta},
\end{split} \] 
the third following from a combination of  Lemma \ref{lemma6.10} and Proposition \ref{LemmaCH}. \end{proof}

\begin{corollary} \label{coroz} 
There exists a unique $z_{\psi,\vartheta}\in H^1_f\bigl(K,V_{\chi^*}(1)\otimes_{\mathcal{E}_\chi}\mathcal{E}_{\psi,\vartheta}\bigr)$ whose restriction to $H_{c,\mathfrak{N}}$ is $\xi_{\psi,\vartheta}$. 
\end{corollary}

\begin{proof} The self-duality condition on $\chi$ ensures that $V_{\chi^\ast}(1)\simeq V_{\chi^{-1}}$. To simplify our notation, put $\mu\defeq(\chi\varphi)^{-1}$. Note that $\mathcal{E}_{\mu,\varphi}=\mathcal{E}_{\chi,\varphi}$. Applying Lemma \ref{replem} with $W_\mu=V_{\mu}\otimes_{\mathcal{E}_\mu}\mathcal{E}_{\psi,\vartheta,\varphi}$, we see that the restriction map $H^1\bigl(K,W_{\mu}(\varphi)\bigr)\rightarrow H^1(H_{c,\mathfrak{N}},W_{\mu})^{\varphi^{-1}}$ is an isomorphism. By Proposition \ref{PropCH}, the element $\xi_{\psi,\vartheta}$ in \eqref{xipsivartheta} belongs to 
\[ H^1_f\bigl(H_{c,\mathfrak{N}},V_{\chi^{-1}}\otimes_{\mathcal{E}_\chi}\mathcal{E}_{\psi,\vartheta}\bigr)^{\varphi^{-1}}\subset H^1_f\bigl(H_{c,\mathfrak{N}},V_{\chi^{-1}}\otimes_{\mathcal{E}_\chi}\mathcal{E}_{\psi,\vartheta,\varphi}\bigr)^{\varphi^{-1}}\simeq 
H^1_f(H_{c,\mathfrak{N}},W_\mu)^{\varphi^{-1}}, \]  
where the inclusion on the left is just extension of coefficients. Thus, there exists a unique $z_{\psi,\vartheta}\in H^1_f\bigl(K,W_{\mu}(\varphi)\bigr)$ whose restriction to $H_{c,\mathfrak{N}}$ is $\xi_{\psi,\vartheta}$. Set $\mathcal{G}\defeq\Gal(\mathcal{E}_{\psi,\vartheta,\varphi}/\mathcal{E}_{\psi,\vartheta})$. The projector $e_\mathcal{G}\defeq\frac{1}{|\mathcal{G}|}\cdot\sum_{g\in \mathcal{G}}g$ acting on coefficients fixes the image of $\xi_{\psi,\vartheta}$ in $H^1_f(H_{c,\mathfrak{N}},W_\mu)^{\varphi^{-1}}$, so $z_{\psi,\vartheta}$ is also fixed by $e_\mathcal{G}$ and hence belongs to $H^1_f\bigl(K,V_{\chi^{-1}}\otimes_{\mathcal{E}_\chi}\mathcal{E}_{\psi,\vartheta}\bigr)$, as claimed. \end{proof}

\subsection{A generalized Rubin formula.} \label{sec:GRFII}

Let us fix throughout this subsection a good pair $(\psi,\vartheta)\in S_\chi$ and write $\chi=\psi\vartheta=\psi\varphi\varphi_A^\ell\Norm_K^{-\ell}$. The next result is a direct generalization of a similar result for $k=2$ in \cite{BDP2}. Recall the element $\xi_X$ introduced in \eqref{BDPX}. 

\begin{proposition} \label{prop3.2} 
For any good factorization $\chi=\psi\vartheta$, there is a congruence 
\[ \mathscr{L}_{p,\mathfrak{c}}(\chi^*)\equiv\Omega_p(\chi^*)^{-1}\cdot\Bigl(\vartheta(\mathfrak{a})\log_{X,\mathcal{F}_{\psi,\vartheta}}(\xi_X)\bigl(\omega_{\theta_\psi}\wedge \omega_A^\ell\eta_A^{\ell}\bigr)\!\Bigr)^2\pmod{F_{\psi,\vartheta}^\times}. \] 
\end{proposition}

\begin{proof} In the case of the character $\vartheta^{-1}\Norm_K$, the factor $E_p(\theta_\psi)$ appearing in formula \eqref{BDPTHM} is equal to 
\[ \Bigl(1-\bigl(\vartheta\Norm_K^{-1}\bigr)(\bar{\mathfrak{p}}) a_p(f)+ \bigl(\vartheta\Norm_K^{-1}\bigr)^2(\bar{\mathfrak{p}})\varepsilon_\vartheta(p)p^{k-1}\Bigr)^2, \] 
which belongs to $E_{\vartheta}^\times=E_\psi^\times$. Then, again by \eqref{BDPTHM}, we have 
\begin{equation} \label{eqpf1}
\begin{split}
{L_p(\theta_\psi,\vartheta^{-1}\Norm_K)} &\equiv\left(\sum_{[\mathfrak{a}] \in \Pic(\mathcal{O}_c)} \vartheta(\mathfrak{a}) \log_{X,\mathcal{F}_{\psi,\vartheta} }(\Delta_{\varphi_\mathfrak{a}})\bigl(\omega_{\theta_\psi}\wedge \omega_A^\ell\eta_A^{\ell}\bigr) \right)^2 \\&\equiv\Bigl(\vartheta(\mathfrak{a})\log_{X,\mathcal{F}_{\psi,\vartheta} }(\xi_X )\bigl(\omega_{\theta_\psi}\wedge \omega_A^\ell\eta_A^{\ell}\bigr)\!\Bigr)^2\pmod{E_{\psi,\vartheta}^\times}.
\end{split}
\end{equation}
Now set $\psi^\dagger\defeq(\psi^*)^{-1}$. By congruence \eqref{factorization-modulo-E} and the equality $\chi^\ast=\psi^{-1}\vartheta^{-1}\Norm_K$, we have 
\begin{equation}\label{eqpf2}
\begin{split}
L_p(\theta_\psi,\vartheta^{-1}\Norm_K)&\equiv\mathscr{L}_{p,c \mathfrak{d}_K}(\psi^{-1}\vartheta^{-1}\Norm_K)\cdot\mathscr{L}_{p,c \mathfrak{d}_K \Norm(\mathfrak{f}_\psi)}(\psi^{\dagger}\vartheta^{-1}\Norm_K)\pmod{F_{\psi,\vartheta}^\times}\\&\equiv\mathscr{L}_{p,c \mathfrak{d}_K}(\chi^*)\cdot\mathscr{L}_{p,c \mathfrak{d}_K \Norm(\mathfrak{f}_\psi)}(\psi^{\dagger}\vartheta^{-1}\Norm_K)\pmod{F_{\psi,\vartheta}^\times}.
\end{split}  
\end{equation}
Since $\mathscr{L}_{p, \:c \mathfrak{d}_K \Norm(\mathfrak{f}_\psi)}(\psi^{\dagger}\vartheta^{-1}\Norm_K)\neq0$ by \eqref{GP3}, it follows from \cite[Corollary 3.3]{BDP2} that 
\[ {\mathscr{L}_{p, \:c \mathfrak{d}_K \Norm(\mathfrak{f}_\psi)}(\psi^{\dagger}\vartheta^{-1}\Norm_K})\equiv{\Omega_p(\psi^{-1}\vartheta^\dagger\Norm_K)}\pmod{E_\chi^\times}. \]
Since $\chi$ is self-dual, we have $\psi^{-1}\vartheta^\dagger\Norm_K=\chi^*(\vartheta\vartheta^\dagger)$, so we need to compute $\Omega_p\bigl(\chi^*(\vartheta\vartheta^\dagger)\bigr)$. For this, write $\vartheta=\varphi_\vartheta\Norm_K^{-\ell}$ for a finite order character $\varphi_\vartheta$; then $\vartheta^\dagger=\varphi_\vartheta^\dagger \Norm_K^\ell$ because $\Norm_K^*=\Norm_K$, so we see that  $\vartheta\vartheta^\dagger=\varphi_\vartheta\varphi_\vartheta^\dagger$. Write $\chi=\varphi\chi_A^{-(\ell+1)}(\chi_A^*)^\ell$ for a finite order character $\varphi$; then $\chi^*=\varphi^*(\chi_A^*)^{-(\ell+1)}\chi_A^\ell$, hence 
 $\chi^*(\vartheta\vartheta^\dagger)=\bigl(\varphi^*\varphi_\vartheta\varphi_\vartheta^\dagger\bigr)\chi_A^\ell(\chi_A^*)^{-(\ell+1)}$. 
Thus, from the definition of $p$-adic periods we obtain equalities
\[ \Omega_p(\chi^*)=\mathfrak{g}(\varphi^*)\cdot\Omega_p(A)^{-(2\ell+1)} \]
and
\[ \Omega_p\bigl(\chi^*(\vartheta\vartheta^\dagger)\bigr)=\mathfrak{g}\bigl(\varphi^*\varphi_\vartheta\varphi_\vartheta^\dagger\bigr)\cdot\Omega_p(A)^{-(2\ell+1)}. \] 
From the definition of Gauss sum, there is a congruence 
\[ \mathfrak{g}\bigl(\varphi^*\varphi_\vartheta\varphi_\vartheta^\dagger\bigr)\equiv\mathfrak{g}(\varphi^*)\cdot\mathfrak{g}\bigl(\varphi_\vartheta\varphi_\vartheta^\dagger\bigr)\pmod{(E_{\varphi}E_{\varphi_\vartheta})^\times} \]
(\emph{cf.} \cite[(2-12)]{BDP2}), where we observe that $E_{\varphi_\vartheta\varphi_\vartheta^\dagger}=E_{\varphi_\vartheta}$ because $E_{\varphi_\vartheta}=E_{\varphi_\vartheta^\dagger}$. Recall that the Gauss sum $\mathfrak{g}(\varphi_\vartheta\varphi_\vartheta^\dagger)$ is an element of the composite of the fields $E_{\varphi_\vartheta\varphi_\vartheta^\dagger}=E_{\varphi_\vartheta}$ and $H_{\varphi_\vartheta\varphi_\vartheta^\dagger}$. 
Thanks to the equality $\vartheta=\varphi_\vartheta\Norm_K^{-\ell}$, the extension $H_{\varphi_\vartheta\varphi_\vartheta^\dagger}$ of $K$ cut out by $\varphi_\vartheta\varphi_\vartheta^\dagger$ is the same as the extension $H_\vartheta$ of $K$ cut out by $\vartheta$. On the other hand, thanks to the equality $\varepsilon_\chi=\varepsilon_K$, the field $H_{\varepsilon_\vartheta}$ is equal to the extension $H_{\varepsilon_\psi}$ of $K$ cut out by $\varepsilon_\psi$, which by definition is contained in $F_{\psi,\chi}$. Moreover, the equality $\chi=\psi\varphi_\vartheta\Norm_K^{-\ell}$ ensures that the field $E_\varphi$ is contained in $E_{\psi,\chi}$, which by definition is a subfield of $F_{\psi,\chi}$. Therefore, $\mathfrak{g}\bigl(\varphi_\vartheta\varphi_\vartheta^\dagger\bigr)\in F_{\psi,\chi}^\times$. We conclude that 
\[ \mathfrak{g}\bigl(\varphi^*\varphi_\vartheta\varphi_\vartheta^\dagger\bigr)\equiv\mathfrak{g}(\varphi^*)\pmod{F_{\psi,\chi}^\times}, \]
from which it follows that $\Omega_p(\chi^*)$ and $\Omega_p\bigl(\chi^*\varphi_\vartheta\varphi_\vartheta^\dagger\bigr)$ are congruent modulo $F_{\psi,\chi}^\times$. Finally, combined with congruences \eqref{eqpf1} and \eqref{eqpf2}, this last congruence gives the desired result. \end{proof}

Recall the elements $\xi_{\psi,\vartheta}$ defined in \eqref{xipsivartheta} and  $\omega_{\psi,\vartheta}$ defined in \eqref{omegapsi}.

\begin{corollary} \label{lemRub} 
$\mathscr{L}_{p,\mathfrak{c}}(\chi^*)\equiv\Omega_p(\chi^*)^{-1}\cdot\log_{\chi}(\xi_{\psi,\vartheta})(\omega_{\psi,\vartheta})^2\pmod{F_{\psi,\vartheta}^\times}$.  
\end{corollary}

\begin{proof} Combine equality \eqref{accpair2} and Proposition \ref{prop3.2}. \end{proof}

Now we descend the field of definition to $K$. Recall the differential form $\omega_{\chi}$ introduced in \eqref{omegachi} and the cohomology class $z_{\psi,\vartheta}$ defined in Corollary \eqref{coroz}. Denote by $\bomega_\chi$ also the image of the fixed generator 
$\bomega_\chi$ of $H^1_\dR(\chi)$ in $\D_{\dR}(V_\chi)$; 
with this notation, $\omega_\chi=\alpha_\dR(\bomega_\chi)$. 

\begin{proposition}\label{propGRF}
$\mathscr{L}_{p,\mathfrak{c}}(\chi^*)\equiv\Omega_p(\chi^*)^{-1}\cdot\log_{\chi}(z_{\psi,\vartheta})(\bomega_{\chi})^2\pmod{F_{\psi,\vartheta}^\times}$.
\end{proposition}

\begin{proof} There is a commutative square
\[ \xymatrix@C=35pt@R=35pt{H^1_f\bigl(\Q_p,V_{\chi^*}(1)\otimes_{\mathcal{E}_\chi}\mathcal{E}_{\psi,\vartheta}\bigr)\ar[d]^-\simeq \ar[r] &H^1_f\bigl(\mathcal{F}_{\psi,\vartheta},V_{\chi^*}(1)\otimes_{\mathcal{E}_\chi}\mathcal{E}_{\psi,\vartheta}\bigr)\ar[d]^-\simeq\\
\D_{\dR}(V_{\chi}\otimes_{\mathcal{E}_\chi}\mathcal{E}_{\psi,\vartheta})^\vee\ar[r]^-{\gamma_\dR}& 
\D_{\dR,\mathcal{F}_{\psi,\vartheta}}(V_\chi\otimes_{\mathcal{E}_\chi}\mathcal{E}_{\psi,\vartheta})^\vee} \] 
in which the top horizontal map is restriction and the bottom horizontal map comes from the vertical duality isomorphisms; in particular, we have an equality
\begin{equation} \label{xizeta}
\gamma_\dR\bigl(\log_\chi(z_{\psi,\vartheta})\bigr)=\log_\chi(\xi_{\psi,\vartheta}).
\end{equation}
Observe that the map $\gamma_\dR$ is injective, as it is the restriction of the injection
\[ H^1\bigl(\Q_p,{V}_{\chi^{-1}}\otimes_{\mathcal{E}_\chi}\mathcal{E}_{\psi,\vartheta}\bigr)\longmono H^1\bigl(\Q_p,W_\mu(\varphi)\bigr)\xrightarrow[\simeq]{\text{Lemma \ref{replem}}}H^1(\mathcal{F}_{\psi,\vartheta},W_\mu)^{\varphi^{-1}} \]
(here we are using the same notation as in the proof of Corollary \ref{coroz}). Thus, there is a commutative diagram 
\[ \xymatrix@C=-2pt{
H^1_\dR(\chi)\otimes_K\Q_p\ar[d]^-\simeq\ar@/_6pc/[ddd]^-{\eqref{classcoh}}_-{\rho_\dR}&&H^1_f\bigl(\Q_p,V_{\chi^{-1}}\otimes_{\mathcal{E}_\chi}\mathcal{E}_{\psi,\vartheta}\bigr)\ar[d]^-\simeq&&&&&&\\
\D_\dR(V_\chi)\ar[d]^-\simeq&\times&\D_{\dR}(V_\chi\otimes_{\mathcal{E}_\chi}\mathcal{E}_{\psi,\vartheta})^\vee\ar@{=}[d]&\ar[rrrrrr]&&&&&&\;\;\mathcal{E}_\chi\ar@{=}[d]\\
\D_\dR(V_\chi\otimes_{\mathcal{E}_\chi}\mathcal{E}_{\psi,\vartheta} )^{\mathcal{E}_{\psi,\vartheta}}\ar@{^(->}[d]^-{\alpha_\dR}&\times& \D_{\dR}(V_{\chi}\otimes_{\mathcal{E}_\chi}\mathcal{E}_{\psi,\vartheta})^\vee\ar[d]^-{\gamma_\dR}&\ar[rrrrrr] &&&&&&\;\;\mathcal{E}_\chi\ar[d]
\\
\D_{\dR,\mathcal{F}_{\psi,\vartheta}}(V_\chi\otimes_{\mathcal{E}_\chi}\mathcal{E}_{\psi,\vartheta} )^{\mathcal{E}_{\psi,\vartheta}}&\times&\D_{\dR,\mathcal{F}_{\psi,\vartheta}}(V_{\chi}\otimes_{\mathcal{E}_\chi}{\mathcal{E}_{\psi,\vartheta}})^\vee&\ar[rrrrrr] &&&&&&\;\;\mathcal{F}_{\psi,\vartheta}.} \] 
It follows from the duality identifications that the map $\gamma_\dR$ is a section of the dual $\alpha_\dR^\vee$ of the canonical map $\alpha_\dR:\D_\dR(V_\chi\otimes_{\mathcal{E}_\chi}\mathcal{E}_{\psi,\vartheta} )\rightarrow \D_{\dR,\mathcal{F}_{\psi,\vartheta}}(V_\chi\otimes_{\mathcal{E}_\chi}\mathcal{E}_{\psi,\vartheta})$. Thus, we have 
\[ \begin{split}
\log_{\chi}(\xi_{\psi,\vartheta})(\omega_{\psi,\vartheta})&\overset{\eqref{xizeta}}=\gamma_\dR\bigl(\log_\chi(z_{\psi,\vartheta})\bigr)(\omega_{\psi,\vartheta})\overset{\text{Lemma \ref{gaussum}}}\equiv\gamma_\dR\bigl(\log_\chi(z_{\psi,\vartheta})\bigr)(\omega_{\chi})\pmod{{E}_{\psi,\vartheta}^\times}\\
&\;\;\equiv\gamma_\dR\bigl(\log_\chi(z_{\psi,\vartheta})\bigr)\bigl(\alpha_\dR(\bomega_{\chi})\bigr)\pmod{{E}_{\psi,\vartheta}^\times}\\
&\;\;\equiv(\alpha_\dR^\vee\circ\gamma_\dR)\bigl(\log_\chi(z_{\psi,\vartheta})\bigr)(\bomega_{\chi})\pmod{{E}_{\psi,\vartheta}^\times}\\
&\;\;\equiv\log_\chi(z_{\psi,\vartheta})(\bomega_{\chi})\pmod{{E}_{\psi,\vartheta}^\times},
\end{split} \] 
where for the last congruence we used the injectivity of $\gamma_\dR$. Now the desired result follows from Corollary \ref{lemRub}.
\end{proof}

As the notation suggests, the elements $z_{\psi,\vartheta}$ introduced above depend on the choice of the good pair $(\psi,\vartheta)$. Notice also that $z_{\psi,\vartheta}$ might \emph{a priori} be zero, but if $\mathscr{L}_{p,\mathfrak{c}}(\chi^\ast)\neq 0$, then $z_{\psi,\vartheta}\neq 0$ for \emph{all} good pairs $(\psi,\vartheta)$. Our final goal is to remove this dependence and upgrade the congruence modulo $F_{\psi,\vartheta} ^\times$ to a congruence modulo $\mathcal{E}_\chi^\times$.   

\begin{theorem} \label{mainthm}
Assume $\mathscr{L}_{p,\mathfrak{c}}(\chi^*)\neq 0$. The $\mathcal{E}_\chi$-vector space $H^1_f\bigl(K,V_{\chi^*}(1)\bigr)$ has dimension $1$. Moreover, for any basis element $\bz_\chi$ of $H^1_f\bigl(K,V_{\chi^{-1}}\bigr)$, there is a congruence
\[ \mathscr{L}_{p,\mathfrak{c}}(\chi^*)\equiv {\Omega_p(\chi^*)} ^{-1}\cdot\log_{\chi}(z_\chi )(\omega_{\chi})^2\pmod{\mathcal{E}_{\chi}^\times}. \]
\end{theorem}

\begin{proof} Let $(\psi,\vartheta)$ be a good pair. Since $\mathscr{L}_{p,\mathfrak{c}}(\chi^\ast)\neq 0$, it follows from Proposition \ref{propGRF} that $z_{\psi,\vartheta}\neq0$. Take $f=\theta_\psi$ and $\chi =\varphi^\ast\varphi_A^\ell$ in \cite[Theorem B]{CH}; by property (2) in Definition \ref{goodfact}, the sign of the functional equation for $L(\theta_\psi, \chi,s)$ is $-1$. Since $z_{\psi,\vartheta}$ is the projection of the element $z_{f,\chi}$ in \cite[Theorem B]{CH}, the non-vanishing $z_{\psi,\vartheta}\neq 0$ ensures that $z_{f,\chi}$ is non-zero too. Therefore, we are in position to apply \cite[Theorem B]{CH}, which shows that if $V_{\theta_\psi,\varphi^\ast\varphi_A^\ell}$ is the representation $V_{f,\chi}$ in \cite{CH} (equal to $V_{\theta_\psi}\otimes\bigl(\varphi^*\varphi_A^\ell\bigr)(k/2)$ in our notation), then the $\mathcal{E}_{\psi,\vartheta}$-vector space $H^1_f\bigl(K,V_{\theta_\psi,\varphi^\ast\varphi_A^\ell}\bigr)$ is $1$-dimensional. On the other hand, there is an isomorphism
\[ H^1_f\bigl(K,V_{\theta_\psi,\varphi\varphi_A^\ell}\bigr)\simeq H^1_f\Bigl(K,V_\psi\bigl(\varphi^*\varphi_A^\ell\bigr)\!\Bigr)\bigoplus H^1_f\Bigl(K,V_{\psi^*}\bigl(\varphi^*\varphi_A^\ell\bigr)\!\Bigr), \] 
and, since $z_{\psi,\vartheta}\neq 0$, it follows that the $\mathcal{E}_{\chi}$-vector space $H^1_f\bigl(K,V_{\chi^*}(1)\bigr)$
is $1$-dimensional as well. Thus, we can fix a basis element $\bz_\chi$ of it, which is well defined up to elements in $\mathcal{E}_\chi^\times$. We get a congruence 
\[ \log_{\psi,\vartheta}(z_{\psi,\vartheta})(\bomega_{\chi})\equiv\log_{\chi}(\bz_\chi )(\bomega_{\chi} )\pmod{\mathcal{E}_{\psi,\vartheta}^\times}, \] 
so, by Proposition \ref{propGRF}, for all $(\psi,\vartheta)\in S_\chi$ there is a congruence  
\[ \mathscr{L}_{p,\mathfrak{c}}(\chi^*)\equiv {\Omega_p(\chi^*)}^{-1}\cdot\log_{\chi}(\bz_\chi)(\bomega_{\chi})^2\pmod{\mathcal{F}_{\psi,\vartheta}^\times}. \] 
By Proposition \ref{Goodpairs}, we know that $\bigcap_{(\psi,\vartheta)\in S_\chi}F_{\psi,\vartheta} =E_\chi$, and the same must be true for the corresponding completions. Finally, it follows that 
\[ \mathscr{L}_{p,\mathfrak{c}}(\chi^\ast)\equiv{\Omega_p(\chi^\ast)}^{-1}\cdot\log_{\chi}(\bz_\chi)(\bomega_{\chi})^2\pmod{\mathcal{E}_{\chi}^\times}, \] 
as was to be shown. \end{proof}

The previous result is Theorem A in the introduction.

\begin{remark} \label{finalrem}
The proof of Theorem \ref{mainthm} uses the implication $\mathscr{L}_{p,\mathfrak{c}}(\chi^\ast)\neq 0\Rightarrow 
 z_\chi\neq 0$ of Proposition \ref{propGRF} and the implication $   z_\chi\neq 0\Rightarrow \dim_{\mathcal{E}_\chi}(H^1_f(K,V^\dagger_\chi))=1$ of \cite[Theorem B]{CH}. The Bloch--Kato conjecture predicts that if $L^\prime(\chi,1)\neq0$, then $\dim_{\mathcal{E}_\chi}H^1_f\bigl(K,V^\dagger_\chi\bigr)=1$; one expects to prove that if $L^\prime(\chi,1)\neq0$, then $z_\chi\neq 0$ via an extension of Zhang's higher weight Gross--Zagier formula (\cite{Zhang}) combined with the non-triviality of the $p$-adic Abel--Jacobi map. Summing up, our results are connected to the general expectations and conjectures through the logical diagram 
\[ \xymatrix@R=45pt@C=30pt{\mathscr{L}_{p,\mathfrak{c}}(\chi^*)\neq 0\ar@{=>}[rr]^-{\text{Proposition \ref{propGRF}}}&&
 z_\chi\neq 0
\ar@{=>}[rr]^-{\text{\cite[Theorem B]{CH}}}&&\dim_{\mathcal{E}_\chi}H^1_f\bigl(K,V^\dagger_\chi\bigr)=1\\
&& L^\prime(\chi,1)\neq 0 \ar@{==>}[u]^-{\text{generalization of \cite{Zhang} + AJ non-trivial}}\ar@{==>}[rru]_-{\text{{\qquad\qquad\; Bloch--Kato conjecture}}} &&} \]
in which the dotted implications are (as of today) only conjectural. 
\end{remark}

\appendix 

\section{Some auxiliary results} \label{app}

In this appendix, we collect some auxiliary results on $p$-adic Galois representations that are used in the main body of the article.

\subsection{A result on Galois representations} 

Let $V$ be a finite-dimensional $\Q_p$-vector space equipped with a continuous right action of $G_F\defeq\Gal(\bar{F}/F)$ for a field $F$ of characteristic $0$. Let $L/F$ be a finite abelian extension, set $G\defeq\Gal(L/F)$ and let $\varphi:G\rightarrow E_\varphi^\times$ be a character. The group $G$ acts on the right on $H^1(L,V)$ by the formula $(c^g)(\sigma)\defeq\bigl(c(g\sigma g^{-1})\bigr)^{g^{-1}}$. Assume that $V$ has an $E_\varphi$-vector space structure that is compatible with the $\Q_p$-linear structure on $V$; in this case, denote by $H^1(L,V)^{\varphi^{-1}}$ the $G_F$-submodule on which $G$ acts as multiplication by $\varphi^{-1}$. We write $V(\varphi)$ for the $G_F$-module $V$ on which the Galois action is twisted by $\varphi$; more precisely, writing $\ast_\varphi$ for the right action of $G_F$ on $V(\varphi)$, we set $v{\ast_\varphi\sigma}\defeq\varphi(\bar\sigma)v^\sigma$ for all $v\in V(\varphi)$ and $\sigma\in G_F$, where $\sigma\mapsto\bar\sigma$ is the quotient map $G_F\rightarrow G$.   
In particular, the $G_F$-invariant subspace $V(\varphi)^{G_F}$ of $V(\varphi)$ coincides with 
%\begin{equation}\label{Vvarphi}
\[ V^{\varphi^{-1}}\defeq\bigl\{v\in V\mid \text{$v^\sigma=\varphi^{-1}(\bar\sigma)v$ for all $\sigma\in G_F$}\bigr\}. \]
%\end{equation}  
We single out the following standard result. 

\begin{lemma} \label{replem} 
Assume that $V^{G_L}=0$. The restriction map $H^1\bigl(F,V(\varphi)\bigr)\rightarrow H^1(L,V)^{\varphi^{-1}}$ is an isomorphism.
\end{lemma}

\begin{proof} There is an isomorphism $V(\varphi)\simeq V$ of $G_L$-modules; moreover, the conjugation map $\sigma\mapsto g\sigma g^{-1}$ induces an automorphism of $G_L$. Thus, the action of $G$ on $H^1(L,V)$ is via the character $\varphi$ and there is an identification 
\[ H^1\bigl(L,V(\varphi)\big)^G=H^1(L,V)^{\varphi^{-1}}. \] 
The result then follows from the inflation-restriction exact sequence because, by assumption, $V(\varphi) ^{G_L}= V^{G_L}=0$. \end{proof}

\subsection{Some results on de Rham modules of $p$-adic representations} \label{appendixA2} 

Let $K\subset E\subset F$ be inclusions of fields, with $F$ finite over $K$, and let $M$ be a free $F\otimes_{K}E$-module of rank $1$. There is a canonical map $\mu:F\otimes_{ K}E\rightarrow F$ given on pure tensors by $x\otimes y\mapsto xy$, which is clearly not injective; the kernel of $\mu$ is generated by $a\otimes 1-1\otimes a$ for $a\in E$. For $x\in F$ and $m\in M$, let $xm$ be the action of $x$ on $m$ via the map $F\hookrightarrow F\otimes_{ K}E$ sending $x\in F$ to $x\otimes 1\in F\otimes _{ K}E$; similarly, for $x\in E$ and $m\in M$ denote by $[x]m$ the action of $x$ on $m$ via the map $E\hookrightarrow F\otimes _{ K}E$ sending $y\in E$ to $1\otimes y\in F\otimes _{ K}E$. Put 
\[ M^E\defeq\bigl\{m\in M\mid \text{$xm=[x]m$ for all $x\in E$}\bigr\}. \]
%More explicitly,
%\[M^E=\{m\in M| (x\otimes 1)m=(1\otimes x)m,\forall x\in E\}.\] 
Since it is an $F\otimes_{K}E$-module, $M$ is, in particular, an $F$-vector space via the map $x\mapsto x\otimes 1$ considered before; since $M$ is finitely generated over $F\otimes_{ K}E$ and $E$ is finite-dimensional over $ K$, it follows that $M$ is a finite-dimensional $F$-vector space, say of dimension $d$ (since $M$ is free of rank $1$ over $F\otimes_{ K}E$, we have $d=[E: K]$). Now we show that $M^E$ is a subspace of the $F$-vector space $M$; to do this, we need to check that for $x\in F$ and $m\in M^E$ we have $xm\in M^E$. Observe that $xm\in M^E$ if and only if $y(xm)=[y](xm)$ for all $y\in E$; this last condition is easy to check, as there are equalities
\[ (y\otimes1)(x\otimes1)m=(x\otimes1)(y\otimes1)m=(x\otimes1)(1\otimes y)m=(1\otimes y)(x\otimes 1)m. \]
Therefore, $M^E$ is a subspace of the $F$-vector space $M$ and, since $M$ is finite-dimensional over $F$, we conclude that $M^E$ is finite-dimensional over $F$ as well. 

\begin{lemma} \label{lemmaA2}
The $F$-vector space $M^E$ is $1$-dimensional. 
\end{lemma}

\begin{proof} The proof is by induction on $d\defeq\dim_K M$. Fix a generator $m_0$ of $M$ as an $F\otimes_KE$-module. 
If $d=1$, then $E=K$; thus, $F\otimes_KE=F$ and $M^E=M$ by definition, so there is nothing to prove. Now assume $d\geq2$ and that the result is true for all $M^\prime$ as above with $\dim_K M^\prime\leq d-1$.
%; so we assume that for all free of rank one $F\otimes_KE'$-modules $M'$, the $F$-vector space $(M')^{E'}$ has dimension $1$. 
Pick $y\in E\smallsetminus K$; then $\lambda\defeq y\otimes 1-1\otimes y\in(F\otimes_KE)\smallsetminus\{0\}$ and, since $M$ is $F\otimes _KE$-free and $m_0$ is a generator of $M$ over $F\otimes _KE$, we have $\lambda m_0\neq 0$. Let $K_\lambda\defeq K\cdot(\lambda m_0)$ be the 1-dimensional $K$-vector space generated by $\lambda m_0$. With an abuse of notation, let $\lambda:M\rightarrow M$ denote the multiplication-by-$\lambda$ map, which is $K$-linear, and set $M_\lambda\defeq\ker(\lambda)$. Since the image of $\lambda$ contains the non-zero $K$-vector space $K_\lambda$, we see that $M_\lambda\neq M$, so $\dim_KM_\lambda\leq d-1$. Now notice that $M_\lambda$ is an $F\otimes _KE$-module, as for every $x=\sum_ix_i\otimes y_i\in F\otimes _KE$ we have 
\[ \begin{split}
   \lambda x=\lambda\cdot\biggl(\sum_ix_i\otimes y_i\biggr)&=\sum_i(y\otimes 1-1\otimes y)\cdot(x_i\otimes y_i)\\&=\sum_i(x_i\otimes y_i)\cdot(y\otimes 1-1\otimes y)=x\lambda,
   \end{split} \] 
%
%\[\begin{split}
%\lambda x&=\lambda\left(\sum_ix_i\otimes y_i\right)\\
%&=\sum_i(y\otimes 1-1\otimes y)(x_i\otimes y_i)\\
%&=\sum_i(yx_i\otimes y_i-x_i\otimes yy_i)\\
%%&=\sum_i(x_i y\otimes y_i-x_i\otimes y_iy)\\
%&=\sum_i(x_i\otimes y_i)(y\otimes 1-1\otimes y)\\
%&=x\lambda 
%\end{split}
%\] 
which implies that if $\lambda m=0$ (\emph{i.e.}, $m\in M_\lambda$), then $\lambda (xm)=x(\lambda m)=0$, \emph{i.e.}, $xm\in M_\lambda$. We may therefore apply the inductive hypothesis to $M_\lambda$ and conclude that the $F$-vector space $M_\lambda^E$ is $1$-dimensional. Finally, we check that $M_\lambda^E=M^E$. Since $M_\lambda\subset M$, it is obvious that $M_\lambda^E\subset M^E$; on the other hand, if $m\in M^E$, then $\lambda m=0$, so $m\in M_\lambda$ and, finally, $m\in M_\lambda^E$ because $m\in M^E$. \end{proof}

We apply Lemma \ref{lemmaA2} to the following situation. Let $K$ and $E$ be finite field extensions of $\Q_p$ and let $V$ be a $1$-dimensional $E$-vector space equipped with a continuous action of $G_K$. Assume that $V$ is a de Rham representation, \emph{i.e}, the filtered $K$-vector space 
\[ \mathbf{D}_{\dR,K}(V)\defeq (V\otimes_{\Q_p}\mathbf{B}_\dR)^{G_K} \] 
has dimension $d=\dim_{\Q_p}(V)=[E:\Q_p]$, where $\mathbf{B}_\dR$ is Fontaine's de Rham period ring. As $\Q_p$-vector spaces, there are isomorphisms $\mathbf{D}_{\dR,K}(V)\simeq\Q_p^d\simeq E$. Now assume that $F$ is a finite extension of $K$ containing $E$. Then there is an isomorphism of filtered $F$-vector spaces 
\[ \mathbf{D}_{\dR,F}(V)\simeq\mathbf{D}_{\dR,K}(V)\otimes_KF \] 
and, as $\Q_p$ vector spaces, we have $\mathbf{D}_{\dR,K}(V)\otimes_KF\simeq E\otimes_KF$. Since $E\subset F$, $\mathbf{D}_{\dR,F}(V)$ is equipped with an $E$-vector space structure that agrees, when restricted to $\Q_p$, with the $\Q_p$-linear structure on $\mathbf{D}_{\dR,F}(V)$ coming from the inclusion $\Q_p\subset F$. Therefore, $\mathbf{D}_{\dR,F}(V)$ is equipped with an $E\otimes_{\Q_p}F$-module structure. From now on, take $K=\Q_p$ and for two $\Q_p$-vector spaces $W_1$ and $W_2$ set $W_1\otimes W_2\defeq W_1\otimes_{\Q_p}W_2$. 

\begin{lemma} \label{lemmaA3}
The $F\otimes E$-module $\mathbf{D}_{\dR,F}(V)$ is free of rank $1$. 
\end{lemma} 

\begin{proof} We saw above that there are isomorphisms $\mathbf{D}_{\dR,F}(V)\simeq \mathbf{D}_{\dR,K}(V)\otimes F\simeq E\otimes F$ of $\Q_p$-vector spaces and that $\mathbf{D}_{\dR,F}(V)$ has a canonical $F\otimes E$-module structure, so the result follows. \end{proof}

\begin{lemma} \label{lemmaA4}
The $F$-vector space $\mathbf{D}_{\dR,F}(V)^{E}$ is $1$-dimensional. 
\end{lemma}

\begin{proof} On the $F\otimes E$-module $\mathbf{D}_{\dR,F}(V)$ there are two different actions of $E$, namely, the one arising from the $F$-linear structure on $\mathbf{D}_{\dR,F}(V)$ given by the inclusion $F\hookrightarrow F\otimes E$ such that $x\mapsto x\otimes 1$ and the action of $E$ induced by the inclusion $E\hookrightarrow F\otimes E$ defined by $y\mapsto 1\otimes y$. Therefore, we are in a position to apply Lemma \ref{lemmaA2}, which proves the result. \end{proof}

Thanks to the lemma above, we can fix non-canonical isomorphisms of $1$-dimensional $F$-vector spaces 
%\begin{equation}\label{eqapp8} 
\[ \mathbf{D}_{\dR}(V\otimes_EF)\simeq \mathbf{D}_{\dR}(V)\otimes_EF\simeq \mathbf{D}_{\dR,F}(V)^{E}, \]
%\end{equation}
where the first isomorphism follows because the Rham module of the trivial representation $G_{\Q_p}\rightarrow F^\times$ is just $F$. 

Now let $E^\prime/E$ be a finite field extension with $F\supset E^\prime$. There is a canonical injective map 
\[ \mathbf{D}_{\dR,F}(V)\longmono \mathbf{D}_{\dR,F}(V\otimes_EE^\prime) \] 
of $F$-vector spaces that is induced by the canonical map $V\rightarrow V\otimes_EE^\prime$. Since, by Lemma \ref{lemmaA4}, all spaces are $1$-dimensional over $F$, we obtain an isomorphism 
%\begin{equation} \label{eqapp9}
\[ \mathbf{D}_{\dR,F}(V)^E\overset\simeq\longrightarrow \mathbf{D}_{\dR,F}(V\otimes_EE^\prime)^{E^\prime} \]
%\end{equation}
of $1$-dimensional $F$-vector spaces.

\bibliographystyle{amsplain}
\bibliography{references}

\end{document}